\documentclass[11pt]{amsart}

\usepackage{amssymb,amsmath,amsthm}
\numberwithin{equation}{section}
\usepackage{xparse} 
\usepackage{aligned-overset}
\usepackage{geometry}
\usepackage{comment}
\usepackage{enumitem}
\usepackage{fancyhdr}
\usepackage[usenames,dvipsnames,table]{xcolor}
\usepackage[unicode,
bookmarksopen=true,
bookmarksopenlevel=1,
colorlinks=true,
linkcolor=darkblue,
linktoc=page,
citecolor=darkblue,
]{hyperref}
\usepackage{tikz}
\usepackage{float}

\definecolor{darkblue}{rgb}{0,0,0.5}

\newenvironment{subproof}[1][\proofname]{%
  \begin{proof}[#1]%
}{%
  \end{proof}%
}

\def\COMMENT#1{}

\def\O{\mathcal{O}}

\theoremstyle{plain}
\newtheorem{theorem}{Theorem}[section]
\newtheorem{lemma}[theorem]{Lemma}
\newtheorem{corollary}[theorem]{Corollary}
\newtheorem{conjecture}[theorem]{Conjecture}
\newtheorem{proposition}[theorem]{Proposition}

\newtheorem{claim}{Claim}

\theoremstyle{definition}
\newtheorem{defn}[theorem]{Definition}

\theoremstyle{remark}

\title[]{New bounds on the size of Nearly Perfect Matchings \\ in almost regular hypergraphs}
\date{\today}

\author[Dong~Yeap~Kang]{Dong~Yeap~Kang}
\address{Dong Yeap Kang, School of Mathematics,
	University of Birmingham,
	Edgbaston,
	Birmingham,
	United Kingdom.}
\email{\texttt{D.Y.Kang.1@bham.ac.uk}} 

\author[Daniela~K\"uhn]{Daniela~K\"uhn}
\address{Daniela K\"uhn, School of Mathematics,
	University of Birmingham,
	Edgbaston,
	Birmingham,
	United Kingdom.}
\email{\texttt{d.kuhn@bham.ac.uk}}

\author[Abhishek~Methuku]{Abhishek~Methuku}
\address{Abhishek~Methuku, Department of Mathematics, ETH Z\"urich, Z\"urich, Switzerland}
\email{\texttt{abhishekmethuku@gmail.com}}

\author[Deryk~Osthus]{Deryk~Osthus}
\address{Deryk Osthus, School of Mathematics,
	University of Birmingham,
	Edgbaston,
	Birmingham,
	United Kingdom.}
\email{\texttt{d.osthus@bham.ac.uk}}

\thanks{This project has received partial funding from the European Research
Council (ERC) under the European Union's Horizon 2020 research and innovation programme (grant agreement no. 786198, D.~Kang, D.~K\"uhn and D.~Osthus).
The research leading to these results was also partially supported by the EPSRC, grant nos. EP/N019504/1 (D.~K\"uhn) and EP/S00100X/1 (A.~Methuku and D.~Osthus).}

\begin{document}
	\begin{abstract} \noindent
		Let $H$ be a $k$-uniform $D$-regular simple hypergraph on $N$ vertices. Based on an analysis of the R\"odl nibble, Alon, Kim and Spencer (1997) proved that if $k \ge 3$, then $H$ contains a matching covering all but at most $ND^{-1/(k-1)+o(1)}$ vertices, and asked whether this bound is tight. In this paper we improve their bound by showing that for all $k > 3$, $H$ contains a matching covering all but at most $ND^{-1/(k-1)-\eta}$ vertices for some $\eta = \Theta(k^{-3}) > 0$, when $N$ and $D$ are sufficiently large. Our approach consists of showing that the R\"odl nibble process not only constructs a large matching but it also produces many well-distributed `augmenting stars' which can then be used to significantly improve the matching constructed by the R\"odl nibble process.

		Based on this, we also improve the results of Kostochka and R\"odl (1998) and Vu (2000) on the size of matchings in almost regular hypergraphs with small codegree. As a consequence, we improve the best known bounds on the size of large matchings in combinatorial designs with general parameters. Finally, we improve the bounds of Molloy and Reed (2000) on the chromatic index of hypergraphs with small codegree (which can be applied to improve the best known bounds on the chromatic index of Steiner triple systems and more general designs).

	\end{abstract}
	\maketitle
	
	\section{Introduction}
	
	\subsection{Background} The problem of finding nearly perfect matchings in regular hypergraphs has a long history in discrete mathematics and such results have applications to other areas e.g., \cite{ford2018long}. One such application is R\"odl's resolution of a famous conjecture of Erd\H{o}s and Hanani \cite{ErdosHanani} on the existence of partial Steiner systems. Here a \emph{partial Steiner system} with parameters $(t, k, n)$ is a $k$-uniform hypergraph on $n$ vertices such that every set of $t$ vertices is contained in at most one edge. It is easy to see that any such system has at most $\binom{n}{t}/\binom{k}{t}$ edges. 
	Erd\H{o}s and Hanani \cite{ErdosHanani} conjectured that for every fixed $t < k$ and for every sufficiently large $n$, there exists a partial Steiner system with parameters $(t, k, n)$ having $(1-o(1))\binom{n}{t}/\binom{k}{t}$ edges. R\"odl \cite{rodl1985packing} confirmed this conjecture by introducing a by now celebrated technique called the `R\"odl nibble' which is a versatile approach for finding large matchings in a semi-random manner. In fact, in his paper \cite{rodl1985packing} R\"odl used this technique to prove the existence of a nearly perfect matching in a specific hypergraph. Frankl and R\"odl \cite{FranklRodl} extended this result to matchings in $D$-regular $n$-vertex hypergraphs with codegree at most $D/(\log n)^4$. (Here a hypergraph $H$ is $D$-regular if every vertex of $H$ lies in precisely $D$ hyperedges. The \emph{codegree} of $H$ is the maximum over the codegrees of all pairs of distinct vertices of $H$, where the codegree of a pair $\{x, y\}$ of distinct vertices is defined as the number of edges of $H$ containing both $x$ and $y$.)
	
	One of the early extensions of R\"odl's work is a result of Pippenger showing that if $H$ is a $k$-uniform $D$-regular hypergraph on $N$ vertices with codegree $o(D)$, then there is a matching in $H$ covering all but at most $o(N)$ vertices.

	However, his proof does not supply an explicit estimate for the error term $o(N)$. Sharpening this error term, which is useful for many applications, is a challenging problem which attracted the attention of several researchers \cite{grable1996nearly,grable1999more,AKS1997,kuzjurin1997almost,kuzjurin1995difference,KR1998}. Far reaching generalizations and further extensions were given e.g., in \cite{ford2018long, PippengerSpencer, kahn1991recent, gordon1996asymptotically, kahn1996asymptotically,kahn1996linear,spencer1995asymptotic,rodl1996asymptotic, EGJ2020}.

	In 1997, Alon, Kim and Spencer \cite{AKS1997} proved the following theorem, which improved all previous bounds for simple hypergraphs. Here a hypergraph is called \emph{simple} (or \emph{linear}) if its codegree is at most $1$.

	\begin{theorem}[Alon, Kim and Spencer~\cite{AKS1997}]\label{thm:aks}
		Let $H$ be a simple $k$-uniform $D$-regular hypergraph on $N$ vertices. If $k > 3$, there is a matching in $H$ covering all but at most $\O(ND^{-\frac{1}{k-1}})$ vertices. If $k = 3$, there is a matching in $H$ covering all but at most $\O(ND^{-1/2} \ln^{3/2}D)$ vertices. 
	\end{theorem}
	
	Kostochka and R\"odl \cite{KR1998} extended Theorem ~\ref{thm:aks} to hypergraphs with small codegrees. Their proof consists a reduction to the case when $H$ is simple, i.e., to Theorem~\ref{thm:aks}.
	
	\begin{theorem}[Kostochka and R\"odl \cite{KR1998}]\label{thm:kr}
		Let $k \ge 3$ and let $0 < \delta, \gamma < 1$. Then there exists $D_0$ such that the following holds for $D \ge D_0$.
		
		Let $H$ be a $k$-uniform $D$-regular hypergraph on $N$ vertices with codegree at most $C < D^{1- \gamma}$. Then there is a matching in $H$ covering all but at most $\O\left(N(\frac{D}{C})^{-\frac{1-\delta}{k-1}}\right)$ vertices.
	\end{theorem}
	
	Vu \cite{vu2000} improved the above result by removing $\delta$ and the assumption $C \le D^{1-\gamma}$ at the cost of a logarithmic term in the bound on the number of uncovered vertices. 
	
	\begin{theorem}[Vu \cite{vu2000}]
		\label{thm:Vu2000}
		Let $k \ge 3$. Then there exist $D_0$ and a constant $c > 0$ such that the following holds for $D \ge D_0$.
		
		Let $H$ be a $k$-uniform $D$-regular hypergraph on $N$ vertices with codegree at most $C$. Then there is a matching in $H$ covering all but at most $\O(N(\frac{D}{C})^{-\frac{1}{k-1}}\log^cD)$ vertices.
	\end{theorem}
	
	Vu \cite{vu2000} also proved a general result using additional information about `higher codegrees'. In particular, he defined the $j$-codegree as the maximum number of edges containing a set of $j$ vertices, and proved a general upper bound on the number of uncovered vertices as a function of the $j$-codegrees.
	
	\subsection{New results} Alon, Kim, and Spencer~\cite{AKS1997} asked whether Theorem~\ref{thm:aks} is tight. Making progress on this twenty three year old question, we improve Theorem~\ref{thm:aks} for all $k > 3$, when $N$ and $D$ are sufficiently large. Our proof involves a method that goes beyond the classical R\"odl nibble technique. 
	
	Our main result, stated below, also applies to hypergraphs with small codegree and thus improves the bounds in Theorems ~\ref{thm:kr} and ~\ref{thm:Vu2000} as well, for all $k > 3$.

	\begin{theorem}\label{thm:main3intro}
		Let $k > 3$ and let $0 < \gamma,\mu < 1$ and 
		$0 < \eta < \frac{k-3}{(k-1)(k^3 - 2k^2 - k + 4)}$. Then there exists $N_0 = N_0(k, \gamma , \eta ,\mu)$ such that the following holds for $N \geq N_0$ and $D \geq \exp(\log^\mu N)$.
		
		Let $H$ be a $k$-uniform $D$-regular hypergraph on $N$ vertices with codegree at most $C \leq D^{1-\gamma}$. Then $H$ contains a matching covering all but at most $N (\frac{D}{C})^{-\frac{1}{k-1} - \eta}$ vertices.
		
		In particular, if $H$ is simple the number of uncovered vertices is at most $N D^{-\frac{1}{k-1} - \eta}$.
	\end{theorem}
	
Theorem ~\ref{thm:main3intro} (as well as Theorems ~\ref{thm:aks}--\ref{thm:Vu2000}) also hold for hypergraphs which are close to being regular (rather than regular). We provide the precise statement in Theorem ~\ref{thm:main3proof}.\COMMENT{Vu mentioned that his results allowed the largest degree fluctuation among the results known at the time. His results work when degrees are within $D(1 \pm \delta)$ where $\delta = \Theta(C/D)^{\frac{1}{k-1}} \log^c D$. Our result works when $\delta = \Theta(C/D)^{\frac{1}{k-1}+\eta}$, so it is almost as good as his.} Moreover, Theorem~\ref{thm:main3proof} also guarantees the following `quasirandomness' property of the matching $M$ returned by Theorem~\ref{thm:main3proof}: Suppose we are given a reasonably small set $\mathcal V$ of subsets of $V(H)$. Then for every $S \in \mathcal V$, the proportion of vertices of $S$ not covered by $M$ is small. Such a result was first derived by Alon and Yuster~\cite{AY2005}. We provide better bounds than in ~\cite{AY2005} under slightly stronger assumptions. (This quasirandomness property will be important when we deduce Corollary~\ref{cor:chromindex} from Theorem~\ref{thm:main3proof}.)

\subsection{Corollaries} In this subsection we present several applications of our results. In particular, we give new bounds on the chromatic index of hypergraphs with small codegree, and improve well-known results on the size of matchings and the chromatic index of combinatorial designs.
	
	\subsubsection{Matchings in Steiner systems}
	
	A well-known example of a regular hypergraph with small codegrees is a (full) \emph{Steiner system} $S(t,k,n)$, which is a $k$-uniform hypergraph $S$ on $n$ vertices such that every set of $t$ vertices belongs to exactly one edge of $S$.  Solving a problem going back to the nineteenth century, Keevash \cite{keevash2014existence} proved the existence of (full) Steiner systems via randomised algebraic constructions. A combinatorial proof was given in \cite{GKLO2020}. Since $S(t,k,n)$ is $\binom{n-1}{t-1}/\binom{k-1}{t-1}$-regular and has codegree $\binom{n-2}{t-2}/\binom{k-2}{t-2}$, we can apply Theorem ~\ref{thm:main3intro} to show the existence of large matchings in (full) Steiner systems as follows.

	\begin{corollary}
		\label{cor:largematchingSteiner}
		Let $k > t > 1$ and let $0 < \eta < \frac{k-3}{(k-1)(k^3 - 2k^2 - k+4)}$. Then there exists $n_0 = n_0(k, t, \eta)$ such that every (full) Steiner system $S(t,k,n)$ with $n \geq n_0$ has a matching covering all but at most $n^{\frac{k-2}{k-1} - \eta}$ vertices.
	\end{corollary}
	
	Corollary~\ref{cor:largematchingSteiner} improves the best known estimates for $S(2, k,  n)$ with $k > 3$ given in \cite{AKS1997}. Better bounds are known for $S(2,3,n)$, which is usually referred to as a \emph{Steiner triple system} of order $n$.  The problem of the existence of large matchings in Steiner triple systems has a long history. About forty years ago, Brouwer \cite{brouwer1981size} conjectured that every Steiner triple system of order $n$ contains a matching of size $(n -4)/3$. 
	Very recently, Keevash, Pokrovskiy, Sudakov and Yepremyan~\cite{KPSY2020} showed that any Steiner triple system of order $n$ has a matching covering all but at most $\O(\log n / \log \log n)$ vertices. This improved a sequence of previous bounds in \cite{wang1978self,lindner1978note,brouwer1981size,AKS1997}. Their method utilises the R\"odl nibble as well as robust expansion properties of edge-coloured pseudorandom graphs to find augmenting edge-coloured paths. It even applies in a more general setting: they show that $3$-uniform $n$-vertex hypergraphs satisfying certain pseudorandomness properties have a matching covering all but at most $\O(\log n / \log \log n)$ vertices.
	
	\subsubsection{Chromatic index of hypergraphs with small codegree}
	
	The \emph{chromatic index} $\chi'(H)$ of a hypergraph $H$ is the minimum number of colours needed to colour its edges so that no two edges which intersect receive the same colour. A classical theorem of Vizing states that every graph with maximum degree at most $D$ has chromatic index at most $D+1$. The corresponding problem for $k$-uniform hypergraphs with $k \geq 3$ is still open. 
	
	In 1989, Pippenger and Spencer~\cite{PippengerSpencer} proved that every $k$-uniform hypergraph with maximum degree at most $D$ and codegree $o(D)$ has the chromatic index $D+o(D)$. Sharpening this $o(D)$-term, and improving a result of Kahn \cite{kahn1996asymptotically}, the best known asymptotic bound on the chromatic index of a $k$-uniform hypergraph was shown by Molloy and Reed~\cite{MR2000}, who proved that every $k$-uniform hypergraph with maximum degree at most $D$ and codegree at most $C$ has chromatic index at most $D + \O(D(D/C)^{-1/k} (\log D / C)^4)$. We improve this result as follows.
	
	\begin{corollary}\label{cor:chromindex}
	Let $k \geq 3$, and let $D, N, C > 0$, $0 < \gamma,\mu < 1$ and $0 < \eta < \frac{k-2}{k(k^3 + k^2 - 2k + 2)}$. Then there exists $N_0 = N_0(k , \gamma , \eta, \mu)$ 
	such that the following holds for all $N \geq N_0$ and $D \geq \exp(\log^\mu N)$.
		
	Let $H$ be a $k$-uniform multi-hypergraph on $N$ vertices with maximum degree at most $D$ and codegree at most $C \leq D^{1-\gamma}$. Then the chromatic index of $H$ is at most $D + D (D/C)^{-1/k - \eta}$.
	\end{corollary}

	 The above results have applications to the chromatic index of Steiner triple systems (and designs with more general parameters). Indeed, since any colour class of an edge-colouring of a Steiner triple system $H$ on $n$ vertices contains at most $\lfloor n/3 \rfloor$ edges, 
	 \begin{equation}
	 \label{eq:g(n)}
	 \chi'(H) \ge g(n) :=
	\begin{cases}
	(n+1)/2& \text{if }n\equiv 1 ({\rm mod}\:6), \\
	(n-1)/2& \text{if }n\equiv 3 ({\rm mod}\:6).
	\end{cases}
	\end{equation}

	There are several constructions of Steiner triple systems on $n$ vertices having their chromatic index between $g(n)$ and $g(n)+2$~\cite{BCHW2017, meszka2013, RW1971, VSSRRCCC1993}; it was conjectured by Meszka, Nedela and Rosa~\cite{MNR2006} that every Steiner triple system $H$ on $n > 7$ vertices satisfies $\chi'(H) \le g(n)+2$. 
	
	On the other hand, the (vertex) degree of a Steiner triple system $H$ on $n$ vertices is $(n-1)/2$. Thus an application of the Pippenger--Spencer theorem ~\cite{PippengerSpencer} discussed above implies that  $\chi'(H) \le n/2 + o(n)$. This result resolved a longstanding open problem and is asymptotically best possible by \eqref{eq:g(n)}. The result of Molloy and Reed in ~\cite{MR2000} implies that the error term $o(n)$ can be improved to $\O(n^{2/3} \log^4 n)$ and Corollary~\ref{cor:chromindex} implies that $\chi'(H) \le n/2+\O(n^{2/3-\eta})$. More generally, one can apply our results to designs with arbitrary parameters to obtain similar results and improvements.

        We now discuss applications of Corollary~\ref{cor:chromindex} to the chromatic index of Latin squares. A \emph{Latin square} $L$ is an $n \times n$ array with $n$ symbols such that each symbol appears exactly once in each row and in each column. For each $n \times n$ Latin square $L$, we define a corresponding 3-uniform $n$-regular simple hypergraph $H_L$ on $3n$ vertices as follows. 
    Let $V(H) = R \cup C \cup S$ where $R$, $C$, and $S$ are the set of rows, columns, and symbols of $L$ respectively. 
    Let $E(H)$ be the collection of triples $\{i,j,c \}$ where $i \in R$, $j \in C$, and $c \in S$ is the symbol in the $(i,j)$-th entry of $L$.

    A famous conjecture by Ryser, Brualdi, and Stein~\cite{BR1991, ryser1967, stein1975} states that $H_L$ has a matching of size $n-1$, and if $n$ is odd, then $H_L$ has a perfect matching. The current best known bound for this conjecture is $n - O(\log n / \log \log n)$ by Keevash, Pokrovskiy, Sudakov and Yepremyan~\cite{KPSY2020}.
    
    Cavenagh and Kuhl~\cite{CK2015} and Besharati, Goddyn, Mahmoodian, and Mortezaeefar~\cite{BGMM2016} posed the following conjecture regarding the chromatic index of Latin squares.
    
    \begin{conjecture}
        Let $L$ be an $n \times n$ Latin square, and let $H_L$ be the corresponding 3-uniform  hypergraph on $3n$ vertices. Then $\chi'(H_L) \leq n+2$. Moreover, if $n$ is odd, then $\chi'(H_L) \leq n+1$.
    \end{conjecture}

    Since $H_L$ is a simple hypergraph and $n$-regular, an application of the Pippenger--Spencer theorem ~\cite{PippengerSpencer} implies that  $\chi'(H) \le n + o(n)$. A result of Molloy and Reed ~\cite{MR2000} implies that the error term $o(n)$ can be improved to $\O(n^{2/3} \log^4 n)$ and Corollary~\ref{cor:chromindex} implies that $\chi'(H) \le n + \O(n^{2/3-\eta})$, thus improving it further. 
 
    For an overview of more related results, methods and open problems, we refer to the recent survey~\cite{KKKMO2021survey}.


\subsection{Organisation of the paper} In Section ~\ref{sec:Notation} we introduce our basic terminology. In Section ~\ref{sec:prelim} we begin with stating the Martingale inequality that we use throughout the paper to prove several concentration inequalities. Another tool is Lemma ~\ref{cor:comb} which allows us to find simple subhypergraphs in hypergraphs with small codegree. In Section ~\ref{proofonline} we introduce our method and give an overview of the proof of our main result (Theorem~\ref{thm:main3intro}). In Section ~\ref{sec:simple},
we analyze the classical R\"odl nibble (on an input hypergraph which is simple) to prove Theorem ~\ref{thm:mainsec3}. Theorem ~\ref{thm:mainsec3} shows that in addition to constructing a large matching, the R\"odl nibble process actually produces more complex structures (which we call `augmenting stars'). In Section ~\ref{sec:AKSimproved}, we then use these augmenting stars to significantly increase the size of the matching produced by the R\"odl nibble. This implies Theorem~\ref{thm:main3intro} in the case where $H$ is simple. In Section ~\ref{sec:KRimproved}, we use Lemma ~\ref{cor:comb} to extend this result from simple hypergraphs to hypergraphs whose codegree is small. Finally, in Section ~\ref{sec:ChromaticIndex}, we deduce Corollary~\ref{cor:chromindex} from Theorem~\ref{thm:main3proof}.

	\section{Notation}
	\label{sec:Notation}
	
	Let $[N]$ denote the set $\{1,2, \ldots, N\}$. We write $c = a \pm b$ if $a-b \leq c \leq a+b$. Given a set $S$ and an integer $i \geq 0$, let $\binom{S}{i}$ be the collection of all subsets of $S$ of size $i$. A \emph{$k$-set} is a set of size $k$. Given a graph $G$, the maximum degree of a vertex in $G$ is denoted by $\Delta(G)$.

	A \emph{hypergraph} $H$ is a pair $(V(H), E(H))$ where $V(H)$ is the vertex set and the edge set $E(H)$ is a set of subsets of $V(H)$. For convenience sometimes we write $H$ instead of $E(H)$. A multi-hypergraph $H$ is a hypergraph which is permitted to have multiple edges. We say $H$ is $k$-uniform if every edge of $H$ contains precisely $k$ vertices. The \emph{degree} of a vertex $x$ in $H$ is the number of edges containing $x$ and is denoted by $d_H(x)$. A (multi-)hypergraph is \emph{regular} if all of its vertices have the same degree, and it is \emph{$D$-regular} if the degree of each of its vertices is $D$. Given a set of vertices $U \subseteq V(H)$, we write $H[U]$ for the sub-hypergraph of $H$ induced by $U$. A \emph{matching} in $H$ is a collection of pairwise disjoint edges.  
	
	For integers $a,b \geq 1$, a (multi)-hypergraph $H$ is \emph{$(a,b)$-partite} if there exists a partition $V(H) = A \cup B$ such that $|e \cap A| = a$ and $|e \cap B| = b$ for any $e \in E(H)$. A (multi)-hypergraph $H$ is \emph{$(a \pm b)$-regular} if every vertex of $H$ has degree $a \pm b$.
	
	We usually denote the number of vertices in a hypergraph by $N$. We say that an event holds \emph{with high probability (whp)} if the probability that it holds tends to 1 as $N \to \infty$.
	
	Given functions $f$ and $g$, we write $f = \O(g)$ if $|f| \le C |g|$ for some constant $C$ (which usually depends on the uniformity $k$). We write $f = \Omega(g)$ if $f \ge c|g|$ for some constant $c > 0$ (which again usually depends on $k$). We write $f = o(g)$ if for any $c > 0$, there exist $M = M(c)$ such that $|f(N)| \leq c|g(N)|$ for $N \geq M$. Similarly, we write $f = \omega(g)$ if for any $C > 0$, there exist $M = M(C)$ such that $|f(N)| \geq C |g(N)|$ for $N \geq M$. We write $\O_{\gamma}$, $\Omega_{\gamma}$, $o_{\gamma}$, $\omega_{\gamma}$ to indicate that the implicit constant may depend on $\gamma$. 
	
	We use the asymptotic `$\ll$' notation to state our results. The constants in the hierarchies used to state our results have to be chosen from right to left. More precisely, if we claim that a result holds whenever $1/N \ll a \ll b \ll c \le 1$, then this means that there are non-decreasing functions $f : (0,1] \mapsto (0,1]$, $g : (0,1] \mapsto (0,1]$ and $h : (0,1] \mapsto (0,1]$ such that the result holds for all $0 < a,b,c \le 1$ and all $N$ with $b \le f(c)$, $a \le g(b)$
and $1/N \le h(a)$. We will not calculate these functions explicitly. Hierarchies with more constants are defined in a similar way.

	\section{Preliminaries}
	\label{sec:prelim}
	
	\subsection{Probabilistic tools}
	In this subsection we collect the large deviation inequalities we need throughout our proof. We will use a martingale inequality due to Alon, Kim and Spencer \cite{AKS1997}. For this, we assume our probability space is generated by a finite set of mutually independent Yes or No choices, indexed by $i \in I$. Let $Y$ be a random variable defined on this space. Let $p_i$ denote the probability that choice $i$ is Yes. Let $c_i$ be such that changing the choice $i$ (but keeping all other choices the same) can change $Y$ by at most $c_i$. We call $p_i(1-p_i)c_i^2$ the \emph{variance} of choice $i$. Let $T$ be an upper bound on all $c_i.$ 
	
	Suppose Paul finds the value of $Y$ by asking queries of an always truthful oracle Carole. The queries always ask whether a choice $i \in I$ is Yes or No. Paul's choice of the next query can depend on Carole's previous responses. Thus, a strategy for Paul can be represented in the form of a decision tree. A \emph{line of questioning} is a path from the root to a leaf of this tree -- a sequence of questions and responses that determine $Y$. The \emph{total variance} of a line of questioning is the sum of variances of the queries in it.

	\begin{lemma}[Martingale inequality~\cite{AKS1997}]\label{cor:martingale}
		There exists $\delta > 0$ such that if Paul has a strategy for finding $Y$ such that every line of questioning has total variance at most $\sigma^2$, then 
		\begin{equation*}
		\mathbb{P}(|Y - \mathbb{E}(Y)| > \lambda \sigma) \leq 2e^{-\frac{\lambda^2}{4}},
		\end{equation*}
		for any $0 < \lambda < 2 \sigma \delta / T$.
	\end{lemma}
	
	We will also use the following.
	
	\begin{lemma}[Chernoff-Hoeffding inequality \cite{janson2000random}]\label{lem:chernoff}
		Suppose $X_1,\ldots, X_t$ are independent
		random variables taking values $0$ or $1$. Let $X := \sum_{i \in [t]}X_i$. Then
		\begin{equation*}
		\mathbb{P}(|X - \mathbb{E}(X)| \geq \lambda) \leq 2 e^{\frac{-\lambda^2}{2(\mathbb{E}(X) + \lambda/3)}}.
		\end{equation*}
		In particular, if $\lambda \ge 7 \mathbb{E}[X]$, then  $\mathbb{P}(|X - \mathbb{E}(X)| \ge \lambda) \le 2e^{-\lambda}$.
	\end{lemma}

	\subsection{Useful estimates}
	We will often use the following estimates in the calculations.\COMMENT{Proofs of both propositions are in the Appendix. In the final version, we can remove them.}
	\begin{proposition}
		\label{lem:boundprod}
		Let $a, b > 0$ with $b \ne 1$. Let $x_0 , \dots , x_n , y_0 , \dots , y_n > 0$ be such that 
		\begin{itemize}
			\item[$(\rm 1)$] $x_i \leq y_i$ and $y_i = a \cdot b^i$ for $0 \leq i \leq n$.
			\item[$(\rm 2)$] If $b \in (0,1)$, then $y_0 \leq \frac{1-b}{2}$. If $b \in (1,\infty)$, then $y_i \leq \frac{1-b^{-1}}{2}$ for $0 \leq i \leq n$.
		\end{itemize}
		Then $$1 - 2\sum_{i=0}^{n} x_i \le \prod_{i=0}^{n} (1 \pm x_i) \le 1 + 2\sum_{i=0}^{n} x_i.$$
	\end{proposition}
	
	The following proposition is straightforward to prove using Taylor's theorem.
	\begin{proposition}
		\label{lem:experror}
		Let $u \in \mathbb{R}$. Then
		\begin{equation*}
		\left | e^u - \left ( 1 + \frac{u}{N} \right )^N \right | =  \O_u \left(\frac{1}{N}\right).
		\end{equation*}
	\end{proposition}

	\subsection{Finding simple subhypergraphs of hypergraphs with small codegrees}

	The next result allows us to find an almost regular simple subhypergraph in an almost regular hypergraph with small codegree. It is a slight generalization of the results of Kostochka and R\"odl~\cite{KR1998}. 
	
	\begin{lemma}\label{cor:comb}
		Let $k \geq 3$ be an integer, and let $\alpha , \beta, D > 0$, $0 < \gamma, \delta  < 1$ be real numbers satisfying 
		$D^{-1} \ll \delta \ll \alpha , \beta , \gamma , k^{-1}.$
		
		Let $H$ be a $k$-uniform multi-hypergraph such that $\alpha D \leq d_H(v) \leq \beta D$ for any $v \in V(H)$ and the codegree of $H$ is at most $C$, where $\log D \leq C \leq D^{1-\gamma}$.
		
		Then there exists $E_{\rm simp} \subseteq E$ such that the hypergraph $H_{\rm simp} := (V(H) , E_{\rm simp})$ satisfies the following conditions.
		\begin{itemize}
			\item[$(\rm i)$] $H_{\rm simp}$ is simple. (In particular, $H_{\rm simp}$ does not have multiple edges.)
			\item[$(\rm ii)$] There exists an even integer $s \in \left (1 + 2 \binom{k}{2}\delta^{-1}\:,\: 3 + 2 \binom{k}{2} \delta^{-1} \right )$ such that for every $v \in V(H)$ we have
			\begin{equation*}
			d_{H_{\rm simp}}(v) = \left ( 1 - \frac{1}{\log(D/C)} \right )^s \frac{d_H (v)}{\log(D/C) C^{1-\delta}D^{\delta}} \pm 8s \sqrt{ \frac{(s+1) d_H(v)}{C^{1-\delta} D^\delta}}.
			\end{equation*}
		\end{itemize}
	\end{lemma}

The proof of Lemma~\ref{cor:comb} consists of several probabilistic edge-colouring arguments and is very similar to the proof of Theorem 4 in~\cite{KR1998} (but omitting the final step of the proof of Theorem 4, which consists of an application of Theorem 3 in~\cite{KR1998}). Therefore, we omit the proof of Lemma~\ref{cor:comb} here, full details are given in Appendix~\ref{appendix:simplesubhypergraphs}.
	
\section{An outline of our method}
\label{proofonline}
In this section we will introduce our method by sketching the proof of Theorem ~\ref{thm:main3intro} in the case when $H$ is simple. For the proof to work, slightly more generally, we will assume $H$ is $(D \pm K D^{\varepsilon})$-regular, rather than $D$-regular (see Theorem ~\ref{thm:main2proof} for the precise statement in the case when $H$ is simple, and which also guarantees a `quasirandomness property'). Using an argument based on random edge-colourings (Lemma~\ref{cor:comb}) it is then not difficult to extend the result to the case when $H$ has small codegree. 
	
	Let us start by outlining the Nibble process considered in \cite{AKS1997}. Suppose $H$ is a $k$-uniform simple hypergraph on $N$ vertices. 
	
	\vspace{2mm}
	
	\noindent \textbf{The Nibble process.}
	The Nibble process for $H$ consists of $\zeta$ stages defined as follows. Start with $H_0 := H$ and let $M_0 := \emptyset$, $U_{-1} := U_0 := V(H)$, $N_0 := N$, $D_0 := D$ and $\Delta_0 := K D^{\varepsilon}$. Let $1 \le i \le \zeta$.
	
	\vspace{2mm}
	
	\begin{center}
		\begin{minipage}{\textwidth}
			\textbf{The $i$th stage of Nibble process.} Suppose we are given a hypergraph $H_{i-1}$ with $V(H_{i-1}) = U_{i-1}$, where $H_{i-1}$ is $(D_{i-1} \pm \Delta_{i-1})$-regular.
			
			\vspace{2mm}
			
			Select every edge of $H_{i-1}$ with probability $1/D_{i-1}$, and let $B_i$ be the set of selected edges.
			Independently, we also carefully select a random subset $W_i \subseteq U_{i-1}$. 
			The $i$th stage of the Nibble process is determined by a choice of $B_i$ and $W_i$. Given $B_i$, $W_i$, we set 
			$$M_i := \{ e \in B_i : e \cap e' = \emptyset \text{ for all } e' \in B_i \text{ with } e' \ne e\}.$$
			
			(Note that a vertex $w \in W_i$ may also lie in some edge $e \in M_i$.) The edges of $M_i$ are clearly pairwise disjoint. Now set
			$$U_i := U_{i-1} \setminus (V(M_i) \cup W_i)$$ and 
			$$H_i := H_{i-1}[U_i].$$
			We show that $H_i$ is $(D_i \pm \Delta_i)$-regular where $D_i \sim (1 - e^{- k})^{k-1}D_{i-1} \sim (1 - e^{- k})^{i(k-1)} D_0$,  $\Delta_i = o(D_i)$ and $|U_i| \sim (1 - e^{-k})|U_{i-1}|$.
			
		\end{minipage}
	\end{center}
	
	\vspace{2mm}
	
We iterate the above process for $1 \le i \le \zeta$ until we get $H_{\zeta}$ with vertex-set $U_{\zeta}$, where $\zeta$ is chosen such that 
\begin{equation}\label{eq:Domega}
    D_{\zeta} \le D_0^{\gamma} \text{ but } D_{\zeta-1} > D_0^{\gamma},  
\end{equation}
at which point we stop the process. Now note that $M := \bigcup_{1 \le i\le\zeta} M_i$ is a matching of $H$. Let $W := \bigcup_{1 \le i\le\zeta}  W_i$  be the set of \emph{waste vertices}. (The purpose of $W_i$ is to assist in bounding $\Delta_i$.) Then $U_\zeta = V(H) \setminus (V(M) \cup W)$ is the set of \emph{leftover vertices} at the end of the Nibble process. 
	
	\vspace{2mm}
	\noindent \textbf{Augmenting stars.} Our key new idea is to show that there are still many augmenting structures at the end of the Nibble process that can be used to improve $M$ and produce a larger matching. The augmenting structures we use are called \emph{augmenting stars}. These augmenting structures and the hypergraph $H_A$ associated with them are defined as follows.

	\begin{defn}[Augmenting stars]
		Let $H$ be a $k$-uniform hypergraph and let $M$ be a matching of $H$ and $W \subseteq V(H)$. 
		
		An augmenting star of $H$ with respect to $(M,W)$ is an ordered triple $(e_M , \{e_1 , \dots , e_{k} \}) \in E(M) \times \binom{E(H)}{k}$ satisfying the following conditions. 

		\begin{itemize}
			\item $e_1 , \dots , e_k \in E(H)$ are pairwise disjoint.
			\item $|e_i \cap e_M| = 1$ for $1 \leq i \leq k$ and $(e_1 \cup \dots \cup e_k) \setminus e_M \subseteq V(H) \setminus (V(M) \cup W)$.
		\end{itemize}
	\end{defn}
	Let $\mathcal{A}(H)$ denote the collection of all augmenting stars of $H$ with respect to $(M,W)$. 
	
	\begin{defn}[Multi-hypergraph of augmenting stars]
		\label{Def:HypergraphofAugStars}
		Let $M$ be a matching of a $k$-uniform hypergraph $H$ and $W \subseteq V(H)$. The multi-hypergraph $H_A$ of augmenting stars of $H$ with respect to $(M,W)$ is the $(1 , k(k-1))$-partite multi-hypergraph with the vertex-set
		$$ V(H_A) := E(M) \cup (V(H) \setminus (V(M) \cup W)) $$
		and with the multiset of edges
		$$ E(H_A) := \{ \{ e_M \} \cup ((e_1 \cup \dots \cup e_k) \setminus e_M) \: : \: (e_M , \{e_1,\dots,e_k\}) \in \mathcal{A}(H) \}.$$
	\end{defn}
	
	\vspace{2mm}
	\noindent \textbf{Improving $M$ using Augmenting stars.}
        \begin{figure}
        \centering
        \begin{tikzpicture}[scale=0.85]
        
        \filldraw[fill=gray!15, draw=black] (0,0) rectangle (6,8);
        \node at (3,-0.3) {\small Vertices in the matching $M$};
        
        \filldraw[fill=gray!15, draw=black] (0,8.5) rectangle (6,9.5);
        \node at (3,9.8) {\small Waste vertices in $W$};
        
        \filldraw[fill=gray!15, draw=black] (8,0) rectangle (9,9.5);
        \node at (8.5, -0.3) {\small $U_\zeta = V(H) \setminus (V(M) \cup W)$};
        
        \foreach \x in {0,1,2}
        {
        	\fill (0.8 + 0.2*\x, 9) circle (1.5pt);
        	\fill (2.8 + 0.2*\x, 9) circle (1.5pt);
        	\fill (4.8 + 0.2*\x, 9) circle (1.5pt);
        }
        
        \draw[fill=gray!35] (1,1.5) ellipse [x radius=0.5,y radius=1];
        \draw[fill=gray!35] (1,4) ellipse [x radius=0.5,y radius=1];
        \draw[fill=gray!35] (1,6.5) ellipse [x radius=0.5,y radius=1];
        
        \draw[fill=gray!35] (3,1.5) ellipse [x radius=0.5,y radius=1];
        \draw[fill=gray!35] (3,4) ellipse [x radius=0.5,y radius=1];
        \draw[fill=gray!35] (3,6.5) ellipse [x radius=0.5,y radius=1];
        
        \draw[fill=gray!35] (5,1.5) ellipse [x radius=0.5,y radius=1];
        \draw[fill=gray!35] (5,4) ellipse [x radius=0.5,y radius=1];
        \draw[fill=gray!35] (5,6.5) ellipse [x radius=0.5,y radius=1];
        
        \foreach \x in {0,1,2}
        {
        	\foreach \y in {0,1,2}
        	{
        		\fill (2 * \x + 1 , 2.5 * \y + 0.9) circle (0.1cm);
        		\fill (2 * \x + 1 , 2.5 * \y + 1.3) circle (0.1cm);
        		\fill (2 * \x + 1 , 2.5 * \y + 1.7) circle (0.1cm);
        		\fill (2 * \x + 1 , 2.5 * \y + 2.1) circle (0.1cm);
        	}
        }
        
        \foreach \x in {0,1,2,3}
        {
        	\filldraw[fill=gray!45] (5, 0.9+0.4*\x)--(8.4 , \x + 0.4)--(8.4 , 1.03*\x + 0.6 + 0.6) -- (5, 0.9+0.4*\x);
        	\filldraw[fill=gray!45] (5, 5.9+0.4*\x)--(8.4 , \x + 0.4 + 5)--(8.4 , 1.03*\x + 0.6 + 0.6 + 5) -- (5, 5.9+0.4*\x);
        
        	\foreach \y in {0,1,2}
        	{
        		\fill (8.3 , \x + 0.3*\y + 0.5) circle (1.5pt);
        		\fill (8.3 , \x + 0.3*\y + 5.5) circle (1.5pt);
        		\fill (8.7 , \x + 0.3*\y + 0.5) circle (1.5pt);
        		\fill (8.7 , \x + 0.3*\y + 5.5) circle (1.5pt);
        	}
        }
        \foreach \y in {0,1,2}
        {
        	\fill (8.3 , 3 + 0.3*\y + 1.5) circle (1.5pt);
        	\fill (8.7 , 3 + 0.3*\y + 1.5) circle (1.5pt);
        }
        \end{tikzpicture}
        \caption{Improving the matching $M$ using augmenting stars}
		\label{fig:augmentingM}
	\end{figure}
	Let $H_A$ be the multi-hypergraph of augmenting stars of $H$ with respect to $(M,W)$. Note that $H_A$ is $(k(k-1)+1)$-uniform and each edge of $H_A$ contains exactly one element of $L := E(M)$ and $k(k-1)$ elements of $R := V(H) \setminus (V(M) \cup W)$. Analysing the Nibble process one can (easily) show that $|L| = \Theta(N)$ and $|R| = \Theta (N D^{\frac{\gamma-1}{k-1}})$.  
	
	Suppose, by some means, we found a large matching $M_A$ in $H_A$. Then each $\{ e_M \} \cup ((e_1 \cup \dots \cup e_k) \setminus e_M) \in E(M_A)$ corresponds to an augmenting star $(e_M , \{e_1,\dots,e_k\}) \in \mathcal{A}(H)$. Thus for each $\{ e_M \} \cup ((e_1 \cup \dots \cup e_k) \setminus e_M) \in E(M_A)$, replacing the edge $e_M$ of $M$ with the edges $e_1 , \dots , e_k$  produces a larger matching $M^*$ of $H$ (see Figure ~\ref{fig:augmentingM}).
	
	Hence, our goal is to find a large matching in $H_A$. To this end, we want to show that $H_A$ is almost regular and has small codegree so that we may apply a variant of Theorem ~\ref{thm:kr}. More precisely, we prove the following crucial properties of $H_A$ (in $(\rm M3)$ of Theorem~\ref{thm:mainsec3}). 
	
	\begin{itemize}
		\item[(a)] Every vertex in $L$ has degree $(1 + \O (K D^{\varepsilon - 1} + D^{-\gamma/2} \log N)) D_{L}$ where $D_L = \Theta (D^{k \gamma})$.
		\item[(b)] Every vertex in $R$ has degree $(1 + \O (K D^{\varepsilon - 1} + D^{-\gamma/2} \log N)) D_{R}$ where $D_{R} = \Theta ( D^{k \gamma + \frac{1-\gamma}{k-1}})$.
		\item[(c)] The codegree of $H_A$ is at most $\O (D^{\gamma(k-1)} \log^2 N)$.
	\end{itemize}
	
	Looking at properties (a) and (b) one immediately notices that $H_A$ is not almost regular; the degrees of vertices in $L$ are smaller than the degrees of vertices in $R$. However, crucially, properties (a) and (b) show that the degrees of vertices in $L$ are close to one another, and the degrees of vertices in $R$ are close to one another. To overcome the problem of almost regularity, we boost the degrees of vertices in $L$ as follows. We take $D_R/D_L$ vertex-disjoint copies of $R$, say $R^1, R^2, \dots, R^{D_R / D_L}$, where $R^i := \{ v^i \: : \: v \in R \}$, and define an auxiliary hypergraph $H_A'$ with
	$ V(H_A') := L \cup (R^1 \cup \dots \cup R^{D_R / D_L})$ and for each edge $\{e\} \cup \{v_1 , \dots , v_{k(k-1)} \} \in E(H_A)$, we add
	$ \{e\} \cup \{v_1^i , \dots , v_{k(k-1)}^i \} $ to be an edge in $H_A'$ for all $1 \leq i \leq D_R / D_L$. Now the degree of any vertex in $H_A'$ is sufficiently close to $D_{R},$ so $H_A'$ is almost regular. Moreover, this boosting does not increase the codegree, so $H_A'$ also has small codegree, as desired. Thus applying Lemma ~\ref{cor:comb} we find a simple almost regular subhypergraph of $H_A'$, to which we apply Theorem ~\ref{thm:mainsec3} to find a large matching in $H_A'$, which then provides the required large matching $M_A$ in $H_A$. 
	
	A large part of this paper is devoted to proving properties (a), (b) and (c). This is done in Lemma~\ref{lem:nibble} (Nibble lemma) by carefully tracking random variables that count certain configurations (defined in Section~\ref{subsec:keyrandomvariables}) throughout the Nibble process and proving concentration inequalities for them using a martingale inequality. This lemma is at the heart of our proof.
	
	\vspace{2mm}
	\noindent \textbf{Extensions and limitations of the method.}
	It is worth noting that the bound in Theorem ~\ref{thm:main3intro} can be further improved by iterating our method as follows. 
	
	To augment the matching $M$ of $H$, our strategy was to find a large matching $M_A$ in the multi-hypergraph $H_A$ of augmenting stars of $H$ by applying Theorem~\ref{thm:mainsec3} (to $H_A'$). However, Theorem~\ref{thm:mainsec3} not only provides the matching $M_A$, but it also ensures the property $(\rm M3)$. Thus, one can consider the multi-hypergraph $H_{AA}$ of augmenting stars of $H_A$ and use $(\rm M3)$ to guarantee the properties analogous to (a), (b) and (c) for $H_{AA}$. This allows us to similarly improve $M_A$ to obtain $M^*_A$. One can then use $M^*_A$ instead of $M_A$ for improving $M$ to obtain an even larger matching in $H$. 
	
    Note that one can iterate the above procedure by considering the multi-hypergraph of augmenting stars of $H_{AA}$ again and so on, to improve the size of the matching even further. However, the uniformities of these hypergraphs grow very quickly, so we do not expect a significant improvement. Hence, for the sake of presentation, we will not optimise our bound in Theorem ~\ref{thm:main3intro}.
	
	Also note that the reason why Theorem ~\ref{thm:main3intro} does not apply when $k =3$ is that the size of the set $W$ of waste vertices that we need to remove to keep the degree error low during the Nibble process will be too large in this case.

	\section{Finding matchings with augmenting stars \texorpdfstring{\\}{} in almost regular simple hypergraphs}
	\label{sec:simple}
	
	In this section, we prove the following theorem.

	\begin{theorem}\label{thm:mainsec3}
		Let $k > 3$, $\Delta_0, D_0, N$ be integers and let $\gamma, \varepsilon \in (0,1)$ and $K > 0$ be real numbers satisfying
		\begin{align}
		\label{eq:settingupparameters}
		0 < N^{-1} \ll \gamma, \varepsilon, k^{-1} , K^{-1} < 1, \hspace{6mm} D_0 \geq \log^{\frac{8(k-1)}{\gamma(1 - \varepsilon)}} N, \hspace{3mm} \text{and} \hspace{3mm} \Delta_0 &\leq KD_0^\varepsilon.
		\end{align}
		
		Let $H$ be a $k$-uniform $(D_0 \pm \Delta_0)$-regular simple hypergraph on $N$ vertices. Let $\mathcal{V}$ be a collection of subsets of $V(H)$, where $|S| \geq \sqrt{D_0} \log N$ for every $S \in \mathcal{V}$, and $|\mathcal{V}| \leq \exp(\log^{3/2} N)$.
		
		Then there exists a matching $M$ of $H$, and a set $W \subseteq V(H)$ of waste vertices satisfying the following.
		\begin{itemize}
			\item[$(\rm M1)$] There exists $p = \Theta(D_0^{\frac{\gamma-1}{k-1}})$ such that 
			for every $S \in \mathcal{V}$, we have
			\begin{equation}\label{eqn:totaluncov}
			|S \setminus (V(M) \cup W)| = (1+o(1))p|S|.
			\end{equation}
			\item[$(\rm M2)$] For every $S \in \mathcal{V}$,
			\begin{equation}\label{eqn:totalwaste}
			|S \cap W|  = \O (|S| D_0^{-1}\Delta_0 + |S| D_0^{-\frac{1}{k-1}} D_0^{\gamma \left ( \frac{1}{k-1} - \frac{1}{2} \right )} \log N).
			\end{equation}
			\item[$(\rm M3)$] Let $H_A$ be the $(1,k(k-1))$-partite multi-hypergraph of augmenting stars of $H$ with respect to $(M,W)$. Then for any $e \in E(M)$ and $x \in V(H) \setminus (V(M) \cup W)$,
			\begin{align*}
			d_{H_A}(e) &= (1 + \O (D_0^{-1}\Delta_0 + D_0^{-\gamma/2} \log N)) D_{L},\\
			d_{H_A}(x) &= (1 + \O (D_0^{-1}\Delta_0 + D_0^{-\gamma/2} \log N)) D_{R},
			\end{align*}
			where $D_L = \Theta ( D_0^{k\gamma} )$, and $D_{R} = \Theta( D_0^{k \gamma + \frac{1-\gamma}{k-1}})$. Moreover, the codegree of $H_A$ is at most
			$\O (D_0^{\gamma(k-1)} \log^2 N)$.
		\end{itemize}
	\end{theorem}
	
	The constants in $(\rm M3)$ implicit in $\O(\cdot)$, $\Theta(\cdot)$ do not depend on $e$ or $x$. More generally, throughout this section, the constants implicit in $\O(\cdot)$, $\Theta(\cdot)$ and $o(\cdot)$ depend only on the parameters $k$, $K$, $\gamma$, $\varepsilon$. For convenience we mostly do not indicate these dependencies.
	
	Note that the properties $(\rm M1)$ and $(\rm M2)$ bound the number of leftover vertices and waste vertices in $S \in \mathcal V$. The crucial property is $(\rm M3)$. Much of this section is devoted to proving it. In the following subsections we develop the required tools and using them we finish the proof of Theorem ~\ref{thm:mainsec3} in Section ~\ref{subsection:theorem3.1}. 
	
	\subsection{The Nibble process}
	Throughout the remainder of Section~\ref{sec:simple}, let $k > 3$, $\Delta_0$, $D_0$, $N$ be integers, and $\gamma, \varepsilon \in (0,1)$, $K > 0$ be real numbers satisfying \eqref{eq:settingupparameters}, and suppose $H$ is a $k$-uniform $(D_0 \pm \Delta_0)$-regular simple hypergraph on $N$ vertices. For convenience, define $p_{-1}^* := 0$.

	Let $1 \leq i \leq \zeta$. Suppose we are given the hypergraph $H_{i-1}$ with $V(H_{i-1}) = U_{i-1}$, and parameters $D_{i-1}, \Delta_{i-1}$ such that for all $x \in V(H_{i-1})$, 
	\begin{equation}
	\label{Hi-1degrees}
	d_{H_{i-1}}(x) = D_{i-1} \pm \Delta_{i-1}, \text{ where } \Delta_{i-1} = o(D_{i-1}).
	\end{equation}
	
	Let $B_i \subseteq E(H_{i-1})$ such that 
	$$\mathbb{P}(F \in B_i) = \frac{1}{D_{i-1}}$$
	for all $F \in E(H_{i-1})$, where the events $F \in B_i$ are mutually independent.  
	
	Let $M_i$ be the set of isolated edges in $B_i$. For any vertex $v \in U_{i-1}$, let $$p_{M_i}(v) := \mathbb{P}(v \in V(M_i))$$ and $$p_{i-1}^* := \max_{v \in U_{i-1}} p_{M_i}(v).$$

	Now let $W_i$ be a random subset of $V(H_{i-1}) = U_{i-1}$, the events $v \in W_i$ being mutually independent (and also independent of the choice of edges in $B_i$) such that
	$$\mathbb{P}(v \in W_i) = p_{W_i}(v),$$ where $p_{W_i}(v)$ is defined by 
	\begin{equation}\label{eqn:pwi}
	p_{M_i}(v)+p_{W_i}(v) -p_{M_i}(v)p_{W_i}(v) = p_{i-1}^*.
	\end{equation}
A vertex of $W_i$ is called a \emph{waste vertex}. Let $$U_i := U_{i-1} \setminus (V(M_i) \cup W_i)$$ and $$H_i := H_{i-1}[U_i].$$
	By \eqref{eqn:pwi}, for any vertex $v \in U_{i-1}$, $$\mathbb{P}(v \in V(M_i) \cup W_i) = p_{M_i}(v)+ p_{W_i}(v) - p_{M_i}(v)p_{W_i}(v) = p_{i-1}^*.$$  Thus \eqref{eqn:pwi} ensures that for any $v \in U_{i-1}$,
	\begin{equation}
	\label{eqn:pUi}
	\mathbb{P}(v \in U_i) = 1 - p_{i-1}^*.
	\end{equation}

	The following proposition estimates the probability that an edge or a vertex is in the matching $M_i$.
	
	\begin{proposition}\label{claim:matching}
		Let $1 \leq i \leq \zeta$. Suppose $H_{i-1}$ is $(D_{i-1} \pm \Delta_{i-1})$-regular, where $\Delta_{i-1} = o(D_{i-1})$. Then there are constants $C_0 = C_0(k)$ and $C_1 = C_1(k)$ independent of the index $i$ satisfying the following. 
		\begin{itemize}
		    \item[$(\rm i)$] For any $F \in E(H_{i-1})$, 
	$$ \mathbb{P}(F \in E(M_i)) = \left (1 \pm C_0 \frac{\Delta_{i-1}}{D_{i-1}} \right ) \frac{e^{- k}}{D_{i-1}}. $$
	       \item[$(\rm ii)$] 	For any $v \in U_{i-1}$, 
		$$ p_{M_i}(v) = \left (1 \pm C_1 \frac{\Delta_{i-1}}{D_{i-1}} \right )  e^{- k}. $$
		\item[$(\rm iii)$] We have
		$$ p_{i-1}^* = \left(1 \pm C_1 \frac{\Delta_{i-1}}{D_{i-1}}\right)  e^{- k}. $$
		\end{itemize}	
	\end{proposition}
	\begin{proof}
		First observe that Proposition~\ref{lem:experror} implies that
		\begin{align}\label{eqn:qi}
		\left (1 -  D_{i-1}^{-1} \right )^{k D_{i-1} - k} = \left ( 1 \pm C D_{i-1}^{-1} \right )  e^{- k}
		\end{align}
		for some constant $C = C(k)$ independent of the index $i$.\COMMENT{By Proposition~\ref{lem:experror}, we have
		$\left ( 1 -  D_{i-1}^{-1} \right )^{D_{i-1}}  = e^{-1} + \O ( D_{i-1}^{-1}  )  = \left(1 + \O (D_{i-1}^{-1})\right) e^{-1}.$
		Hence 
		\begin{align*}
		\left (1 - D_{i-1}^{-1} \right )^{kD_{i-1}} &= \left( 1 + \O ( D_{i-1}^{-1} ) \right)^k e^{- k} = (1 + \O(D_{i-1}^{-1}) + \O(D_{i-1}^{-2})) e^{-k}. 
		\end{align*}
		Moreover,
		\begin{align*}
		\left ( 1 - D_{i-1}^{-1} \right )^{-k} &= 1 + k D_{i-1}^{-1} + \O ( D_{i-1}^{-2}). 
		\end{align*}
		Thus~\eqref{eqn:qi} follows by multiplying the above two equations.}
		
		Using~\eqref{eqn:qi}, now we prove the proposition. For each $F \in E(H_{i-1})$, let $t(F)$ be the number of edges in $H_{i-1}$ intersecting $F$. Since $H_{i-1}$ is $(D_{i-1} \pm \Delta_{i-1})$-regular and simple, we have $t(F) = k(D_{i-1} \pm \Delta_{i-1} -1)$. Moreover, we assumed $\Delta_{i-1} = o(D_{i-1})$. Thus we have
		\begin{align*}
		\mathbb{P}(F \in E(M_i)) &= \frac{1}{D_{i-1}} \left (1 - \frac{1}{D_{i-1}} \right )^{t(F)} 
		\overset{\eqref{eqn:qi}}{=} \frac{ e^{- k}}{D_{i-1}} \left (1 \pm \frac{C}{D_{i-1}} \right )\left (1 - \frac{1}{D_{i-1}} \right )^{\pm k \Delta_{i-1}} \\
		&= \frac{ e^{- k}}{D_{i-1}} \left (1 \pm \frac{C}{D_{i-1}} \right ) \left ( 1 \pm k \frac{\Delta_{i-1}}{D_{i-1}} + \O \left ( \frac{\Delta_{i-1}^2}{D_{i-1}^2} \right ) \right )
		= \frac{ e^{- k}}{D_{i-1}} \left (1 \pm \frac{C_0 \Delta_{i-1}}{D_{i-1}} \right )
		\end{align*}
		for some constant $C_0 = C_0(k)$ independent of the index $i$, which proves $(\rm i)$. Since $H_{i-1}$ is $(D_{i-1} \pm \Delta_{i-1})$-regular, the probability that $v \in V(M_i)$ is 
		\begin{align*}
		\sum_{F~:~v \in F} \mathbb{P}(F \in E(M_i)) &= \left ( D_{i-1} \pm \Delta_{i-1}\right ) \frac{ e^{- k}}{D_{i-1}}   \left (1 \pm \frac{C_0 \Delta_{i-1}}{D_{i-1}} \right ) = \left (1 \pm \frac{C_1 \Delta_{i-1}}{D_{i-1}} \right )  e^{- k}
		\end{align*}
		for some constant $C_1 = C_1(k)$ independent of the index $i$, proving $(\rm ii)$ and $(\rm iii)$ as desired.
	\end{proof}

	\subsection{Introducing the key random variables}
	\label{subsec:keyrandomvariables}
	Below we define the key random variables that we track throughout the Nibble process.

	\begin{defn}
		For any $0 \leq i \leq \zeta$ and $x \in V(H)$, let $D_i(x)$ be the number of edges $e \in E(H)$ incident to $x$ such that $e \setminus \{x \} \subseteq U_i$. We call such an edge $e \in E(H)$ \emph{an instance of $D_i(x)$}.
	\end{defn}
	
	\begin{defn}
		For all vertices $x \in V(H) = U_{-1}$, let us define $Z_0(x) := 0$. For $1 \leq i \leq \zeta$ and $x \in U_{i-1}$, let $Z_i(x)$ be the number of edges $e \in E(H)$ incident to $x$ such that one vertex of $e \setminus \{ x \}$ is in $\bigcup_{j \leq i} M_j$ and the remaining $k-2$ vertices of $e \setminus \{ x \}$ are in $U_i$. We call such an edge $e \in E(H)$ \emph{an instance of $Z_i(x)$}.
	\end{defn}

	\begin{defn}
		For $0 \leq i \leq \zeta$ and any distinct $x,y \in U_{i-1}$, let $Y_i(x,y)$ be the number of ordered triples $(e_1 , e_2 , e_3) \in E(H)^3$ of edges satisfying the following conditions.
		\begin{itemize}
			\item[$(\rm Y1)_i$] $x \in e_1$, $y \in e_2$ and $e_1 \cap e_2 = \emptyset$.
			\item[$(\rm Y2)_i$] $e_1 \cap e_3 = \{x' \}$, $e_2 \cap e_3 = \{y' \}$ such that $x,y,x',y'$ are distinct.
			\item[$(\rm Y3)_i$] $(e_1 \cup e_2 \cup e_3) \setminus \{x,y\} \subseteq U_i$.
		\end{itemize}
		We call such an ordered triple $(e_1 , e_2 , e_3) \in E(H)^3$ \emph{an instance of $Y_i(x,y)$}, and $e_3$ the \emph{central edge} of the ordered triple $(e_1 , e_2 , e_3)$ (see Figure~\ref{fig:x_iy_i}).
	\end{defn}
	
	For any given pair of distinct vertices $x,y \in V(H)$, we can estimate $Y_0(x,y)$ as follows. There are $d_H(x) \leq 2D_0$ choices of $e_1 \in E(H)$ incident to $x$. For each choice of $e_1$ with $e_1 \setminus \{ x \} =: \{x_1 , \dots , x_{k-1} \}$, the number of possible choices for $e_3$ is at most $d_H(x_1) + \cdots + d_H(x_{k-1}) \leq 2(k-1) D_0$. Once we have chosen $e_3$, there are $k-1$ vertices $z_1 , \dots , z_{k-1}$ in $e_3 \setminus e_1$, and for each $z_j$ ($1 \leq j \leq k-1$) there is at most one edge containing both $z_j$ and $y$ since $H$ is simple. Hence there are at most $k-1$ choices of $e_2$. In total,
	\begin{equation}\label{eqn:boundy0}
	Y_0(x,y) \leq 2D_0 \cdot 2(k-1) D_0 \cdot (k-1) \leq 4k^2 D_0^2.
	\end{equation}

	\begin{defn}
		For $0 \leq i \leq \zeta$ and any distinct $x,y \in U_{i-1}$, let $X_i(x,y)$ be the number of ordered triples $(e_1 , e_2 , e_3) \in E(H)^3$ of edges satisfying the following conditions:
		\begin{itemize}
			\item[$(\rm X1)_i$] $x \in e_1$, $y \in e_2$ and $e_1 \cap e_2 = \emptyset$.
			\item[$(\rm X2)_i$] $e_1 \cap e_3 = \{x' \}$, $e_2 \cap e_3 = \{y' \}$ such that $x,y,x',y'$ are distinct.
			\item[$(\rm X3)_i$] $(e_1 \cup e_2) \setminus \{x,x',y,y' \} \subseteq U_i$.
			\item[$(\rm X4)_i$] $e_3 \in \bigcup_{j \leq i} E(M_j)$.
		\end{itemize}
		We call such an ordered triple $(e_1 , e_2 , e_3) \in E(H)^3$ \emph{an instance of $X_i(x,y)$}, and $e_3$ the \emph{central edge} of the ordered triple $(e_1 , e_2 , e_3)$ (see Figure~\ref{fig:x_iy_i}).
	\end{defn}
	Note that $X_0(x,y) = 0$ for any distinct $x,y \in V(H)$ since $M_0 = \emptyset$.

     \begin{figure}
        \centering
        \begin{tikzpicture}[scale=0.85]
        \filldraw[fill=gray!35, draw=black] (3.8,-0.2) rectangle (6.2,2.2);
        \draw (1.8,-2.2) rectangle (4.2,0.2);
        \draw (5.8,0.2) rectangle (8.2,-2.2);

        \fill(4, 0) circle (0.1cm);
        \fill(4, 2) circle (0.1cm);
        \fill(6, 0) circle (0.1cm);
        \fill(6, 2) circle (0.1cm);
        
        \fill(2, 0) circle (0.1cm);
        \filldraw[fill=red, draw=black] (2, -2) circle (0.1cm);
        \fill(4, -2) circle (0.1cm);

        \fill(8, 0) circle (0.1cm);
        \fill(6, -2) circle (0.1cm);
        \filldraw[fill=red, draw=black] (8, -2) circle (0.1cm);
        
        \node at (5,1) {$e_3$};
        \node at (3,-1) {$e_1$};
        \node at (7,-1) {$e_2$};
        \node at (3.6,-0.4) {\small $x'$};
        \node at (6.4,-0.4) {\small $y'$};
        \node at (1.6,-2.4) {\small $x$};
        \node at (8.4,-2.4) {\small $y$};

        \draw (12.8,-0.2) rectangle (15.2,2.2);
        \draw (10.8,-2.2) rectangle (13.2,0.2);
        \draw (14.8,0.2) rectangle (17.2,-2.2);

        \fill(13, 0) circle (0.1cm);
        \fill(13, 2) circle (0.1cm);
        \fill(15, 0) circle (0.1cm);
        \fill(15, 2) circle (0.1cm);
        
        \fill(11, 0) circle (0.1cm);
        \filldraw[fill=red, draw=black] (11, -2) circle (0.1cm);
        \fill(13, -2) circle (0.1cm);

        \fill(17, 0) circle (0.1cm);
        \fill(15, -2) circle (0.1cm);
        \filldraw[fill=red, draw=black] (17, -2) circle (0.1cm);
        
        \node at (14,1) {$e_3$};
        \node at (12,-1) {$e_1$};
        \node at (16,-1) {$e_2$};
        \node at (12.6,-0.4) {\small $x'$};
        \node at (15.4,-0.4) {\small $y'$};
        \node at (10.6,-2.4) {\small $x$};
        \node at (17.4,-2.4) {\small $y$};
        \end{tikzpicture}
        \caption{An example of an instance of $X_i(x,y)$ and $Y_i(x,y)$ respectively.}\label{fig:x_iy_i}
    \end{figure}
	
	\subsection{The Nibble lemma and its analysis}
	\label{sec:Nibblelemma}
	Before stating the Nibble lemma, we define the following parameters recursively for all $0 \leq j \leq i \leq \zeta$.

	\begin{align}
	&\text{If $j = i$, then $q_{j,i} := 1$. If $j < i$, then } q_{j,i} := (1 - p_{i-1}^*)q_{j,i-1}, \label{qjidef}\\
	&D_i := (1 - p_{i-1}^*)^{k-1} D_{i-1}, \label{Didef}\\
	&\Delta_i := (1 - p_{i-1}^*)^{k-1} \Delta_{i-1} + \sqrt{D_{i-1}} \log N. \label{definingDeltai}
	\end{align}
	Recursively applying the above definition shows that 
	\begin{equation}
	\label{eq:***}
	 q_{j,i} = \prod_{u=j}^{i-1}(1 - p_u^*) \text{ and } D_i = \prod_{u=0}^{i-1} (1 - p_u^*)^{k-1} D_0 = q_{0,i}^{k-1} D_0.   
	\end{equation}
It is worth noting that \eqref{eq:Domega} implies that for any $0 \leq i < \zeta$, 
\begin{equation}
\label{eq:*}
D_i \ge D_{0}^{\gamma} = \Omega(\log^{\frac{8}{(1 - \varepsilon)}} N)   
\end{equation}
by \eqref{eq:settingupparameters}. Moreover, by Proposition~\ref{claim:matching}(iii) if $\Delta_{\zeta-1} = o(D_{\zeta-1})$ then 
$D_{\zeta} = \Theta(D_{0}^{\gamma}) = \Omega(\log^{\frac{8}{(1 - \varepsilon)}} N)$. 
	
For convenience, let $W_0 := \emptyset$. The following lemma is the heart of our proof.
	
	\begin{lemma}[Nibble lemma]\label{lem:nibble}
		For $0 \leq i \leq \zeta$, define the following statements.
		\begin{itemize}
			\item[$(\rm L1)_i$] $D_i(x) = D_i \pm \Delta_i$ for any $x \in V(H)$.
			
			\item[$(\rm L2)_i$] For each $S \in \mathcal{V}$, we have 
			\begin{itemize}
			    \item[\rm ($a$)]   $$|S \cap W_i| \leq 2C_5 \frac{|S \cap U_{i-1}| \Delta_{i-1}}{D_{i-1}} + \frac{|S \cap U_{i-1}| \log N}{\sqrt{D_{i-1}}},$$
			    \item[\rm ($b$)]  $$|S \cap U_i| = (1 \pm |S \cap U_{i-1}|^{-1/4}) |S \cap U_{i-1}| (1 - p_{i-1}^*),$$ 
			    \item[\rm ($c$)] $$|S \cap U_i| = (1+o(1))(1 - e^{-k})^{i}|S|.$$
			\end{itemize}
			\item[$(\rm L3)_i$] For any $x \in U_{i-1}$, we have\COMMENT{Note that (b) and (c) are asymptotic forms of (a), and are more useful in our applications. The reader is not supposed to immediately understand how (b) and (c) follow from (a). This is shown in the proof of Nibble lemma.} 
			\begin{align*}
			(a)\hspace{1mm} Z_i(x) =& (k-1)e^{- k} \sum_{0\leq j<i} q_{j,i}^{k-2} D_j 
			\\ &\pm  \sum_{0\leq j<i} \left[ q_{j,i}^{k-2}  \left ( \frac{C_9 Z_j(x)}{D_j} + (C_1 + 3)(k-1)e^{-k} \Delta_j \right ) + q_{j+1,i}^{k-2} \sqrt{Z_j(x) + D_j} \log N \right ],
			\\
			(b)\hspace{1mm} Z_i(x) =& (k-1)e^{- k} \sum_{0\leq j<i} q_{j,i}^{k-2} D_j + \O ( (1 - e^{-k})^{-i} + (1 - e^{-k})^{i(k-2)} \Delta_0 + (1 - e^{-k})^{\frac{i(k-2)}{2}} \sqrt{D_0} \log N ), \\
			(c)\hspace{1mm} Z_i(x) =& (1 + o(1)) (k-1) (1-(1-e^{-k})^i) (1 - e^{-k})^{i(k-2)} D_0 = \Theta( (1 - e^{-k})^{i(k-2)} D_0), \\
			(d)\hspace{1mm} Z_i(x) =& (1+o(1)) (k-1)e^{- k} \sum_{0\leq j<i} q_{j,i}^{k-2} D_j.
			\end{align*}
			
			\item[$(\rm L4)_i$] $Y_i(x,y) \leq C_7 D_i^2 \log^2 N$ for any distinct $x,y \in U_{i-1}$.
			
			\item[$(\rm L5)_i$] $X_i(x,y) \leq C_8 D_i \log^2 N$ for any distinct $x,y \in U_{i-1}$.
		\end{itemize}

		\vspace{2mm}
		
		There exist constants $C_5, C_7, C_8, C_9$ depending on $k$ (but not on the index $i$ or the vertices $x, y$) such that if the statements $(\rm L1)_j$--$(\rm L5)_j$ hold for all $0 \leq j \leq i-1$, then with probability at least $1 - e^{-\Omega(\log^2 N)}$, the statements $(\rm L1)_i$--$(\rm L5)_i$ hold.
	\end{lemma}
	
	\textit{Remark.} Note that $(\rm L1)_i$ not only implies that $H_i$ is $(D_i \pm \Delta_i)$-regular but also the much stronger statement that for \emph{any} vertex $x \in V(H)$ the number of edges $e$ of $H$ incident to $x$ such that $e \setminus \{x\} \subseteq U_i = V(H_i)$ is $D_i \pm \Delta_i$.
	
	\vspace{2mm}
	
	The rest of this subsection is devoted to proving Lemma~\ref{lem:nibble} (Nibble lemma). Choose constants $C_5, C_7, C_8, C_9, C_{10}$ such that \begin{equation}
	\label{eq:mainhierarchy}
	  C_{10} := 2e^{-k} \text{ and } \frac{1}{C_8} \ll \frac{1}{C_5}, \frac{1}{C_7}, \frac{1}{C_9} \ll \frac{1}{k}.
	\end{equation}
	
	Since the degree of every vertex of $H$ is $D_0 \pm \Delta_0$, $W_0 = \emptyset$, $Z_0(x) = 0$, $X_0(x,y) = 0$, $U_{-1} = U_0 = V(H)$, $p_{-1}^* = 0$ and \eqref{eqn:boundy0} holds, it is clear that $(\rm L1)_0$--$(\rm L5)_0$ hold. Hereafter we fix the index $1 \leq i \leq \zeta$ and assume that $(\rm L1)_j$--$(\rm L5)_j$ hold for all $0 \leq j \leq i-1$. 
	
	\subsubsection{Estimating the basic parameters of the Nibble process}
	
	In the following proposition we estimate the parameters  $D_j, \Delta_j$ and $q_{t,j}$ defined in \eqref{qjidef}--\eqref{definingDeltai}.
	
	\begin{proposition}\label{claim:bound_delta}
		For $1 \leq j \leq \zeta$, define the following statements.
		
		\begin{itemize}
		    \item[$(\rm A0)_j$] If $\Delta_t = o(D_t)$ for all $0 \le t < j$, then $j \le C^* \log D_0$.

			\item[$(\rm A1)_j$] For all $0 \leq t \leq j$, $q_{t,j} = (1 \pm C_2 \frac{\Delta_{j-1} \log D_0}{D_{j-1}} ) (1 -  e^{- k})^{j-t}$. In particular, $q_{t,j} = (1+o(1))(1 - e^{- k})^{j-t}$. \label{state:a1}
			\item[$(\rm A2)_j$] $D_j = (1 \pm C_3 \frac{\Delta_{j-1} \log D_0}{D_{j-1}}) (1 - e^{- k})^{j(k-1)} D_0$. In particular, $D_j = (1 + o(1)) (1 - e^{- k})^{j(k-1)} D_0$. \label{state:a2}
			\item[$(\rm A3)_j$] We have $$D_j D_0^{-1} \Delta_0 + c_4 \sqrt{D_j} \log N \leq \Delta_j \leq D_j D_0^{-1} \Delta_0 + C_4 \sqrt{D_j} \log N.$$
			Moreover, $\Delta_{j} \log D_0 = o(D_{j})$.
		\end{itemize}
		
		\vspace{2mm}
		
		There exist constants $C^*, C_2, C_3, c_4$ and $C_4$ depending on $k$ (but not on the index $j$) such that the statements $(\rm A0)_j$-- $(\rm A3)_j$ hold for all $1 \le j \le i$.
	\end{proposition}
	\begin{proof}
	    Choose constants $C', C^*, C_2, C_3, c_4$ and $C_4$ such that 
	    \begin{equation}
	    \label{eq:constantHierarchy}
	        \frac{1}{C_3} \ll \frac{1}{C_2} \ll \frac{1}{C'} \ll c_4, \frac{1}{C_4}, \frac{1}{C^*} \ll \frac{1}{k}.
	    \end{equation}
	    Note that by ~\eqref{eq:***} and Proposition~\ref{claim:matching}(iii)\COMMENT{To apply Proposition~\ref{claim:matching} we need that $H_t$ is $(D_t\pm\Delta_t)$-regular for all $t < j$, but this fine since we are assuming that $(\rm L1)_t$ holds.} we have $1 \le D_j \le (1-\frac{e^{-k}}{2})^{j(k-1)} D_0$. Thus $(\rm A0)_j$ holds. 
	    
        We now prove $(\rm A1)_j$--$(\rm A3)_j$ by using induction on $j$. $(\rm A1)_1$--$(\rm A3)_1$ trivially follow from the definitions \eqref{qjidef}--\eqref{definingDeltai} and Proposition~\ref{claim:matching}(iii), so we assume $j \geq 2$. Let us fix $j \le i$. Assuming $(\rm A1)_s$--$(\rm A3)_s$ hold for $1 \leq s \leq j-1$, we will show that they hold for $s = j$. 
	    
	By $(\rm A2)_s$ and $(\rm A3)_s$ for $1 \leq s \leq j-1$, we have 
	\begin{align}\label{eqn:sumdj}
            \begin{split}
            \sum_{s=t}^{j-1} \frac{\Delta_s}{D_s} &\leq \sum_{s=0}^{j-1} \frac{\Delta_s}{D_s} \leq j\frac{\Delta_0}{D_0}  + 2 C_{4}\sum_{s=0}^{j-1} (1 - e^{- k})^{\frac{-s(k-1)}{2}} D_0^{-1/2} \log N \\
            & \leq j \frac{\Delta_0}{D_0} + \frac{2 C_{4} (1 - e^{- k})^{\frac{-(j-1)(k-1)}{2}} D_0^{-1/2} \log N}{1 - (1 - e^{- k})^{(k-1)/2}} \overset{\eqref{eq:constantHierarchy}}{\leq} j C' \frac{\Delta_{j-1}}{D_{j-1}} \\
            \overset{(\rm A0)_{j-1}}&{\leq} C^*C' \frac{\Delta_{j-1}}{D_{j-1}} \log D_0. 
            \end{split}
        \end{align}
        
        Since $q_{t,j} = \prod_{u=t}^{j-1}(1 - p_u^*)$, using Proposition~\ref{claim:matching}$(\rm iii)$\COMMENT{To apply Proposition~\ref{claim:matching}, we need that $H_s$ is $(D_s \pm \Delta_s)$-regular and $\Delta_s = o(D_s)$ for all $s \le j-1$. Both of these conditions are satisfied by our assumption that $(\rm L1)_s$ and $(\rm A3)_s$ hold for all $0 \leq s \leq j-1$.} and Proposition~\ref{lem:boundprod}\COMMENT{Explanation for the second equation below: Let $x_s := C_1 \Delta_s / D_s$. By $(\rm A3)_s$, we deduce that $x_s \leq KD_0^{\varepsilon-1} + C_4 \log N / \sqrt{D_s}.$ Let $t := \min(1-\varepsilon , 1/2)$. Then we have $x_s \leq (K + C_4 \log N) D_s^{-t}$. By the induction hypothesis, we know $$0.5(1 - e^{-k})^{-st(k-1)} D_0^{-t} \leq D_s^{-t} \leq 2 (1 - e^{-k})^{-st(k-1)} D_0^{-t}.$$ Thus $x_s \leq 2(K D_0^{-t} + C_4 \log N D_0^{-t}) (1 - e^{-k})^{-st(k-1)}$. Now we let $a := 2(K + C_4 \log N )D_0^{-t}$ and $b := (1 - e^{-k})^{-t(k-1)}$, then both $x_s$ and $y_s$ satisfy the conditions of Proposition~\ref{lem:boundprod}, since $y_s \leq 2(K + C_4 \log N)D_s^{-t} = o(1)$. Indeed
		$$\log N \cdot D_s^{-t} \leq \log N \cdot D_0^{-t\gamma} \leq \log^{1-\frac{8t(k-1)}{1-\varepsilon}} N \leq \log^{1 - 4(k-1)} N = o(1).$$}, for any $1 \le t \le j$, we have
		\begin{align*}
		q_{t,j} &= \left [ \prod_{s=t}^{j-1} \left ( 1 \pm C_1 \frac{\Delta_s}{D_s} \right ) \right ] (1 - e^{- k})^{j-t} 
		= \left ( 1 \pm 2C_1 \sum_{s=t}^{j-1} \frac{\Delta_s}{D_s} \right )  (1 - e^{- k})^{j-t} \\
		\overset{\eqref{eqn:sumdj}}&{=} \left ( 1 \pm 2 C_1 C^* C' \frac{\Delta_{j-1} \log D_0}{D_{j-1}} \right ) (1 - e^{- k})^{j-t} \overset{\eqref{eq:constantHierarchy}}{=} \left ( 1 \pm C_2 \frac{\Delta_{j-1} \log D_0}{D_{j-1}} \right ) (1 - e^{- k})^{j-t},
		\end{align*}
		so the first part of ~$(\rm A1)_j$ follows. Moreover, by $(\rm A3)_{j-1}$ we have
		\begin{equation}
		\label{eq:+}
		 \Delta_{j-1} \log D_0 = o(D_{j-1}),
		\end{equation}
	so $q_{t,j} = (1+o(1))(1 - e^{- k})^{j-t}$, proving ~$(\rm A1)_j$.

		Now using $(\rm A1)_j$, we have 
		\begin{align*}
		D_j \overset{\eqref{eq:***}}&{=} q_{0,j}^{k-1} D_0 = \left ( 1 \pm C_2 \frac{\Delta_{j-1} \log D_0}{D_{j-1}} \right )^{k-1} (1 - e^{- k})^{j(k-1)} D_0 \\
		&= \left (1 \pm (k-1) C_2 \frac{\Delta_{j-1} \log D_0}{D_{j-1}} + \O \left ( \frac{C_2^2 \Delta_{j-1}^2 \log^2 D_0 }{D_{j-1}^2} \right ) \right ) (1 - e^{- k})^{j(k-1)} D_0\\
		\overset{\eqref{eq:constantHierarchy},\eqref{eq:+}}&{=} \left (1 \pm C_3 \frac{\Delta_{j-1} \log D_0}{D_{j-1}} \right ) (1 - e^{- k})^{j(k-1)} D_0.
		\end{align*}
		Thus the first part of ~$(\rm A2)_j$ holds. The second part follows by using \eqref{eq:+}.		
		
		It remains to prove $(\rm A3)_j$. Note that
		\begin{align*}
		\Delta_j  \overset{\eqref{definingDeltai}}&{=} (1 - p_{j-1}^*)^{k-1} \Delta_{j-1} + \sqrt{D_{j-1}} \log N \\
		 \overset{(\rm A3)_{j-1}}&{\leq} (1 - p_{j-1}^*)^{k-1} D_{j-1} D_0^{-1} \Delta_0 + C_4 (1 - p_{j-1}^*)^{k-1} \sqrt{D_{j-1}} \log N + \sqrt{D_{j-1}} \log N\\
		 \overset{\eqref{Didef}}&{\leq} D_{j} D_0^{-1} \Delta_0 + [(1 - p_{j-1}^*)^{\frac{k-1}{2}} C_4 + (1 - p_{j-1}^*)^{-\frac{k-1}{2}}] \sqrt{D_{j}} \log N\\
		\overset{\eqref{eq:constantHierarchy}}&{\leq} D_{j} D_0^{-1} \Delta_0 + C_4 \sqrt{D_j} \log N.
		\end{align*}

		Similarly, we have\COMMENT{\begin{align*}
		\Delta_j & = (1 - p_{j-1}^*)^{k-1} \Delta_{j-1} + \sqrt{D_{j-1}} \log N \\
		& \geq (1 - p_{j-1}^*)^{k-1} D_{j-1} D_0^{-1} \Delta_0 + c_4 (1 - p_{j-1}^*)^{k-1} \sqrt{D_{j-1}} \log N + \sqrt{D_{j-1}} \log N\\
		& \geq D_{j} D_0^{-1} \Delta_0 + [(1 - p_{j-1}^*)^{\frac{k-1}{2}} c_4 + 1] \sqrt{D_{j}} \log N\\
		& \geq D_{j} D_0^{-1} \Delta_0 + c_4 \sqrt{D_j} \log N
		\end{align*}} $\Delta_j \geq D_{j} D_0^{-1} \Delta_0 + c_4 \sqrt{D_j} \log N$. This completes the proof of the first part of $(\rm A3)_j$. Now combining this with the fact that $\Delta_0 \leq K D_0^\varepsilon$, we have 
		$$\Delta_j \leq K D_j^{\varepsilon} + \Theta (\sqrt{D_j} \log N) = \O(D_j^{\max(\varepsilon , 1/2)} \log N) \overset{\eqref{eq:*}}{=} \O(D_j^{1 - \min(\frac{7-7\varepsilon}{8} , \frac{3+\varepsilon}{8})}),$$\COMMENT{$D_j^{\max(\varepsilon , 1/2)} \log N \overset{\eqref{eq:settingupparameters}}{\leq} D_j^{\max(\varepsilon , 1/2)} D_0^{\frac{\gamma(1-\varepsilon)}{8}}\overset{\eqref{eq:*}}{\le} D_j^{\max(\varepsilon , 1/2)} D_j^{\frac{(1-\varepsilon)}{8}} \le D_j^{\max(\frac{1+7\varepsilon}{8} , \frac{5-\varepsilon}{8})} = D_j^{1 - \min(\frac{7-7\varepsilon}{8} , \frac{3+\varepsilon}{8})}.$} since $D_0 \geq \log^{8/\gamma(1-\varepsilon)} N$. Thus $\Delta_{j} \log D_0 = o(D_{j})$ by $(\rm A0)_j$ and $(\rm A2)_j$, proving $(\rm A3)_j$ and completing the proof of the proposition.
	\end{proof}

Recall that $C_5 = C_5(k)$ was chosen in \eqref{eq:mainhierarchy}.
	
	\begin{proposition}\label{claim:waste}
		For every $v \in U_{i-1}$, 
		\begin{equation}\label{eqn:wasteprob}
		p_{W_i}(v) \leq C_5 \frac{\Delta_{i-1}}{D_{i-1}}.
		\end{equation}
		Thus for every $S \in \mathcal{V}$, we have
		\begin{equation}\label{eqn:wasteexpectation}
		\mathbb{E}[|S \cap W_i|] \leq C_5 \frac{\Delta_{i-1}}{D_{i-1}} |S \cap U_{i-1}|.    
		\end{equation}
	\end{proposition}
	\begin{proof}
		By Proposition~\ref{claim:matching}, we have
		\begin{equation*}
		p_{W_i}(v) \overset{\eqref{eqn:pwi}}{=} \frac{p_{i-1}^* - p_{M_i}(v)}{1 - p_{M_i}(v)} \leq \frac{2 C_1 \Delta_{i-1} D_{i-1}^{-1}  e^{- k}}{1 - e^{- k} \pm C_1 \Delta_{i-1} D_{i-1}^{-1}  e^{- k}}.
		\end{equation*}
		Since $\Delta_{i-1} = o(D_{i-1})$ (by Proposition~\ref{claim:bound_delta} $(\rm A3)_{i-1}$) we have $1 - e^{- k} \pm C_1 \Delta_{i-1} D_{i-1}^{-1}  e^{- k} \ge (1 - e^{- k})/2$. Thus $p_{W_i}(v) \leq C_5 \frac{\Delta_{i-1}}{D_{i-1}}$, as desired.
	\end{proof}
	
	In the next subsection we will present two results that play a key role in the proof of the Nibble lemma.
	
    \subsubsection{Almost independence of events}
    The following proposition is a generalisation of~\cite[Claim 1]{AKS1997}.
    It is heavily used in the proof of the Nibble lemma to show that certain events of whether a given set of vertices are in $U_i$, $V(M_i)$ are `nearly' independent. 
    The key difference between Claim 1 of \cite{AKS1997} and the lemma below is that here we remove the assumption that the vertices are contained in a hyperedge, which is essential for our applications. Recall that $C_9 = C_9(k)$ was chosen in \eqref{eq:mainhierarchy}.
	
    \begin{proposition}[Almost independence]\label{lem:almost_indep}
        Let $0 \leq n_1 , n_2 \leq 3k$ be integers, and let $x_1 , \dots , x_{n_1},$ $y_1,\dots , y_{n_2} \in U_{i-1}$ be distinct vertices. Then
	\begin{align*}
        &\mathbb{P} \left (\bigwedge_{j=1}^{n_1}(x_j \in U_i) \:\wedge\:\bigwedge_{j=1}^{n_2} (y_j \in V(M_i)) \right) = \left (1 \pm \frac{C_9}{D_{i-1}} \right ) \prod_{j=1}^{n_1} \mathbb{P}(x_j \in U_i) \prod_{j=1}^{n_2} \mathbb{P}(y_j \in V(M_i)).
	\end{align*}
    \end{proposition}

    The proof of Proposition~\ref{lem:almost_indep} is similar to the proof of Claim 1 in \cite{AKS1997}. Thus, we omit it here. Full details are given in Appendix~\ref{appendix:almostindependence}.
	
	\vspace{2mm}

	Recall that for the $i$th stage of the Nibble process, our probability space is generated from mutually independent events of the form $F \in B_i$ and $w \in W_i$. To prove concentration for various random variables in our analysis of the Nibble lemma, we will use Lemma \ref{cor:martingale} (Martingale inequality) where `Paul' determines a random variable by asking whether each choice of the form $F \in B_i$ or $w \in W_i$ is Yes or No. The following observation (which follows from $(\rm L1)_{i-1}$) bounds the number of such queries Paul needs in order to determine whether a vertex of $H_{i-1}$ is in $U_i$.
	
	\begin{proposition}
		\label{PaulsQueries}
		Let $y \in U_{i-1} = V(H_{i-1})$. Then Paul needs at most $\O(D_{i-1})$ queries of whether an edge is in $B_i$ to determine whether $y$ is in $V(M_i)$ and one query to determine whether $y$ is in $W_i$.
		
		Hence, to determine whether $y$ is in $U_i$, Paul needs at most $\O(D_{i-1})$ queries of whether an edge is in $B_i$, and one query of whether a vertex is in $W_i$.
	\end{proposition}
	\COMMENT{\begin{proof}
		By $(\rm L1)_{i-1}$, the degree of any vertex in $H_{i-1}$ is $\O(D_{i-1})$. To determine whether $y \in U_{i-1}$ is in $V(M_i)$, Paul asks $\O(D_{i-1})$ queries of whether each of the edges incident to $y$ in $E(H_{i-1})$ is in $B_i$.  If  at least two (or none) of these edges belong to $B_i$, then clearly $y \not \in V(M_i)$. Otherwise, exactly one of these edges $e'$ is in $B_i$, and to determine whether $e'$ is in $E(M_i)$ it suffices for Paul to ask $\O(D_{i-1})$ queries of whether the edges of $H_{i-1}$ sharing a vertex with $e'$ are in $B_i$ (where the number of such edges is again at most $\O(D_{i-1})$ by $(\rm L1)_{i-1}$). Thus in total, to determine whether $y$ is in $V(M_i)$ Paul needs at most $\O(D_{i-1})$ queries of whether an edge is in $B_i$. To determine whether $y \in W_i$, Paul only needs one query of whether a vertex is in $W_i$. Finally note that to determine whether $y \in U_i$, it suffices to determine whether $y \in V(M_i)$ and whether $y \in W_i$. This proves the lemma. 
	\end{proof}}

    \subsubsection{Proof of Lemma~\ref{lem:nibble} (Nibble lemma)}
	\subsubsection*{\textbf{Step 1.} Proving that $(\rm L1)_{i}$ holds with high probability}
	
	Note that $D_i(x)$ is the number of instances $e$ of $D_{i-1}(x)$ such that $e \setminus \{x\} \subseteq U_i$. Hence by \eqref{eqn:pUi} and  Proposition~\ref{lem:almost_indep}, 
	
    \begin{align}\label{eqn:exp_deg}
        \begin{split}    
            \mathbb{E}[D_i(x)] &= (1 \pm C_9 D_{i-1}^{-1}) (1 - p_{i-1}^*)^{k-1}D_{i-1}(x) \\
            \overset{(\rm L1)_{i-1}}&{=}  (1 \pm C_9 D_{i-1}^{-1}) (1 - p_{i-1}^*)^{k-1} (D_{i-1} \pm \Delta_{i-1}) \\ 
            &= (1 - p_{i-1}^*)^{k-1} D_{i-1} \pm (1 - p_{i-1}^*)^{k-1} (\Delta_{i-1} + 2C_9).
        \end{split}
    \end{align}
    Note that in the last equality we used $\Delta_{i-1} = o(D_{i-1})$ (which holds by Proposition~\ref{claim:bound_delta}).
    \begin{claim}\label{claim:L1}
        With probability $1 - e^{-\Omega(\log^2 N)}$, $(\rm L1)_{i}$ holds.
    \end{claim}
    \begin{subproof}
        Let $x \in V(H)$. By $(\rm L1)_{i-1}$ and $(\rm A3)_{i-1}$ of Proposition~\ref{claim:bound_delta}, note that 
	\begin{equation*}
            0.5 D_{i-1} \leq D_{i-1}(x) = D_{i-1} \pm \Delta_{i-1} \leq 2 D_{i-1}.
	\end{equation*}
 
        We will prove the claim by applying the Martingale inequality (Lemma~\ref{cor:martingale}). Indeed, to determine the random variable $D_i(x)$, for each instance $e$ of $D_{i-1}(x)$, Paul has to determine whether each of the $k-1$ vertices of $e \setminus \{x\}$ is in $U_i$. By Proposition ~\ref{PaulsQueries}, this can be done by asking at most $\O(D_{i-1} \cdot D_{i-1}(x)) = \O(D_{i-1}^2)$ queries of whether an edge is in $B_i$ and $\O(D_{i-1}(x)) = \O(D_{i-1})$ queries of whether a vertex is in $W_i$. The former is yes with probability $1/D_{i-1}$ and the latter is yes with probability at most $\O(\Delta_{i-1}/D_{i-1})$ by Proposition~\ref{claim:waste}. Moreover, we claim that changing the response to a query (keeping all others the same), changes $D_i(x)$ by at most $\O_k(1)$. Indeed, changing whether an edge $e$ is in $B_i$ changes $V(M_i)$ (and thus $U_i$) by at most $\O_k(1)$ vertices, as there are at most $k$ pairwise disjoint edges in $H_{i-1}$ that intersect $e$. Changing whether a vertex is in $W_i$ changes $U_i$ by at most one vertex. Since $H_{i-1}$ is simple, it follows that $D_i(x)$ changes by at most $\O_k(1)$. 
	
        Hence, we can choose $\sigma = \Theta(\sqrt{D_{i-1}})$ in Lemma~\ref{cor:martingale}. 
        By choosing $\lambda = \Theta(\log N)$ with an appropriate multiplicative constant and since $\sqrt{D_{i-1}} = \omega(\log N)$ by ~\eqref{eq:*}, we obtain that with probability at least $1 - e^{-\Omega_{k}(\log^2 N)}$, 
	\begin{align*}
            D_i(x) &= \mathbb{E}[D_i(x)] \pm \frac{\sqrt{D_{i-1}} \log N}{2}\\ 
            \overset{\eqref{eqn:exp_deg}}&{=} (1 - p_{i-1}^*)^{k-1} D_{i-1} \pm \left ( (1 - p_{i-1}^*)^{k-1} \Delta_{i-1} + \sqrt{D_{i-1}} \log N \right ) \overset{~\eqref{Didef},~\eqref{definingDeltai}}{=} D_i \pm \Delta_i.
        \end{align*}
	By the union bound, with probability at least
        $1 - |V(H)|e^{-\Omega_{k}(\log^2 N)} = 1 - e^{-\Omega_{k}(\log^2 N)},$
        we have $D_i(x) = D_i \pm \Delta_i$ for all $x \in V(H)$. This proves $(\rm L1)_{i}$.
    \end{subproof}
    
    \subsubsection*{\textbf{Step 2.} Proving that $(\rm L2)_{i}$ holds with high probability}
	
	\begin{claim}\label{claim:L2}
		With probability at least $1 - e^{-\Omega( \log^2 N )}$, $(\rm L2)_{i}$ holds.
	\end{claim}
	\begin{subproof}
	    Let us fix an $S \in \mathcal{V}$. By~\eqref{eqn:pUi}, we have
		\begin{align}\label{eqn:expui1}
		\mathbb{E}|S \cap U_i| = |S \cap U_{i-1}| (1 - p_{i-1}^*).
		\end{align}
		For any $1 \leq j \leq i-1$, we assumed $(\rm L2)_j$ holds, so we have 
		\begin{align}
		    |S \cap U_{j}| &= |S \cap U_{j-1}| \cdot \left ( 1 \pm |S \cap U_{j-1}|^{-1/4} \right ) (1 - p_{j-1}^*) {\rm ,\:\:and} \label{eqn:first_l2j}\\
		    |S \cap U_j| &= (1+o(1))(1 - e^{-k})^j |S| \overset{(\rm A2)_j}{=} (1+o(1)) |S| (D_j D_0^{-1})^{\frac{1}{k-1}} \geq \Omega(D_0^{\frac{1}{2} - \frac{1}{k-1}}), \label{eqn:second_l2j}
		\end{align}
	since $|S| \geq \sqrt{D_0} \log N$ by the assumption of Theorem~\ref{thm:mainsec3} and $D_j \geq D_0^{\gamma} \ge 1$ by \eqref{eq:Domega}.
		Hence by recursively applying \eqref{eqn:first_l2j} and then using 
		Proposition~\ref{lem:boundprod} and Proposition~\ref{claim:bound_delta}$(\rm A1)_{i-1}$, we obtain\COMMENT{In the summation below, it is worth noting that $|S \cap U_j|$ behaves like a geometric series by~\eqref{eqn:second_l2j}, so the summation is bounded by the last term regardless of the bound on $\zeta$.}
		\begin{equation}\label{eqn:ui-1}
		|S \cap U_{i-1}| \overset{\eqref{eq:***}}{=} \left (1 \pm 2 \sum_{j=0}^{i-2} |S \cap U_j|^{-1/4} \right ) q_{0,i-1} |S| \overset{\eqref{eqn:second_l2j}}{=} (1+o(1))(1 - e^{- k})^{i-1}|S|.
		\end{equation}
		By Proposition~\ref{claim:bound_delta} we have $D_{i-1} = (1+o(1))(1-e^{-k})^{(i-1)(k-1)}D_0$. Combining this with ~\eqref{eqn:ui-1} and our assumption that $|S| \geq \sqrt{D_0} \log N$, we obtain that
		\begin{equation*}
		\frac{|S \cap U_{i-1}| \log N}{\sqrt{D_{i-1}}} = (1+o(1)) (1 - e^{-k})^{(i-1)\left ( 1 - \frac{k-1}{2} \right )} \frac{|S| \log N}{\sqrt{D_0}} = \Omega(\log^2 N).
		\end{equation*}
		Hence by Proposition~\ref{claim:waste} and Lemma~\ref{lem:chernoff} (Chernoff-Hoeffding inequality), with probability at least $1 - e^{-\Omega(\log^2 N)}$.\COMMENT{If $\frac{|S \cap U_{i-1}| \log N}{\sqrt{D_{i-1}}} \geq 7 \mathbb{E}|S \cap W_i|$, then by Chernoff we have the error probability $\exp \left (- \frac{|S \cap U_{i-1}| \log N}{\sqrt{D_{i-1}}} \right ) \leq \exp(-\log^2 N)$. Otherwise, by Chernoff we have the error probability $\exp(-\mathbb{E}|S \cap W_i|) \leq \exp \left (-\Omega\left(\frac{|S \cap U_{i-1}| \log N}{\sqrt{D_{i-1}}}\right) \right ) \leq \exp(-\Omega(\log^2 N))$. Note that we added the extra coefficient 2 in front of $\mathbb{E}|S \cap W_i|$ to be able to use Chernoff for either case.}
		\begin{equation*}
		|S \cap W_i| \leq 2\mathbb{E}|S \cap W_i| + \frac{|S \cap U_{i-1}| \log N}{\sqrt{D_{i-1}}} \overset{\eqref{eqn:wasteexpectation}}{\leq} 2C_5 \frac{\Delta_{i-1}}{D_{i-1}} |S \cap U_{i-1}| + \frac{|S \cap U_{i-1}| \log N}{\sqrt{D_{i-1}}}.
		\end{equation*}
	
	Since $|\mathcal{V}| \le \exp(\log^{3/2}N)$ by the assumption of Theorem~\ref{thm:mainsec3}, we can take the union bound over all $S \in \mathcal V$ to deduce that (\rm a) of $(\rm L2)_{i}$ holds. Now note that\COMMENT{We have changed the 5 to 8 in the lower bound on $D_0$ in ~\eqref{eq:settingupparameters} in order to satisfy this inequality.}
		\begin{align}
		\label{ratioUD}
        |S \cap U_{i-1}|^{1/4} \overset{\eqref{eq:settingupparameters}, \eqref{eqn:second_l2j}}{\geq} \log N.
		\end{align}
		
		Now we will use the Martingale inequality (Lemma~\ref{cor:martingale}) to prove a concentration inequality for the random variable $|S \cap U_i|$. By Proposition ~\ref{PaulsQueries}, $|S \cap U_i|$ is determined by at most $\O(|S \cap U_{i-1}|D_{i-1})$  queries of whether an edge $X \in B_i$ and at most $|S \cap U_{i-1}|$ queries of whether a vertex $w \in W_i$. The first one is yes with probability $1/D_{i-1}$ and the second one is yes with probability at most $\O(\Delta_{i-1}/D_{i-1})$ (by Proposition ~\ref{claim:waste}). Moreover, changing whether $X \in B_i$ changes $|S \cap U_i|$ by at most $\O(k^2) = \O_k(1)$, as there are at most $k$ pairwise disjoint edges in $H_{i-1}$ that intersect $X$ and changing whether $w \in W_i$, of course, only changes $|S \cap U_i|$ by at most one. Thus we can apply Lemma~\ref{cor:martingale} with $\sigma = \Theta(|S \cap U_{i-1}|^{1/2})$ and $\lambda = \Theta(|S \cap U_{i-1}|^{1/4})$ (with an appropriately chosen multiplicative constant factor), to obtain that with probability at least $1 - e^{-\Omega(|S \cap U_{i-1}|^{1/2})} \overset{\eqref{ratioUD}}{=} 1 - e^{-\Omega(\log^2 N)}$,
		
		\begin{align*}
		|S \cap U_i| = \mathbb{E}|S \cap U_i| \pm (1-p_{i-1}^*)|S \cap U_{i-1}|^{3/4} \overset{\eqref{eqn:expui1}}{=} \left ( 1 \pm |S \cap U_{i-1}|^{-1/4} \right ) (1 - p_{i-1}^*)|S \cap U_{i-1}|.
		\end{align*}
	Hence by~\eqref{eqn:ui-1} and Proposition~\ref{claim:matching}$(\rm iii)$, with probability at least $1 - e^{-\Omega(\log^2 N)}$, we have
		\begin{align*}
		    |S \cap U_i| =  (1+o(1))(1 - e^{-k})^{i} |S|.
		\end{align*}
		Hence by taking the union bound over all $S \in \mathcal{V}$, we deduce that (b) and (c) of $(\rm L2)_i$ hold as well. 
	\end{subproof}

	\subsubsection*{\textbf{Step 3.} Proving $(\rm L3)_{i}$ holds with high probability}

	For any $x \in U_{i-1}$, note that $Z_i(x)$ is the sum of 
	\begin{itemize}
		\item the number of instances $e$ of $Z_{i-1}(x)$ such that the $k-2$ vertices of $e \setminus\{x\} \cap U_{i-1}$ are in $U_i$, and
		\item the number of instances $e$ of $D_{i-1}(x)$ such that one vertex of $e \setminus \{ x \}$ is in $V(M_i)$ and the remaining $k-2$ vertices of $e \setminus \{ x \}$ are in $U_i$.
	\end{itemize}
	Thus we have the following recurrence relation by Proposition~\ref{claim:matching}$(\rm ii)$, \eqref{eqn:pUi} and Proposition~\ref{lem:almost_indep}.
	\begin{align}\label{eqn:expei}
	\mathbb{E} Z_i(x) &= \left ( 1 \pm \frac{C_9}{D_{i-1}} \right ) (1 - p_{i-1}^*)^{k-2} \left ( Z_{i-1}(x) + \left (1 \pm \frac{C_1 \Delta_{i-1}}{D_{i-1}} \right ) (k-1)e^{-k} D_{i-1}(x) \right ).
	\end{align}
	
	\begin{claim}\label{claim:eq_e}
		With probability at least $1 - e^{-\Omega(\log^2 N)}$, for any $x \in U_{i-1}$,
		\begin{equation*}
		Z_i(x) = \mathbb{E}Z_i(x) \pm \sqrt{Z_{i-1}(x) + D_{i-1}} \log N.
		\end{equation*}
	\end{claim}
	\begin{subproof}
		We will use the Martingale inequality (Lemma ~\ref{cor:martingale}) to prove the claim. 
	 
		To determine the random variable $Z_i(x)$, Paul must determine \begin{itemize}
			\item[(a)] for each instance $e$ of $Z_{i-1}(x)$, whether each of the $k-2$ vertices of $e \setminus \{x\} \cap U_{i-1}$ is in $U_i$ (or equivalently whether each of those vertices is in $V(M_i)$ or $W_i$), and
			\item[(b)] for each instance $e$ of $D_{i-1}(x)$, whether each of the $k-1$ vertices of $e \setminus \{x \}$ is in $V(M_i)$ or $W_i$.
		\end{itemize}
		
		Since the number of vertices involved in (a) and (b) is at most $\O(Z_{i-1}(x) + D_{i-1}(x))$, it follows from Proposition ~\ref{PaulsQueries} that Paul can determine $Z_i(x)$ by asking $\O((Z_{i-1}(x) + D_{i-1}(x))D_{i-1})$ queries of whether an edge is in $B_i$ and $\O(Z_{i-1}(x) + D_{i-1}(x))$ queries of whether a vertex is in $W_i$. The former is yes with probability $1/D_{i-1}$ and the latter is yes with probability at most $\O(\Delta_{i-1}/D_{i-1})$ by Proposition ~\ref{claim:waste}. Moreover, note that changing the answer to any of the queries (keeping all others the same) changes $Z_i(x)$ by at most $\O_k(1)$.  Indeed, changing whether $e \in B_i$ changes $V(M_i)$ (and $U_i$) by at most $\O(k^2) = \O_k(1)$ vertices as there are at most $k$ pairwise disjoint edges in $H_{i-1}$ that intersect $e$, and changing whether $w \in W_i$ changes $U_i$ by at most one vertex. Hence, as $H_{i-1}$ is simple, at most $\O_k(1)$ instances of $Z_{i-1}(x)$ and $D_{i-1}(x)$ are affected by such a change, implying that $Z_i(x)$ changes by at most $\O_k(1)$, as desired.
		
		Thus we can apply Lemma ~\ref{cor:martingale} with $\sigma = \Theta(\sqrt{Z_{i-1}(x) + D_{i-1}(x)}) \overset{(\rm L1)_{i-1}}{=} \Theta(\sqrt{Z_{i-1}(x) + D_{i-1}})$ (thus $\sigma = \omega(\log N)$ by \eqref{eq:*}) and $\lambda = \Theta(\log N)$ with an appropriately chosen multiplicative constant, to obtain that with probability at least $1 - e^{-\Omega(\log^2 N))}$, 
		$$  Z_i(x) = \mathbb{E}Z_i(x) \pm \sqrt{Z_{i-1}(x) + D_{i-1}} \log N. $$
		Claim~\ref{claim:eq_e} now follows by taking the union bound. 
	
	\end{subproof}

    \begin{claim}\label{claim:bounde}
        With probability $1 - e^{-\Omega(\log^2 N)}$, $(\rm L3)_i$ holds.
    \end{claim}
    \begin{subproof}
    Let $x \in U_{i-1}$. Assuming $(\rm L1)_j$--$(\rm L5)_j$ hold for all $0 \le j \le i-1$, we have to show that $(\rm L3)_i$ holds with high probability. By Claim~\ref{claim:eq_e} and $(\rm L1)_{i-1}$, with probability $1 - e^{-\Omega(\log^2 N))}$,
    \begin{align*}
        Z_i(x) & = & & \mathbb{E}Z_i(x) \pm \sqrt{Z_{i-1}(x) + D_{i-1}} \log N \\
        \overset{\eqref{eqn:expei}}&{=}& & \left ( 1 \pm \frac{C_9}{D_{i-1}} \right ) (1 - p_{i-1}^*)^{k-2} \left ( Z_{i-1}(x) + \left (1 \pm \frac{C_1 \Delta_{i-1}}{D_{i-1}} \right )\left (1 \pm \frac{\Delta_{i-1}}{D_{i-1}} \right ) (k-1)e^{-k} D_{i-1} \right ) \\
        & & & \pm \sqrt{Z_{i-1}(x) + D_{i-1}} \log N.
    \end{align*}
    
    Together with Proposition~\ref{claim:bound_delta}$(\rm A3)_{i-1}$, this implies that $Z_i(x) = A_1 + A_2 + A_3 + A_4 + A_5$, where
    \begin{align}
        A_1&:= (1 - p_{i-1}^*)^{k-2} (Z_{i-1}(x) + (k-1)e^{-k} D_{i-1}), \label{eqn:ei_prem}\\
        A_2&:= \pm \frac{C_9}{D_{i-1}} (1 - p_{i-1}^*)^{k-2} Z_{i-1}(x), \label{eqn:ei_second} \\
        A_3&:= \pm \frac{C_9}{D_{i-1}} (1 - p_{i-1}^*)^{k-2} \left (1 \pm \frac{C_1 \Delta_{i-1}}{D_{i-1}} \right ) \left (1 \pm \frac{\Delta_{i-1}}{D_{i-1}} \right )(k-1) e^{-k} D_{i-1} \label{eqn:ei_third}\\ &= \pm C_9 (1+o(1)) (k-1) e^{-k} (1 - p_{i-1}^*)^{k-2}, \nonumber\\
        A_4&:= \pm (1 - p_{i-1}^*)^{k-2} \left ( \frac{C_1 \Delta_{i-1}}{D_{i-1}} + \frac{\Delta_{i-1}}{D_{i-1}} + \frac{C_1 \Delta_{i-1}^2}{D_{i-1}^2} \right ) (k-1) e^{-k} D_{i-1} \label{eqn:ei_fourth}\\ &= \pm (C_1 + 1 + o(1)) (k-1) e^{-k} (1 - p_{i-1}^*)^{k-2} \Delta_{i-1}, \nonumber\\
        A_5&:= \pm \sqrt{Z_{i-1}(x) + D_{i-1}} \log N. \label{eqn:ei_fifth}
    \end{align}
		
    By $(\rm L3)_{i-1}$ and \eqref{qjidef}, we have $A_1 = A_1^1+A_1^2$, where
    \begin{align}
        A_1^1 &:= (k-1) e^{-k} \sum_{0\leq j<i}q_{j,i}^{k-2} D_j, \label{eqn:ei_mainterm} \\
        A_1^2 &:= \pm \sum_{0\leq j<i-1} \left [ q_{j,i}^{k-2}  \left ( \frac{C_9 Z_j(x)}{D_j} + (C_1+3)(k-1)e^{-k} \Delta_j \right ) + q_{j+1,i}^{k-2} \sqrt{Z_j(x) + D_j} \log N \right ]. \label{eqn:ei_sixth}
    \end{align}
    Observe that $A_3+A_4 = \pm (C_1+3)(k-1) e^{-k} (1 - p_{i-1}^*)^{k-2} \Delta_{i-1}$ by Proposition~\ref{claim:bound_delta} $(\rm A3)_{i-1}$. 
    Thus, 	
    \begin{equation*}
        A_1^2+A_2+ A_3+A_4+A_5=\pm \sum_{0\leq j<i} \left [ q_{j,i}^{k-2}  \left ( \frac{C_9 Z_j(x)}{D_j} + (C_1+3)(k-1)e^{-k} \Delta_j \right ) + q_{j+1,i}^{k-2} \sqrt{Z_j(x) + D_j} \log N \right ].
    \end{equation*}		
    Combining this with \eqref{eqn:ei_mainterm}, and using that $Z_i(x) = A_1^1+ A_1^2 + A_2 + A_3 + A_4 + A_5$ we get
    \begin{align}\label{eqn:ei_comp}
        \begin{split}
            Z_i(x) &= (k-1)e^{- k} \sum_{0\leq j<i} q_{j,i}^{k-2} D_j \\ 
            &\pm \sum_{0\leq j<i} \left [ q_{j,i}^{k-2}  \left ( \frac{C_9 Z_j(x)}{D_j} + (C_1+3)(k-1)e^{-k} \Delta_j \right ) + q_{j+1,i}^{k-2} \sqrt{Z_j(x) + D_j} \log N \right ]. 
        \end{split}
    \end{align}
    
    This proves (\rm a). For convenience let us redefine $B_1 := A_1^1$ and $B_2 := A_1^2 + A_2 + A_3 + A_4 + A_5$. Then by \eqref{eqn:ei_comp}, $Z_i(x) =  B_1 + B_2$.
    
    It remains to prove (\rm b)--(\rm d) by estimating $B_1$ and $B_2$. First let us prove (\rm b) by showing that $B_2 = \O ((1 - e^{-k})^{-i} + (1 - e^{-k})^{i(k-2)} \Delta_0 + (1 - e^{-k})^{\frac{i(k-2)}{2}} \sqrt{D_0} \log N).$ 
    Since this is trivial when $i=0$, we may assume that $i \geq 1$. 

    For $0 \leq j \leq i-1$, by $(\rm c)$ of $(\rm L3)_j$, we have
    \begin{equation}\label{eqn:ejbound}
        Z_j(x) \leq 2(k-1) (1 - e^{-k})^{j(k-2)}D_0.
    \end{equation}
    Therefore, the summand
    $$q_{j,i}^{k-2}  \left ( \frac{C_9 Z_j(x)}{D_j} + (C_1+3)(k-1)e^{-k} \Delta_j \right ) + q_{j+1,i}^{k-2} \sqrt{Z_j(x) + D_j} \log N$$
    of $B_2$ is the sum of the following terms (which are estimated below using \eqref{eqn:ejbound} and Proposition~\ref{claim:bound_delta}).
    \begin{align*}
        q_{j,i}^{k-2} \frac{C_9 Z_j(x)}{D_j} & \leq & & 4C_9 (k-1) (1 - e^{-k})^{i(k-2) - j(k-1)},\\
        q_{j,i}^{k-2} (C_1+3)(k-1)e^{-k} \Delta_j & \leq & & (C_1+3)(k-1)e^{-k} (1+o(1)) (1 - e^{-k})^{(i-j)(k-2)}\\
        & & & \times (1+o(1)) \left [ (1 - e^{-k})^{j(k-1)}\Delta_0 + C_4 (1 - e^{-k})^{\frac{j(k-1)}{2}} \sqrt{D_0} \log N \right ]\\
        & \leq & & 2(C_1+3)(k-1)e^{-k} (1 - e^{-k})^{i(k-2)}\\
        & & & \times \left [ (1 - e^{-k})^j \Delta_0 + (1 - e^{-k})^{-\frac{j(k-3)}{2}} C_4 \sqrt{D_0} \log N \right ],\\
        q_{j+1,i}^{k-2} \sqrt{Z_j(x) + D_j} \log N & \leq & & (1+o(1)) \frac{(1-e^{-k})^{(i-j)(k-2)}}{(1 - e^{-k})^{k-2}} \sqrt{2k(1 - e^{-k})^{j(k-2)} D_0} \log N \\
        & \leq & & \frac{2 \sqrt{2k} \sqrt{D_0} \log N}{(1 - e^{-k})^{k-2}} (1 - e^{-k})^{i(k-2) - \frac{j(k-2)}{2}}.
    \end{align*}
    \COMMENT{In the estimation of the last term we used $Z_j(x) + D_j \leq 2 k(1 - e^{-k})^{j(k-2)} D_0$ by~\eqref{eqn:ejbound} and Proposition~\ref{claim:bound_delta}.}
    Hence $B_2$ is the sum of the following terms: 
    \begin{align*}
        \sum_{j<i} q_{j,i}^{k-2} \frac{C_9 Z_j(x)}{D_j} & \leq & & 4C_9 (k-1) \frac{(1 - e^{-k})^{-i} }{(1 - e^{-k})^{-(k-1)} - 1},\\
        \sum_{j<i} q_{j,i}^{k-2}(C_1+3)(k-1) e^{-k} \Delta_j & \leq & & 2(C_1+3)(k-1)\\
        & & & \times \left [ ( 1 - e^{-k})^{i(k-2)}\Delta_0 + C_4 e^{-k} \sqrt{D_0} \log N \frac{(1 - e^{-k})^{\frac{i(k-1)}{2}}}{(1 - e^{-k})^{-\frac{k-3}{2}} - 1} \right ],\\
        \sum_{j<i} q_{j+1,i}^{k-2} \sqrt{Z_j(x) + D_j} \log N & \leq & & \frac{2 \sqrt{2k} \sqrt{D_0} \log N}{(1 - e^{-k})^{k-2}} \frac{(1 - e^{-k})^{\frac{i(k-2)}{2}}}{(1 - e^{-k})^{-\frac{k-2}{2}} - 1}.
    \end{align*}
    Hence,
    \begin{align}\label{errorterm}
        B_2 \le \O \left ( (1 - e^{-k})^{-i} + (1 - e^{-k})^{i(k-2)} \Delta_0 + (1 - e^{-k})^{\frac{i(k-2)}{2}} \sqrt{D_0} \log N \right )
    \end{align}
    proving (b). To prove (c) note that by Proposition~\ref{claim:bound_delta},
    \begin{align}
        B_1 = (k-1)e^{- k} \sum_{0\leq j<i} q_{j,i}^{k-2} D_j &= (1+o(1)) (k-1) e^{-k} \sum_{j<i} (1 - e^{-k})^{(i-j)(k-2)} (1 - e^{-k})^{j(k-1)}D_0 \nonumber\\
        &= (1+o(1)) (k-1) (1 - (1 - e^{-k})^{i} )(1 - e^{-k})^{i(k-2)} D_0. \label{eqn:ei_main}
    \end{align}
    Then it follows from ~\eqref{eqn:ei_main} and ~\eqref{errorterm} that $B_1$ is asymptotically larger\COMMENT{Note that $(1 - e^{-k})^{i(k-2)}D_0 \gg (1 - e^{-k})^{\frac{i(k-2)}{2}} \sqrt{D_0} \log N$ is equivalent to $D_i \gg (1 - e^{-k})^{i/2} \sqrt{D_i} \log N$, by multiplying both sides with $(1 - e^{-k})^i$. This is in turn equivalent to $D_i \gg (1 - e^{-k})^{i}  \log^2 N$, which holds if $i < \zeta$ since $D_i \ge D_0^{\gamma} \gg \log^2 N$ by ~\eqref{eq:*} and our assumption on $D_0$. If $i = \zeta$ then we have $D_{\zeta} = \Theta(D_0^{\gamma}) \gg \log^2 N$ by $(\rm A3)_{i-1}$ and the sentence after ~\eqref{eq:*}. Also note that $(1 - e^{-k})^{i(k-2)}D_0 \gg (1 - e^{-k})^{-i}$ is equivalent to $D_i \gg 1$, again, by multiplying both sides with $(1 - e^{-k})^i$.} than $B_2$. Indeed the ratio of the third term on the right hand side of \eqref{errorterm} and $B_1$ is $\O((1 - e^{-k})^{\frac{-i(k-2)}{2}} \frac{\log N}{\sqrt{D_0}}) = \O((1 - e^{-k})^{\frac{i}{2}} \frac{\log N}{\sqrt{D_i}}) = \O(\frac{\log N}{\sqrt{D_i}}) = o(1)$ using that $D_i \ge \Omega(\log^{\frac{8}{(1 - \varepsilon)}} N) $ by ~\eqref{eq:*} and the sentence after it. The ratio of the first term on the right hand side of \eqref{errorterm} and $B_1$ is estimated similarly. This completes the proof of (c) and also implies (d). Hence this finishes the proof of the claim.
\end{subproof}

	\subsubsection*{\textbf{Step 4.} Proving $(\rm L4)_{i}$ holds with high probability}
	
	For any distinct $x,y \in U_{i-1}$, note that $Y_i(x,y)$ is the number of instances $(e_1,e_2,e_3)$ of $Y_{i-1}(x,y)$ that satisfy $(e_1 \cup e_2 \cup e_3) \setminus \{x,y\} \subseteq U_i$ (namely 
	$(\rm Y_3)_i$). Hence by ~\eqref{eqn:pUi}, and Proposition~\ref{lem:almost_indep}, we have the following recurrence relation.
	\begin{align}
	\label{eq:recuurenceYi}
	\mathbb{E}Y_i(x,y) &= \left ( 1 \pm \frac{C_9}{D_{i-1}} \right ) (1 - p_{i-1}^*)^{3k-4} Y_{i-1}(x,y).
	\end{align}

	\begin{claim}\label{claim:eq_y}
		 With probability at least $1 - e^{-\Omega(\log^2 N)}$, the following holds. For any distinct $x,y \in U_{i-1}$,
		\begin{equation*}
		Y_i(x,y) = \mathbb{E}Y_{i}(x,y) \pm  D_{i-1} \sqrt{Y_{i-1}(x,y)} \log N,
		\end{equation*}
		provided that $ Y_{i-1}(x,y) = \omega(\log^2 N)$.
	\end{claim}
	\begin{subproof}
		We will use the Martingale inequality (Lemma ~\ref{cor:martingale}) to prove the claim. Let $x,y \in U_{i-1}$ be distinct vertices. Recall that $Y_i(x,y)$ is the number of instances $(e_1,e_2,e_3)$ of $Y_{i-1}(x,y)$ such that $(e_1 \cup e_2 \cup e_3) \setminus \{x,y\} \subseteq U_i$. We call all the instances of $Y_{i-1}(x,y)$ \emph{candidates}. \\
		
		{\bf Subclaim 1}. \textit{Paul needs at most $\O(Y_{i-1}(x,y) D_{i-1})$ queries of whether an edge is in $B_i$ and at most $\O(Y_{i-1}(x,y))$ queries of whether a vertex is in $W_i$ in order to determine $Y_i(x,y)$.} \\
		
		Note that $(\rm L1)_{i-1}$ and Proposition ~\ref{PaulsQueries} immediately imply Subclaim 1.\COMMENT{To determine $Y_i(x,y)$, for each candidate $(e_1,e_2,e_3)$, Paul must determine whether each of the vertices of $(e_1 \cup e_2 \cup e_3) \setminus \{x,y\}$ is in $U_i$. Since the number of candidates is $Y_{i-1}(x,y)$, and there are at most $3k$ vertices in each of these candidates, Proposition ~\ref{PaulsQueries} implies Subclaim 1.}\\

		{\bf Subclaim 2}. \textit{If we change the response to a single query (keeping all others the same), then $Y_i(x,y)$ changes by at most $\O(D_{i-1})$.} \\

		Using $(\rm L1)_{i-1}$ and the fact that $H_{i-1}$ is simple, it is easy to show the following. 
		\begin{equation}
		\label{candidateswithv}
		\parbox{.85\textwidth}{For any given vertex $v \in U_{i-1}\setminus \{x,y\}$, the number of candidates $(e_1, e_2, e_3)$ such that $v \in e_1 \cup e_2 \cup e_3 \setminus \{x,y\}$ is $\O(D_{i-1})$.} \tag{$*$}
		\end{equation}
		Now we prove the subclaim using ($*$). First consider queries asking if a vertex is in $W_i$. Changing whether a given vertex $v$ is in $W_i$ affects at most $\O(D_{i-1})$ candidates by ($*$). Thus in this case $Y_i(x,y)$ changes by at most $\O(D_{i-1})$. Now consider queries asking whether if an edge, say $e'$, is in $B_i$. There are at most $k$ pairwise disjoint edges that intersect with $e'$, so changing whether $e'$ is in $B_i$ changes $V(M_i)$ by at most $k^2$ vertices, and each of those vertices affects at most $\O(D_{i-1})$ candidates by ($*$). Thus, $Y_i(x,y)$ changes by at most $\O(D_{i-1})$ in this case as well, proving Subclaim 2.

		Thus using Subclaims 1 and 2, Proposition ~\ref{claim:waste}, and the fact that an edge is in $B_i$ with probability $1/D_{i-1}$, we can apply Lemma ~\ref{cor:martingale} with $\sigma = \Theta(D_{i-1}\sqrt{Y_{i-1}(x,y)})$ and $T = \Theta(D_{i-1})$ (so $\sigma/T = \omega(\log N)$ by our assumption) and $\lambda = \Theta(\log N)$ with an appropriately chosen multiplicative constant, to obtain that with probability at least $1 - e^{-\Omega(\log^2 N)}$,
		\begin{equation*}
		Y_i(x,y) = \mathbb{E}Y_{i}(x,y) \pm D_{i-1}\sqrt{Y_{i-1}(x,y)}\log N.
		\end{equation*}
		By the union bound, with probability at least
		$1 - |U_{i-1}|^2e^{-\Omega(\log^2 N)} = 1 - e^{-\Omega(\log^2 N)}$,
		the claim holds.
	\end{subproof}

	\begin{claim}\label{claim:boundy}
		With probability at least $1 - e^{-\Omega(\log^2 N)}$, $(\rm L4)_i$ holds.
	\end{claim}
	\begin{subproof}
		
		Recall that $C_7 = C_7(k)$ was chosen in \eqref{eq:mainhierarchy}. We may assume that $Y_{i-1}(x,y) \geq C_7 D_i^2 \log^2 N = \omega(\log^2 N)$, otherwise
		\begin{equation*}
		Y_{i}(x,y) \leq Y_{i-1}(x,y) < C_7 D_i^2 \log^2 N
		\end{equation*}
		and the claim holds. Hence by Claim~\ref{claim:eq_y}, with probability at least $1 - e^{-\Omega(\log^2 N)}$,
		\begin{align*}
		Y_i(x,y) &= \mathbb{E}Y_i(x,y) \pm \sqrt{Y_{i-1}(x,y)} D_{i-1} \log N \\ 
		\overset{\eqref{eq:recuurenceYi}}&{=}\left ( 1 \pm \frac{C_9}{D_{i-1}} \right ) (1 - p_{i-1}^*)^{3k-4} Y_{i-1}(x,y) \pm \sqrt{Y_{i-1}(x,y)} D_{i-1} \log N\\
		\overset{(\rm L4)_{i-1}}&{\leq} (1 + o(1)) (1 - p_{i-1}^*)^{3k-4} C_7 D_{i-1}^2 \log^2 N + \sqrt{C_7} D_{i-1}^2 \log^2 N \\
		\overset{\eqref{Didef}}&{\leq} \left [ (1 + o(1))(1 - p_{i-1}^*)^{k-2} C_7 + \frac{\sqrt{C_7}}{(1 - p_{i-1}^*)^{2k-2}} \right ] D_i^2 \log^2 N\\
		& < C_7 D_i^2 \log^2 N
		\end{align*}
		by \eqref{eq:mainhierarchy} and the fact that $1 - p_{i-1}^* = (1 + o(1))(1 - e^{-k})$ by Propositions ~\ref{claim:matching} and ~\ref{claim:bound_delta}. This completes the proof of the claim.	
	\end{subproof}
	
	\subsubsection*{\textbf{Step 5.} Proving $(\rm L5)_{i}$ holds with high probability}
	
	For any distinct $x,y \in U_{i-1}$, note that $X_i(x,y)$ is the sum of
	\begin{itemize}
		\item the number of instances $(e_1,e_2,e_3)$ of $X_{i-1}(x,y)$ that satisfy $(e_1 \cup e_2) \setminus (e_3 \cup \{x,y \}) \subseteq U_i$, and 
		\item the number of instances $(e_1,e_2,e_3)$ of $Y_{i-1}(x,y)$ that satisfy $(e_1 \cup e_2) \setminus (e_3 \cup \{x,y \}) \subseteq U_i$ and $e_3 \in E(M_i)$.
	\end{itemize}
	
	By \eqref{eqn:pUi} and Proposition~\ref{lem:almost_indep}, the expected value of the former number is $ \left ( 1 \pm \frac{C_9}{D_{i-1}} \right ) (1 - p_{i-1}^*)^{2k-4} X_{i-1}(x,y)$, and  by Proposition~\ref{claim:matching} the expected value of the latter number is at most $(1 \pm C_0 \Delta_{i-1} D_{i-1}^{-1}) \frac{e^{-k}}{D_{i-1}} Y_{i-1}(x,y) \le \frac{C_{10}}{D_{i-1}} Y_{i-1}(x,y)$ by \eqref{eq:mainhierarchy}. Hence we have the following recurrence relation.
	\begin{align}\label{eqn:expxi}
	\mathbb{E}X_i(x,y) &\leq \left ( 1 \pm \frac{C_9}{D_{i-1}} \right ) (1 - p_{i-1}^*)^{2k-4} X_{i-1}(x,y) + \frac{C_{10}}{D_{i-1}} Y_{i-1}(x,y).
	\end{align}
	
	\begin{claim}\label{claim:eq_x}
		With probability at least $1 - e^{-\Omega(\log^2 N)}$, the following holds. 
		For any distinct $x,y \in U_{i-1}$, 
		\begin{equation*}
		X_i(x,y) = \mathbb{E}X_{i}(x,y) \pm \sqrt{X_{i-1}(x,y) + Y_{i-1}(x,y)} \log N,
		\end{equation*}
		provided that $X_{i-1}(x,y) + Y_{i-1}(x,y) = \omega(\log^2 N)$.
	\end{claim}
	\begin{subproof}
		We will use the Martingale inequality (Lemma ~\ref{cor:martingale}) to prove the claim. Let $x,y \in U_{i-1}$ be distinct vertices. Note that $X_i(x,y)$ is the number of instances $(e_1,e_2,e_3)$ of $X_{i-1}(x,y)$ and $Y_{i-1}(x,y)$ such that both the events 
		\begin{itemize}
			\item[$(\rm T1)$] $e_3 \in \bigcup_{j \leq i} E(M_j)$ and
			\item[$(\rm T2)$] $(e_1 \cup e_2) \setminus (e_3 \cup \{x,y \}) \subseteq U_i$
		\end{itemize}
		hold. We call all the instances of $X_{i-1}(x,y)$ and $Y_{i-1}(x,y)$ \emph{candidates}. \\
		
		{\bf Subclaim 1}. \textit{Paul needs at most $\O((X_{i-1}(x,y) + Y_{i-1}(x,y)) D_{i-1})$ queries of whether an edge is in $B_i$ and at most $\O((X_{i-1}(x,y) + Y_{i-1}(x,y))$ queries of whether a vertex is in $W_i$ in order to determine $X_i(x,y)$.} \\
		
	    Subclaim 1 follows easily from $(\rm L1)_{i-1}$ and Proposition ~\ref{PaulsQueries}.\COMMENT{Since the total number of candidates is $X_{i-1}(x,y) + Y_{i-1}(x,y)$, it suffices to show that Paul needs at most $\O(D_{i-1})$ queries of whether an edge is in $B_i$ and at most $\O_k(1)$ queries of whether a vertex is in $W_i$ to determine whether each of these candidates satisfies $(\rm T1)$ and $(\rm T2)$.
		Let us fix a candidate $(e_1 , e_2 , e_3)$. If $e_3 \in \bigcup_{j < i} E(M_j)$, then $(\rm T1)$ automatically holds. Otherwise, the event that $e_3$ is in $E(M_i)$ is determined by the queries of whether edges in $E(H_{i-1})$ sharing a vertex with $e_3$ are in $B_i$. Since the number of such edges is $\O(D_{i-1})$ by $(\rm L1)_{i-1}$, Paul needs $\O(D_{i-1})$ queries of whether an edge is in $B_i$ to check whether $(\rm T1)$ holds for this candidate.
		In order to check whether $(\rm T2)$ holds, one must verify whether each vertex $w \in (e_1 \cup e_2) \setminus (e_3 \cup \{x,y \})$ is in $U_i$. Thus Proposition ~\ref{PaulsQueries} implies that Paul needs at most $\O(D_{i-1})$ queries of whether an edge is in $B_i$ and $\O_k(1)$ queries of whether a vertex is in $W_i$ to verify whether $(\rm T2)$ holds for this candidate.  This completes the proof of Subclaim 1.}\\

		{\bf Subclaim 2}. \textit{If we change the response to a single query (keeping all others the same), then $X_i(x,y)$ changes by at most $\O_k(1)$.} \\

		First let us consider queries asking if a vertex is in $W_i$.  Changing whether a given vertex $v$ is in $W_i$ does not affect whether $(\rm T1)$ events hold for the candidates. Moreover, there are at most $\O_k(1)$ candidates $(e_1,e_2,e_3)$ such that $v \in (e_1 \cup e_2) \setminus (e_3 \cup \{x,y \})$ and $(\rm T1)$ holds, due to the following reason. By symmetry, we may assume that $v \in e_1 \setminus (e_3 \cup \{x,y \})$. There is exactly one choice of $e_1$ (as $H$ is simple), at most $k$ choices of $e_3$ that satisfy $(\rm T1)$, $k-1$ choices of the vertex in $e_2 \cap e_3$, and given such a vertex, there is at most one choice of $e_2$ (as $H$ is simple).
		Thus, changing whether $v \in W_i$ affects at most $\O_k(1)$ candidates. Thus $X_i(x,y)$ changes by at most $\O_k(1)$.

		Now we focus on queries asking if an edge, say $e'$, is in $B_i$.
		There are at most $k$ pairwise disjoint edges in $H_{i-1}$ that intersect with $e'$, so changing whether $e'$ is in $B_i$ changes $E(M_i)$ by at most $k+1$ edges. Call such edges \emph{affected}. Each affected edge is the central edge $e_3$ of at most $\O_k(1)$ candidates $(e_1, e_2, e_3)$. Indeed, if the central edge $e_3$ and vertices $x, y$ are fixed, there are at most $k$ choices of the vertex in $e_1 \cap e_3$ and at most $k-1$ choices of the vertex in $e_2 \cap e_3$, and for each such choice, there is at most one choice for $e_1$ or $e_2$ (as $H$ is simple and $x \in e_1$, $y \in e_2$). Thus changing whether $e'$ is in $B_i$ affects at most $\O_k(1)$ candidates by affecting their $(\rm T1)$ events.

		It remains to count the number of candidates for which $(\rm T2)$ events are affected.  Note that changing whether $e'$ is in $B_i$ changes
		$V(M_i)$ by at most $k^2$ vertices. Let us call those vertices \emph{affected}. Recall that, for a candidate to contribute towards $X_i(x,y)$, it is necessary that $(\rm T1)$ holds for it. Since $H_{i-1}$ is simple, for each affected vertex $v$, there are at most $\O_k(1)$ candidates $(e_1,e_2,e_3)$ such that $v \in (e_1 \cup e_2) \setminus (e_3 \cup \{x,y \}) $, and $(\rm T1)$ holds.\COMMENT{by the same reason as in the first paragraph of this subclaim. 
		By symmetry, we may assume that $v \in e_1 \setminus (e_3 \cup \{x,y \})$. There is exactly one choice of $e_1$ (as $H$ is simple), at most $k$ choices of $e_3$ that satisfy $(\rm T1)$, $k-1$ choices of the vertex in $e_2 \cap e_3$, and given such a vertex, at most one choice of $e_2$ (as $H$ is simple).} Thus, $X_i(x,y)$ changes by at most $\O_k(1)$. This completes the proof of Subclaim 2. \\
		
		Thus using Subclaims 1 and 2, Proposition ~\ref{claim:waste}, and the fact that an edge is in $B_i$ with probability $1/D_{i-1}$, we can apply Lemma ~\ref{cor:martingale} with $\sigma = \Theta(\sqrt{X_{i-1}(x,y) + Y_{i-1}(x,y)})$ (so $\sigma = \omega(\log N)$ by the assumption of Claim~\ref{claim:eq_x}) and $\lambda = \Theta(\log N)$ with an appropriately chosen multiplicative constant, to obtain that with probability at least $1 - e^{-\Omega(\log^2 N)}$,
		\begin{equation*}
		X_i(x,y) = \mathbb{E}X_{i}(x,y) \pm \sqrt{X_{i-1}(x,y) + Y_{i-1}(x,y)}\log N.
		\end{equation*}

		By the union bound, with probability at least
		$1 - |U_{i-1}|^2e^{-\Omega(\log^2 N)} = 1 - e^{-\Omega(\log^2 N)},$
		the claim holds.
	\end{subproof}
	
	\begin{claim}\label{claim:boundx}
		With probability at least $1 - e^{-\Omega(\log^2 N)}$, $(\rm L5)_i$ holds.
	\end{claim}
	\begin{subproof}

		We may assume that
		\begin{equation*}
		X_{i-1}(x,y) + Y_{i-1}(x,y) = \omega(\log^2 N),
		\end{equation*}
		since otherwise $X_i(x,y) \leq X_{i-1}(x,y) + Y_{i-1}(x,y) = \O(\log^2 N)$ and the claim automatically holds.

		In the rest of the proof of Claim ~\ref{claim:boundx}, for convenience, we write $X_{i-1}$ and $Y_{i-1}$ instead of $X_{i-1}(x,y)$ and $Y_{i-1}(x,y)$, respectively.
		
		By~\eqref{eqn:expxi}, Claim~\ref{claim:eq_x}, $(\rm L4)_{i-1}$ and $(\rm L5)_{i-1}$ the following inequality holds with probability at least $1 - e^{-\Omega(\log^2 N)}$,
		\begin{align*}
		X_i(x,y) & \leq \left ( 1 \pm \frac{C_{9}}{D_{i-1}} \right ) (1 - p_{i-1}^*)^{2k-4} X_{i-1} + C_7C_{10} D_{i-1} \log^2 N + \sqrt{X_{i-1} + Y_{i-1}} \log N \\
		& \leq \left ( 1 \pm \frac{C_{9}}{D_{i-1}} \right ) (1 - p_{i-1}^*)^{2k-4} X_{i-1} + C_7 C_{10} D_{i-1} \log^2 N + \sqrt{2 C_7} D_{i-1} \log^2 N\\
		\overset{\eqref{Didef}}&{\leq} \left [(1 + o(1)) (1 - p_{i-1}^*)^{k-3}C_8 + \frac{C_7 C_{10} + \sqrt{2C_7}}{(1 - p_{i-1}^*)^{k-1}} \right ] D_i \log^2 N\\
		& < C_8 D_i \log^2 N,
		\end{align*}
		where in the second and third inequality we used $X_{i-1} \leq C_8 D_{i-1} \log^2 N$ and $Y_{i-1} \leq C_7 D_{i-1}^2 \log^2 N$ by $(\rm L4)_{i-1}$ and $(\rm L5)_{i-1}$, and the last inequality follows by 
		~\eqref{eq:mainhierarchy} and the fact that $1 - p_{i-1}^* = (1 + o(1))(1 - e^{-k})$ by Propositions ~\ref{claim:matching} and ~\ref{claim:bound_delta}. This completes the proof of the claim.
	\end{subproof}
	This concludes the proof of Lemma~\ref{lem:nibble} (Nibble lemma). \qed
	
	\subsection{Proof of Theorem~\ref{thm:mainsec3}}
	\label{subsection:theorem3.1}
	Iteratively applying Lemma~\ref{lem:nibble} (Nibble lemma), we obtain that with non-zero probability $(\rm L1)_i$--$(\rm L5)_i$ hold for all $1 \leq i \leq \zeta$. Consider some outcome of the Nibble process for which $(\rm L1)_i$--$(\rm L5)_i$ hold for all $1 \leq i \leq \zeta$.
	
	Let $M := M_1 \cup \dots \cup M_\zeta$ and $W := W_1 \cup \dots \cup W_\zeta$. Then $U_\zeta = V(H) \setminus (V(M) \cup W).$ We will first show that $(\rm M1)$ of Theorem~\ref{thm:mainsec3} holds. Note that Proposition~\ref{claim:bound_delta} implies that for all $1 \le i \le \zeta$ and all $0 \le t \le i$,
	\begin{align}
	q_{t,i} &= (1+o(1)) (1 -  e^{- k})^{i-t}, \label{eqn:boundqjcor}\\
	D_i &= (1+o(1)) (1 - e^{- k})^{i(k-1)} D_0. \label{eqn:bounddjcor}
	\end{align}
	
	\begin{claim}\label{claim:bounddomega}
		$(1 - e^{-k})^{\zeta(k-1)} D_0 = \Theta (D_0^{\gamma})$ and $D_\zeta = \Theta (D_0^{\gamma})$.
	\end{claim}
	\begin{subproof}
		By the definition of $\zeta$, we have  $D_\zeta \leq D_0^\gamma$, and $D_{\zeta-1} > D_0^\gamma$. Thus we have
		\begin{equation}\label{eqn:d_omega}
		\frac{(1 - e^{-k})^{k-1} D_0^\gamma}{4} < \frac{(1 - e^{-k})^{k-1}D_{\zeta - 1}}{4} \overset{\eqref{eqn:bounddjcor}}{\leq} \frac{(1 - e^{-k})^{\zeta(k-1)} D_0}{2}  \overset{\eqref{eqn:bounddjcor}}{\leq} D_\zeta \leq D_0^\gamma,
		\end{equation}
		proving the claim.
	\end{subproof}
	
	Now define $p := (D_\zeta / D_0)^{\frac{1}{k-1}}$. By~\eqref{eqn:bounddjcor} and Claim~\ref{claim:bounddomega}, we have $p = (1+o(1)) (1-e^{-k})^{\zeta} = \Theta(D_0^{\frac{\gamma-1}{k-1}})$. Hence, by $(\rm L2)_\zeta$(c) we have 
	\begin{equation*}
	|S \cap U_{\zeta}| = (1+o(1))(1 - e^{-k})^{\zeta} |S| = (1+o(1)) p |S|
	\end{equation*}
	for every $S \in \mathcal{V}$. This proves $(\rm M1)$.
	
	\begin{claim}
		$(\rm M2)$ holds.
	\end{claim}
	\begin{subproof}
		For any $1 \leq i \leq \zeta$, $(\rm L2)_i$(\rm a) implies that 
		\begin{equation}
		\label{eq:swi}
		  |S \cap W_i| = \O \left(|S \cap U_{i-1}|\frac{\Delta_{i-1}}{D_{i-1}} + |S \cap U_{i-1}| \frac{\log N}{\sqrt{D_{i-1}}}\right). 
		\end{equation}
		By Proposition~\ref{claim:bound_delta} $(\rm A3)_{i-1}$ and $(\rm L2)_{i-1}$(\rm c) we have
		\begin{align}\label{eqn:boundratio}
            \begin{split}
		&|S \cap U_{i-1}| \left ( \frac{\Delta_{i-1}}{D_{i-1}} + \frac{\log N}{\sqrt{D_{i-1}}} \right ) \\ &\overset{\eqref{eqn:bounddjcor}}{=} \Theta ((1 - e^{-k})^{i-1}|S| \cdot D_0^{-1}\Delta_0 + (1 - e^{-k})^{(i-1)-\frac{(i-1)(k-1)}{2}} |S| D_0^{-1/2} \log N),
            \end{split}
		\end{align}
		where the constant implicit in $\Theta (\cdot)$ is independent of the index $i$. Hence we have,\COMMENT{Explanation for \eqref{eqn:boundw}: Summing the decreasing geometric series we get $\sum_{i=0}^{\zeta-1} (1 - e^{-k})^i = \O_k(1)$. Note that since $k > 3$ the sum $\sum_{i=0}^{\zeta-1}(1 - e^{-k})^{i - \frac{i(k-1)}{2}}$ is an increasing geometric series, so the last term dominates; i.e., we get $\Theta((1 - e^{-k})^{\zeta - \frac{\zeta(k-1)}{2}})$. Thus $\sum_{i=0}^{\zeta-1}(1 - e^{-k})^{i - \frac{i(k-1)}{2}} |S| D_0^{-1/2} \log N  ) = \Theta((1 - e^{-k})^{\zeta - \frac{\zeta(k-1)}{2}}) |S| D_0^{-1/2} \log N = \Theta((1 - e^{-k})^{\zeta}D_{\zeta}^{-1/2}|S|\log N)$. By Claim ~\ref{claim:bounddomega}, $(1 - e^{-k})^{\zeta} =D_0^{(\gamma-1)/(k-1)}$ and $D_{\zeta} = \Theta(D_0^{\gamma})$, so $\Theta((1 - e^{-k})^{\zeta}D_{\zeta}^{-1/2}|S|\log N) = \Theta(D_0^{(\gamma-1)/(k-1)} D_0^{-\gamma/2}|S|\log N)$.
		}
		\begin{align}\label{eqn:boundw}
            \begin{split}
		|S \cap W| &= \sum_{i=1}^{\zeta} |S \cap W_i| \overset{\eqref{eq:swi}}{=} \sum_{i=1}^{\zeta} \O\left(|S \cap U_{i-1}| \frac{\Delta_{i-1}}{D_{i-1}} + |S \cap U_{i-1}| \frac{\log N}{\sqrt{D_{i-1}}}\right) \\
		\overset{\eqref{eqn:boundratio}}&{=} \O \left ( \sum_{i=0}^{\zeta-1} (1 - e^{-k})^i |S| \cdot D_0^{-1} \Delta_0 + \sum_{i=0}^{\zeta-1}(1 - e^{-k})^{i - \frac{i(k-1)}{2}} |S| D_0^{-1/2} \log N  \right ) \\
		&= \O ( |S| D_0^{-1}\Delta_0 + |S| D_0^{-\frac{1}{k-1}} D_0^{\gamma \left ( \frac{1}{k-1} - \frac{1}{2} \right ) } \log N), 
            \end{split}
		\end{align}
		where in \eqref{eqn:boundw} we used Claim ~\ref{claim:bounddomega} and the fact that $\sum_{i=0}^{\zeta-1}(1 - e^{-k})^{i - \frac{i(k-1)}{2}} |S| D_0^{-1/2} \log N  ) = \Theta((1 - e^{-k})^{\zeta - \frac{\zeta(k-1)}{2}}) |S| D_0^{-1/2} \log N = \Theta((1 - e^{-k})^{\zeta}D_{\zeta}^{-1/2}|S|\log N)$ by Proposition~\ref{claim:bound_delta}. This completes the proof of the claim.
	\end{subproof}
	
	\begin{claim}\label{claim:bddeomega}
		For any $x \in U_\zeta$, let $Z_\zeta := (k-1) e^{-k} \sum_{j<\zeta}q_{j,\zeta}^{k-2} D_j$ be the main term of $Z_\zeta(x)$. Then
		\begin{equation}
		\label{Zomegax}
		     Z_\zeta(x) = \left ( 1 + \O \left ( \frac{\Delta_\zeta}{D_\zeta} \right ) \right ) Z_\zeta.
		\end{equation}
	\end{claim}
	\begin{subproof}

		We have 
		\begin{equation}
		\label{eq:Zomega}
		  Z_\zeta \overset{(\rm L3)_{\zeta}(c),(d)}{=} \Theta ( (1 - e^{-k})^{\zeta(k-2)} D_0 ) \overset{\text{Claim}~\ref{claim:bounddomega}}{=} \Theta (D_0^{\gamma \left (1 - \frac{1}{k-1} \right ) + \frac{1}{k-1}}).
		\end{equation}
		By $(\rm L3)_{\zeta}$(b),
		$$Z_\zeta(x) =  Z_\zeta + \O ( (1 - e^{-k})^{-\zeta} + (1 - e^{-k})^{\zeta(k-2)} \Delta_0 + (1 - e^{-k})^{\frac{\zeta(k-2)}{2}} \sqrt{D_0} \log N).$$
		We estimate each of these error terms as follows:\COMMENT{The first equation of the second line below is obtained by rearranging the first equation of \eqref{eq:Zomega}. The first equation of the third line below is obtained as follows. $(1 - e^{-k})^{\frac{\zeta(k-2)}{2}} \sqrt{D_0} \log N \overset{\eqref{eq:Zomega}}{=} Z_{\zeta}^{1/2}\log N = Z_{\zeta}^{-1/2} \log N \cdot Z_{\zeta}$ which is equal to  $D_0^{-\frac{1}{2(k-1)} - \frac{k-2}{2(k-1)}\gamma }\log N \cdot Z_\zeta$ by the second equation of \eqref{eq:Zomega}.}
		\begin{align*}
		(1 - e^{-k})^{-\zeta} \overset{\text{Claim}~\ref{claim:bounddomega}}&{=} \Theta(D_0^{\frac{1-\gamma}{k-1}}) \overset{\eqref{eq:Zomega}}{=} \O(D_0^{-\gamma} \cdot Z_\zeta) \leq \frac{\Delta_\zeta}{D_\zeta} Z_\zeta, \\
		(1 - e^{-k})^{\zeta(k-2)} \Delta_0 \overset{\eqref{eq:Zomega}}&{=} \O \left ( \frac{\Delta_0}{D_0} Z_\zeta \right ) \leq \O \left ( \frac{\Delta_\zeta}{D_\zeta} Z_\zeta \right ),\\
		(1 - e^{-k})^{\frac{\zeta(k-2)}{2}} \sqrt{D_0} \log N \overset{\eqref{eq:Zomega}}&{=} \O \left ( D_0^{-\frac{1}{2(k-1)} - \frac{k-2}{2(k-1)}\gamma }\log N \cdot Z_\zeta \right ) \leq \O \left ( \frac{\Delta_\zeta}{D_\zeta} Z_\zeta \right ),
		\end{align*}
		where\COMMENT{Note that $D_0^{-\frac{1}{2(k-1)} - \frac{k-2}{2(k-1)}\gamma} < D_0^{-\gamma/2}$ since $\gamma < 1$.} in each of the above inequalities we used $\frac{\Delta_\zeta}{ D_\zeta} = \frac{\Delta_0}{ D_0} +  \Theta (D_0^{-\gamma/2} \log N)$ (which holds by Proposition~\ref{claim:bound_delta} and Claim~\ref{claim:bounddomega}). Adding up the above estimates proves the claim.
	\end{subproof}

	Let $H_A$ be the multi-hypergraph of augmenting stars of $H$ with respect to $(M,W)$ (see Definition~\ref{Def:HypergraphofAugStars}). Recall that $V(H) \setminus (V(M) \cup W) = U_{\zeta}$, so $V(H_A) = E(M) \cup U_{\zeta}$. Similarly as in Section~\ref{proofonline}, let $L := E(M)$ and $R := U_{\zeta}$. The following three claims prove $(\rm M3)$.
	
	\begin{claim}
		For any $e \in E(M)$,
		\begin{equation}\label{eqn:bounddegleft}
		d_{H_A}(e) = (1 + \O ( D_0^{-1}\Delta_0 + D_0^{-\gamma/2} \log N)) D_{L},
		\end{equation}
		where the constants implicit in $\O(\cdot)$ and $D_L = \Theta (D_0^{k\gamma})$ do not depend on $e$.
	\end{claim}
	\begin{subproof}
		Fix an arbitrary element $e \in E(M)$. Recall that $d_{H_A}(e)$ is the number of augmenting stars $(e_M , \{e_1,\dots , e_k \}) \in \mathcal{A}(H)$ such that $e_M = e$.
		
		Now let $e = \{v_1 , \dots , v_k \}$. For each $1 \leq j \leq k$, the number of choices for $e_j \in E(H)$ such that $\{v_j \} = e \cap e_j$ and $e_1 , \dots , e_j$ are disjoint is $D_\zeta \pm \Delta_\zeta \pm \O_k(1)$ since $H$ is simple and $\rm (L1)_{\zeta}$ holds. Thus,
		\begin{align*}
		d_{H_A}(e) = \prod_{j=1}^{k} (D_\zeta \pm (\Delta_\zeta + \O_k(1)))  = \left ( 1 + \O \left ( \frac{\Delta_\zeta}{D_\zeta} \right ) \right )  D_\zeta^k.
		\end{align*}
		Hence the claim follows by Proposition~\ref{claim:bound_delta} $\rm (A3)_{\zeta}$ and Claim~\ref{claim:bounddomega}.
	\end{subproof}
	
	\begin{claim}
		For any $x \in U_\zeta$,
		\begin{equation}\label{eqn:bounddegright}
		d_{H_A}(x) = (1 + \O ( D_0^{-1}\Delta_0 + D_0^{-\gamma/2} \log N)) D_{R},
		\end{equation}
		where the constants implicit in $\O(\cdot)$ and $D_R = \Theta (D_0^{k\gamma + \frac{1-\gamma}{k-1}})$ do not depend on the vertex $x$.
	\end{claim}
	\begin{subproof}
		Fix an arbitrary vertex $x \in U_{\zeta} \subseteq V(H_A)$. Recall that $d_{H_A}(x)$ is the number of  augmenting stars $(e_M , \{e_1 , \dots , e_k \}) \in \mathcal{A}(H)$ such that $x \in e_1 \cup \dots \cup e_k$. Since $e_1 , \dots , e_k$ are disjoint, $x$ lies in exactly one of them. Without loss of generality, assume that $x \in e_1$. 
		
		There are $Z_\zeta(x)$ such choices for $e_1$, and given such a choice of $e_1$, let $x_1$ be the unique vertex in $e_1 \cap V(M)$ and $e_M \in E(M)$ be the edge that contains $x_1$. 
			
		Let $e_M \setminus \{x_1 \}:= \{x_2 , \dots , x_{k} \}$. For each $2 \leq j \leq k$, we have to choose an edge $e_j \in E(H)$ that contains $x_j$ such that $e_j$ is disjoint from $e_1 \cup \dots \cup e_{j-1}$, $|e_j \cap e_M| = 1$ and $e_j \setminus e_M \subseteq U_\zeta$. Since $H$ is simple, there are at most $\O_k(1)$ edges that contain $x_j$ and intersect $e_1 \cup \dots \cup e_{j-1}$, so the number of such choices of $e_j$ is $D_\zeta(x_j) + \O_k(1) = D_{\zeta} \pm \Delta_{\zeta} + \O_k(1)$ by $(\rm L1)_{\zeta}$. Hence the number of choices of $\{e_2 , \dots , e_{k} \}$ is 
		$$(D_{\zeta} \pm \Delta_{\zeta} + \O_k(1))^{k-1} = \left(1 + \O\left(\frac{\Delta_{\zeta}}{D_{\zeta}}\right) \right) D_\zeta^{k-1}. $$
		Thus there are 
		\begin{equation}
		    \label{case1dHA}
			\left ( 1 + \O \left ( \frac{\Delta_\zeta}{D_\zeta} \right ) \right ) Z_\zeta(x) D_\zeta^{k-1} 
		\end{equation}
		augmenting stars $(e_M ,\{e_1,e_2,\dots , e_k \}) \in \mathcal{A}(H)$ such that $x \in e_1 \cup \dots \cup e_k$.
		Therefore, by Proposition~\ref{claim:bound_delta}$(\rm A3)_{\zeta}$ and Claim ~\ref{claim:bounddomega}, we obtain
		\begin{equation*}
		d_{H_A}(x) \overset{\eqref{Zomegax}, \eqref{case1dHA}}{=}  \left (1 + \O \left ( \frac{\Delta_\zeta}{D_\zeta} \right ) \right ) Z_\zeta D_\zeta^{k-1} \overset{\eqref{eq:Zomega}}{=} (1 + \O (D_0^{-1}\Delta_0 + D_0^{-\gamma/2} \log N) ) D_{R},
		\end{equation*}
		where $D_{R} = \Theta (D_0^{k\gamma + \frac{1-\gamma}{k-1}})$ and the constants implicit in $\O(\cdot)$ and $\Theta(\cdot)$ do not depend on the vertex $x$, as desired.
	\end{subproof}
	
	\begin{claim}
		The codegree of $H_A$ is $\O (D_0^{\gamma(k-1)} \log^2 N)$.
	\end{claim}
	\begin{subproof}
		Let $x,y \in V(H_A) = E(M) \cup U_{\zeta}$ be distinct vertices. If $x,y \in E(M)$, then the codegree of $\{x,y\}$ in $H_A$ is $0$. If $x \in E(M)$ and $y \in U_{\zeta}$, then the codegree of $\{x,y\}$ in $H_A$ is equal to the number of augmenting stars $(e_M , \{e_1 , \dots , e_k\})$ of $H$ with respect to $(M,W)$ where $x = e_M$ and $y \in \bigcup_{i=1}^{k} e_i \setminus e_M$. If $y \in e_j$ for some $j$ with $1 \le j \le k$, the number of choices for $e_j$ is $\O_k(1)$ since $H$ is simple. The number of choices of each $e_i$ with $i \ne j $ is $\O(D_{\zeta})$ by 
		$(\rm L1)_{\zeta}$. Thus the total number of such choices of augmenting stars $(e_M , \{e_1 , \dots , e_k\})$ is $ \O(D_\zeta^{k-1}) = \O(D_0^{\gamma(k-1)})$ by Claim ~\ref{claim:bounddomega}.

		It remains to bound the codegree of $\{x,y\}$ for $x, y \in U_{\zeta}$, which is the number of augmenting stars $(e_M , \{e_1 , \dots , e_k\})$ of $H$ with respect to $(M,W)$, where $x,y \in \bigcup_{i=1}^{k} e_i \setminus e_M$.

		Firstly, note that the number of augmenting stars where $x,y \in e_j \setminus e_M$, for some $1 \le j \le k$ is $\O(D_\zeta^{k-1})$. Indeed,  as $H$ is simple, there is at most one choice for $e_j$, and the number of choices for each $e_i$ with $i \ne j$ is $\O(D_{\zeta})$ by $(\rm L1)_{\zeta}$. 
		Secondly, the number of augmenting stars where $x \in e_{j_1} \setminus e_M$ and $y \in e_{j_2} \setminus e_M$ for $j_1 \ne j_2$ is $\O(X_\zeta(x,y) D_\zeta^{k-2})$. Indeed, the number of choices for such triples $(e_M, e_{j_1}, e_{j_2})$ is $X_\zeta(x,y)$ by definition, and the number of choices for each $e_i$ with $i \not \in \{j_1, j_2\}$ is $\O(D_{\zeta})$ by $(\rm L1)_{\zeta}$. Hence in total, the codegree of $\{x, y\}$ in $H_A$ is at most
		\begin{equation}\label{eqn:codegha}
		\O( D_\zeta^{k-1} + X_\zeta(x,y) D_\zeta^{k-2} ) \overset{(\rm L5)_\zeta}{=} \O( D_\zeta^{k-1} + D_\zeta \log^2 N D_\zeta^{k-2} ) = \O( D_0^{\gamma(k-1)} \log^2 N),
		\end{equation}
		by Claim~\ref{claim:bounddomega}, where the constant implicit in $\O(\cdot)$ does not depend on the choice of $x$ and $y$, as desired.
	\end{subproof}

	\section{\texorpdfstring{Using augmenting stars to find large matchings \\ in almost regular simple hypergraphs}{Using augmenting stars to find large matchings in almost regular simple hypergraphs}}
	\label{sec:AKSimproved}

	Using Theorem ~\ref{thm:mainsec3}, in this section we prove the following result which implies Theorem~\ref{thm:main3intro} in the case when $H$ is simple (by taking $\mathcal{V} = \{V(H)\}$).
	
	\begin{theorem}\label{thm:main2proof}
		Let $k > 3$ be an integer, let $0 < \varepsilon <  1 - \frac{1}{k-1}$, and let 
		\begin{equation}\label{def:eta_0}
		\eta_0 := \min \left ( \frac{k-3}{(k-1)(k^3 - 2k^2 - k+4)},\:\:1 - \frac{1}{k-1} - \varepsilon \right ).
		\end{equation}
		Let $0 < \eta < \eta_0,$ 
		and $\mu,K > 0$. Then there exists $N_0 = N_0(\varepsilon, \eta , \mu , k , K)$ such that the following holds.
		
		Let $H$ be a $k$-uniform $(D \pm K D^\varepsilon)$-regular simple hypergraph on $N$ vertices, where $N \geq N_0$ and $D \geq \exp (\log^\mu N)$. Let $\mathcal{V}$ be a collection of subsets of $V(H)$ such that $|\mathcal{V}| \leq \exp( \log^{4/3} N)$ and for each $S \in \mathcal{V}$, we have $|S| \geq \sqrt{D} \log N$.
		
		Then there is a matching in $H$ covering all but at most $|S| D^{-\frac{1}{k-1} - \eta}$ vertices of $S$, for every $S \in \mathcal{V}$.
	\end{theorem}

	\begin{proof}
		Let us first define
		\begin{equation}
		\label{def:gamma}
		  \gamma := \min\left ( \frac{2}{k^3-2k^2 -k +4} \:,\: \frac{2(k-1)(1-\varepsilon)-2}{k-3} \right ) \text{ and } \gamma' := \min \left ( \frac{1}{4(k-1)} ,\: \eta_0 - \eta \right ). 
		\end{equation}

		Let us choose $N_0 \in \mathbb{N}$ and $\delta$ such that
		\begin{align}
		0 &< {N_0}^{-1} \ll \delta \ll \gamma , \varepsilon,\: 1 - \frac{1}{k-1} - \varepsilon,\: \mu ,\eta, \eta_0 - \eta , k^{-1} , K^{-1} < 1, \label{eqn:hierachy}
		\end{align}
	    In the rest of the proof, the implicit constants in $\O(\cdot)$ and  $\Theta(\cdot)$ will only depend on the parameters $\delta, \eta, \varepsilon,\mu,k,K$.

		Let $H$ be a $k$-uniform $(D \pm KD^\varepsilon)$-regular simple hypergraph on $N$ vertices, where $D \geq \exp(\log^\mu N)$.

		Applying Theorem~\ref{thm:mainsec3} to $H$ (with $D = D_0$, $\Delta_0 = K D^\varepsilon$, and $\mathcal{V} \cup \{V(H) \}$ playing the role of $\mathcal V$) , we obtain a matching $M$ of $H$ and a set of waste vertices $W \subseteq V(H)$ such that the following hold.
		\begin{itemize}
			\item[$(\rm M1)_H$] $|S \setminus (V(M) \cup W)| = \Theta (|S| D^{\frac{\gamma-1}{k-1}})$ for every $S \in \mathcal{V} \cup \{V(H) \}$.
			
			\item[$(\rm M2)_H$] $|S \cap W| = \O ( K D^{\varepsilon-1} |S| + |S| D^{-\frac{1}{k-1}} D^{\gamma \left ( \frac{1}{k-1} - \frac{1}{2} \right )} \log N )$ for every $S \in \mathcal{V} \cup \{ V(H) \}$.
			
			\item[$(\rm M3)_H$] For the $(1,k(k-1))$-partite multi-hypergraph $H_A$ of augmenting stars of $H$ with respect to $(M,W)$ the following holds. For any $e \in E(M)$ and $x \in V(H) \setminus (V(M) \cup W)$,
			\begin{align*}
			d_{H_A}(e) &= (1 + \O (K D^{\varepsilon - 1} + D^{-\gamma/2} \log N)) D_{L},\\
			d_{H_A}(x) &= (1 + \O (K D^{\varepsilon - 1} + D^{-\gamma/2} \log N)) D_{R},
			\end{align*}
			where $D_L = \Theta (D^{k \gamma})$, and $D_{R} = \Theta ( D^{k \gamma + \frac{1-\gamma}{k-1}} )$. 
			Moreover, the codegree $C_{H_A}$ of $H_A$ is at most $\O (D^{\gamma(k-1)} \log^2 N)$.
		\end{itemize}
		
		Let $C_0 = \Theta (D^{\gamma(k-1)} \log^2 N)$ be an upper bound on $C_{H_A}$. Let $L := E(M)$ and $R := V(H) \setminus (V(M) \cup W)$ be the two parts of $H_A$. Note that by $(\rm M1)_H$ and $(\rm M2)_H$, $|L| = \Theta(N)$ and $|R| = \Theta (N D^{\frac{\gamma-1}{k-1}})$. Let 
		\begin{align*}
		\mathcal{V}' := \{S \setminus (V(M) \cup W) \: : \: S \in \mathcal{V} \cup \{V(H) \} \} \subseteq 2^R.
		\end{align*}
		Now we define an auxiliary $(1,k(k-1))$-partite multi-hypergraph $H_A'$ by taking many vertex-disjoint copies of $R$ as follows. For $1 \leq i \leq \lfloor D_R / D_L \rfloor$, let $R^i := \{ v^i \: : \: v \in R \}$. We may define
		$$ V(H_A') := L \cup (R^1 \cup \dots \cup R^{\lfloor D_R / D_L \rfloor})$$
		and for any edge $e \cup \{v_1 , \dots , v_{k(k-1)} \} \in E(H_A)$, we add
		$e \cup \{v_1^i , \dots , v_{k(k-1)}^i \}$
		to be an edge in $E(H_A')$ for all $1 \leq i \leq \lfloor D_R / D_L \rfloor$. Then by $(\rm M3)_H$ for any $e \in L$,
		\begin{align*}
		    d_{H_A'}(e) & = (1 + \O (KD^{\varepsilon-1} + D^{-\gamma/2} \log N)) D_{L} \lfloor D_R/D_L \rfloor\\
		    & = (1 + \O (KD^{\varepsilon-1} + D^{-\gamma/2} \log N + D_L / D_R)) D_{R}= (1 + \O (D^{\frac{\gamma-1}{k-1}} + D^{-\gamma/2} \log N)) D_{R},
		\end{align*}
		where in the last equation we used $\varepsilon - 1 < -\frac{1}{k-1} < \frac{\gamma-1}{k-1}$ and that $D_L/D_R = \Theta(D^{\frac{\gamma-1}{k-1}})$. For any $v \in R$ and $1 \leq i \leq \lfloor D_R / D_L \rfloor$, we have $d_{H_A'}(v^i) = d_{H_A}(v)$. 
		
		In summary, for any $x \in V(H_A')$,
		$$ d_{H_A'}(x) = (1 + \O (D^{\frac{\gamma-1}{k-1}} + D^{-\gamma/2} \log N)) D_{R}.$$
		Hence $H_A'$ is almost regular. Moreover, the codegree of $H_A'$ is still at most $C_{H_A} \le C_0$.

        Note that for each  $T \in \mathcal{V}'$ we have $T \subseteq R$, and thus we can define $T^1 := \{ v^1 \: : \: v \in T \}$ and
        \begin{align*}
            \mathcal{V}_1' := \{T^1 \: : \: T \in \mathcal{V}' \}.
        \end{align*}
		Now we may apply Lemma~\ref{cor:comb} to $H_A'$.\COMMENT{$\gamma = 1/k$ works since $D_R^{1-1/k} \gg D^{k \gamma (1-1/k)} = D^{\gamma(k-1)}$.} Since $H_A'$ is almost regular and has codegree at most $C_0$, we can find a simple spanning subhypergraph $H_A''$ of $H_A'$ such that every vertex $x$ has degree
		\begin{align*}
		\left (1 - \frac{1}{\log(D_R / C_0)} \right )^s \frac{d_{H_A'}(x)}{\log(D_R / C_0) C_0^{1-\delta} D_R^{\delta}} \pm 8s \sqrt{\frac{(s+1) d_{H_A'}(x)}{C_0^{1-\delta} D_R^{\delta}}}
		\end{align*}
		for some even integer $s \in (1 + 2\binom{k}{2}\delta^{-1} , 3 + 2\binom{k}{2}\delta^{-1})$.
		
		Thus every vertex $x \in V(H_A'')$ has degree $D_{H_A''} \pm \Delta_{H_A''}$ in $H_A''$, where 
		\begin{align*}
		D_{H_A''} &:= \left ( 1 - \frac{1}{\log(D_R / C_0)} \right )^s \frac{D_R^{1-\delta}}{\log(D_R / C_0) C_0^{1-\delta}} = \Theta \left ( \frac{D^{(1-\delta)\left (\gamma + \frac{1-\gamma}{k-1} \right ) }}{\log D \log^{2(1-\delta)} N}   \right ),\\
		\Delta_{H_A''} &:= \Theta \left ( \frac{( D^{\frac{\gamma-1}{k-1}} + D^{-\frac{\gamma}{2}} \log N) D^{(1-\delta) \left ( \gamma + \frac{1-\gamma}{k-1} \right) }}{\log D \log^{2(1-\delta)} N} \right ),
		\end{align*}
		since $\delta \ll \gamma , k^{-1}$. \COMMENT{In order to satisfy the equation above on $\Delta_{H_A''}$, we should have
			$$ \frac{(D^{\frac{\gamma-1}{k-1}} + D^{-\gamma/2} \log N ) D^{(1-\delta)\left ( \gamma + \frac{1-\gamma}{k-1} \right )}}{\log D \log^{2(1-\delta)} N} \gtrsim \sqrt{\frac{D^{(1-\delta) \left ( \gamma + \frac{1-\gamma}{k-1} \right )}}{\log^{2(1-\delta)} N}} $$
			which is equivalent to
			$$ (D^{\frac{\gamma-1}{k-1}} + D^{-\gamma/2} \log N)D^{\frac{1-\delta}{2} \left ( \gamma + \frac{1-\gamma}{k-1} \right )} \gtrsim \log D \cdot \log^{(1-\delta)} N, $$
			where this holds if the exponent on $D$ in the left-hand-side is positive. In particular, this holds when $-\frac{\gamma}{2} + \frac{1-\delta}{2} \left (\gamma + \frac{1-\gamma}{k-1} \right ) > 0$ is satisfied, which holds since $\delta \ll \gamma , k^{-1}$.
		}

		Note that $|V(H_A'')| = \Theta(N)$. Since $D \geq \exp(\log^\mu N)$, we can apply\COMMENT{Since $|V(H_A'')| < N$, we might no longer have $|\mathcal{V}_1'| \le \exp(\log^{3/2}N)$, so we adjust the statement of our theorem by changing $3/2$ to $4/3$} Theorem~\ref{thm:mainsec3} to the $(k(k-1)+1)$-uniform simple hypergraph $H_A''$ (with $\mathcal{V}_1'$ 
		and $\gamma'$ playing the role of $\mathcal V$ and $\gamma$ respectively, and with some 
		$\varepsilon' \in (0,1)$ satisfying $1-\varepsilon' \ll \gamma , 1-\gamma , k^{-1}$  
		 playing the role of $\varepsilon$\COMMENT{We have to choose $\varepsilon$ in Theorem~\ref{thm:mainsec3} larger than $ \max \left\{1 -\frac{\frac{1-\gamma}{k-1}}{(\gamma+\frac{1-\gamma}{k-1})(1-\delta)},1 -\frac{\gamma/4}{(\gamma+\frac{1-\gamma}{k-1})(1-\delta)}\right\}$. This uses that $\log N \ll D$.}) to obtain a matching $M'$ of $H_A''$ and $W' \subseteq V(H_A'')$ such that the following holds for every $T^1 \in \mathcal{V}_1'$ where $T^1 = \{ v^1 \: : \: v \in T \}$ and $T = S \setminus (V(M) \cup W)$ for some $S \in \mathcal{V} \cup \{V(H) \}$.
		\begin{align}
		|T^1 \setminus (V(M') \cup W')| + |T^1 \cap W'| 
		&= \O (|T| D_{H_A''}^{\frac{\gamma' - 1}{k(k-1)}} + |T| D_{H_A''}^{-1} \Delta_{H_A''} + |T| D_{H_A''}^{-\frac{1}{k(k-1)}} D_{H_A''}^{\gamma' \left ( \frac{1}{k(k-1)} - \frac{1}{2} \right )} \log N ) \nonumber \\
		\overset{(\rm M1)_H}&{=} \O (|S|D^{\frac{\gamma-1}{k-1}} D_{H_A''}^{\frac{\gamma' - 1}{k(k-1)}} + |S|D^{\frac{\gamma-1}{k-1}} D_{H_A''}^{-1} \Delta_{H_A''} \nonumber \\
		& \hspace{10mm} + |S|D^{\frac{\gamma-1}{k-1}} D_{H_A''}^{-\frac{1}{k(k-1)}} D_{H_A''}^{\gamma' \left ( \frac{1}{k(k-1)} - \frac{1}{2} \right )} \log N) \nonumber \\
		&= \O (|S|D^{\frac{\gamma-1}{k-1}} D_{H_A''}^{\frac{\gamma' - 1}{k(k-1)}} + |S|D^{\frac{\gamma-1}{k-1}} D_{H_A''}^{-1} \Delta_{H_A''}) \label{eqn:finaluncov},
		\end{align}
		where in \eqref{eqn:finaluncov} we used\COMMENT{Note that $D^{\frac{\gamma-1}{k-1}} D_{H_A''}^{\frac{\gamma' - 1}{k(k-1)}} \gg D^{\frac{\gamma-1}{k-1}} D_{H_A''}^{-\frac{1}{k(k-1)}} D_{H_A''}^{\gamma' \left ( \frac{1}{k(k-1)} - \frac{1}{2} \right )} \log N$ provided $D \geq \exp(\log^\mu N)$} $D \geq \exp(\log^\mu N)$.
		
		Let $M''$ be the subhypergraph of $M'$ induced by $L \cup R^1 = E(M) \cup R^1$. Identifying $R^1$ and $R$, the matching $M''$ can be viewed as a matching of $H_A$ and the number of vertices in $T$ not covered by $M''$ is bounded by~\eqref{eqn:finaluncov}, which is asymptotically at most\COMMENT{$|S|D^{\frac{\gamma-1}{k-1}} D_{H_A''}^{\frac{\gamma'-1}{k(k-1)}}$ is at most $|S| D^{\frac{\gamma-1}{k-1}} D^{\frac{\gamma' - 1}{k(k-1)}(1-\delta)\left ( \gamma + \frac{1-\gamma}{k-1} \right )} \log N$. This simplifies to $|S|D^{-\frac{1}{k-1} - \frac{1}{k(k-1)^2}} D^{\gamma \left ( \frac{1}{k-1} - \frac{1}{k(k-1)} + \frac{1}{k(k-1)^2} \right )} D^{f_{k,\gamma} (\delta,\gamma')} \log N$ if we collect all the terms in $\delta$ and $\gamma'$ into $f_{k,\gamma} (\delta,\gamma')$. This yields the first term of the maximum in~\eqref{eqn:aux_uncovered}. The next two terms of the maximum in~\eqref{eqn:aux_uncovered}, are from $|S|D^{\frac{\gamma-1}{k-1}} D_{H_A''}^{-1} \Delta_{H_A''}$, which is asymptotically $|S|D^{-\frac{1}{k-1}}D^{\frac{2\gamma-1}{k-1}} + |S|D^{-\frac{1}{k-1}}D^{\frac{\gamma}{k-1}-\frac{\gamma}{2}} \log N$, hence we can replace the sum of these two terms by the maximum of both terms asymptotically.}
		\begin{equation}\label{eqn:aux_uncovered}
		|S|D^{-\frac{1}{k-1}} \cdot \O \left ( \max \left ( D^{-\frac{1}{k(k-1)^2}} D^{\gamma \left ( \frac{1}{k-1} - \frac{1}{k(k-1)} + \frac{1}{k(k-1)^2} \right )} D^{f_{k,\gamma} (\delta , \gamma')}, \:\: D^{\frac{2\gamma-1}{k-1}}, \:\: D^{\gamma \left ( \frac{1}{k-1} - \frac{1}{2} \right )}  \right ) \right ) \log N,
		\end{equation}
		where
		\begin{equation}\label{eqn:fkg}
		0 < f_{k,\gamma}(\delta , \gamma') := \frac{(1-\delta)\gamma' + \delta}{k(k-1)} \left (\gamma + \frac{1-\gamma}{k-1} \right ) \leq \frac{ \gamma' +  \delta}{k(k-1)} \leq \frac{\eta_0 - \eta}{2}.
		\end{equation}
		
		Since $\gamma < \frac{1}{2}$, it is straightforward to check that\COMMENT{Multiplying by $k-1$ and rearranging we get 
		$\gamma (1 + \frac{1}{k} - \frac{1}{k(k-1)}) < 1 - \frac{1}{k(k-1)}$. Multiplying again by $k(k-1)$ we get $\gamma(k(k-1)+(k-1)-1) < k(k-1)-1$ which simplifies to
		$\gamma < \frac{k^2-k-1}{k^2-2}$. This inequality holds since $\gamma <  \frac{1}{2}$ and $k > 3$.}
		\begin{equation*}
		    -\frac{1}{k(k-1)^2} + \gamma \left ( \frac{1}{k-1} - \frac{1}{k(k-1)} + \frac{1}{k(k-1)^2} \right ) > \frac{2\gamma - 1}{k-1}.
		\end{equation*}
		Hence~\eqref{eqn:aux_uncovered} is at most
		\begin{equation}\label{eqn:aux_uncovered_rev}
		|S|D^{-\frac{1}{k-1}} \cdot \O \left ( \max \left ( D^{-\frac{1}{k(k-1)^2}} D^{\gamma \left ( \frac{1}{k-1} - \frac{1}{k(k-1)} + \frac{1}{k(k-1)^2} \right )} D^{\frac{\eta_0 - \eta}{2}},  \:\: D^{\gamma \left ( \frac{1}{k-1} - \frac{1}{2} \right )}  \right ) \right ) \log N.
		\end{equation}

		Each $\{e_M \} \cup A \in E(M'')$ corresponds to an augmenting star $(e_M , \{e_1 , \dots , e_k \}) \in \mathcal{A}(H)$, where $A = (e_1 \cup \dots \cup e_k) \setminus e_M$. 
		Thus for each $\{e_M \} \cup A \in E(M'')$ with $A = (e_1 \cup \dots \cup e_k) \setminus e_M$, we can replace $e_M$ in $M$ with the edges $e_1 , \dots , e_k$. Let $M^*$ be the resulting matching. Then $M^*$ covers all the vertices of $V(M) \cup \bigcup_{F \in E(M'')} (F \cap R)$.

		In summary, for each $S \in \mathcal{V} \cup \{V(H) \}$, any vertex of $S$ not covered by our augmented matching $M^*$ is either a vertex of
		$S \cap W$
		or it is a vertex of $T$ not covered by $M''$ where $T = S \setminus (V(M) \cup W)$. Hence the total number of vertices of $S$ not covered by $M^*$ is at most the sum of $|S \cap W|$ (bounded in ~$(\rm M2)_H$) and~\eqref{eqn:aux_uncovered_rev}. This sum is asymptotically at most
		\begin{equation}\label{eqn:aux_uncovered_2}
		|S|D^{-\frac{1}{k-1}} \cdot \O \left ( \max \left ( D^{-\frac{1}{k(k-1)^2}} D^{\gamma \left ( \frac{1}{k-1} - \frac{1}{k(k-1)} + \frac{1}{k(k-1)^2} \right )} D^{\frac{\eta_0 - \eta}{2}}, \:\: D^{\varepsilon - 1 + \frac{1}{k-1}}, \:\: D^{\gamma \left ( \frac{1}{k-1} - \frac{1}{2} \right )} \right ) \right ) \log N.
		\end{equation}
		We now analyse~\eqref{eqn:aux_uncovered_2}. Let 
		\begin{equation*}
		    F(x) := \max \left ( D^{-\frac{1}{k(k-1)^2}} D^{x \left ( \frac{1}{k-1} - \frac{1}{k(k-1)} + \frac{1}{k(k-1)^2} \right )}, \:\: D^{\varepsilon - 1 + \frac{1}{k-1}} \:\: , \:\: D^{x \left ( \frac{1}{k-1} - \frac{1}{2} \right )} \right )
		\end{equation*}
		be defined on $(0,1)$. First we claim that\COMMENT{Moreover, $F(x)$ is minimised at $x = \gamma$, see Section~\ref{sec:graph} in Appendix.}
		\begin{equation}\label{eqn:gamma0}
		F(\gamma) = D^{-\eta_0}.
		\end{equation}

		Indeed, let $\gamma_{1,3} := \frac{2}{k^3-2k^2 -k +4}$ and $\gamma_{2,3} := \frac{2(k-1)(1-\varepsilon)-2}{k-3}$. Then $\gamma = \min(\gamma_{1,3} , \gamma_{2,3})$ by~\eqref{def:gamma}. Note that the first term and the third term of $F(x)$ are equal\COMMENT{$-\frac{1}{k(k-1)^2}+ x \left ( \frac{1}{k-1} - \frac{1}{k(k-1)} + \frac{1}{k(k-1)^2} \right ) = x \left ( \frac{1}{k-1} - \frac{1}{2} \right )$ simplifies to $x(\frac{1}{2} - \frac{1}{k(k-1)}+\frac{1}{k(k-1)^2}) = x(\frac{k(k-1)^2-2(k-1)+2}{2k(k-1)^2}) = x(\frac{(k^3-2k^2+k)-2k+2+2}{2k(k-1)^2}) = x(\frac{k^3-2k^2 -k +4}{2k(k-1)^2}) = \frac{1}{k(k-1)^2}$ which gives that $x = \frac{2}{k^3-2k^2 -k +4}$.} at $x = \gamma_{1,3}$ and the second term and the third term of $F(x)$ are equal\COMMENT{$\varepsilon - 1 + \frac{1}{k-1} = x \left ( \frac{1}{k-1} - \frac{1}{2} \right )= x \left ( \frac{3-k}{2(k-1)}\right )$ simplifies to  $x = \frac{2((\epsilon-1)(k-1)+1)}{3-k} = \frac{2(1 -\epsilon)(k-1)-2}{k-3} $} at $x = \gamma_{2,3}$. Moreover,
		\begin{equation}
		\label{eq:blah}
		 D^{-\frac{1}{k(k-1)^2}} D^{\gamma_{1,3} \left ( \frac{1}{k-1} - \frac{1}{k(k-1)} + \frac{1}{k(k-1)^2} \right )} = D^{\gamma_{1,3} \left ( \frac{1}{k-1} - \frac{1}{2} \right )} = D^{-\frac{k-3}{(k-1)(k^3 - 2k^2 - k + 4)}}.   
		\end{equation}
		We will consider two cases. If $\gamma_{1,3} < \gamma_{2,3}$, then at $x = \gamma =  \gamma_{1,3}$, the first and third term of $F(x)$ are equal and are both larger than the second term because\COMMENT{we used in the inequality that the exponent $\frac{1}{k-1}-\frac{1}{2}$ is negative when $k > 3$ and also that $\gamma_{1,3} < \gamma_{2,3}$} $D^{-\frac{k-3}{(k-1)(k^3 - 2k^2 - k + 4)}} \overset{\eqref{eq:blah}}{=} D^{\gamma_{1,3} \left ( \frac{1}{k-1} - \frac{1}{2} \right )} > D^{\gamma_{2,3} \left ( \frac{1}{k-1} - \frac{1}{2} \right )} = D^{\varepsilon - 1 + \frac{1}{k-1}}.$ Hence ~\eqref{eqn:gamma0} holds in this case by~\eqref{def:eta_0}.
		
		If $\gamma_{1,3} \geq \gamma_{2,3}$, then at $x = \gamma = \gamma_{2,3}$, the second and third term of $F(x)$ are equal and are both at least as large as the first term because\COMMENT{in the first inequality we used that $\frac{1}{k-1} - \frac{1}{k(k-1)} + \frac{1}{k(k-1)^2}>0$ and that $\gamma_{1,3} \geq \gamma_{2,3}$}
		\begin{align*}
		    D^{-\frac{1}{k(k-1)^2}} D^{\gamma_{2,3} \left ( \frac{1}{k-1} - \frac{1}{k(k-1)} + \frac{1}{k(k-1)^2} \right )} &\leq D^{-\frac{1}{k(k-1)^2}} D^{\gamma_{1,3} \left ( \frac{1}{k-1} - \frac{1}{k(k-1)} + \frac{1}{k(k-1)^2} \right )} \overset{\eqref{eq:blah}}{=} D^{\gamma_{1,3} \left ( \frac{1}{k-1} - \frac{1}{2} \right )} \\ &\leq D^{\gamma_{2,3} \left ( \frac{1}{k-1} - \frac{1}{2} \right )} = D^{\varepsilon - 1 + \frac{1}{k-1}}.
		\end{align*}
		Hence equation~\eqref{eqn:gamma0} holds in this case as well by~\eqref{def:eta_0}.

		Thus by~\eqref{eqn:aux_uncovered_2}, the number of vertices in $S \in \mathcal{V} \cup \{V(H) \}$ not covered by $M^*$ is at most
		\begin{equation*}
		|S|D^{-\frac{1}{k-1}} \cdot \O(D^{\frac{\eta_0 - \eta}{2}} \log N) \cdot F(\gamma) \overset{\eqref{eqn:gamma0}}{<} |S| D^{-\frac{1}{k-1} - \eta},
		\end{equation*}
		since $\O (\log N) < D^{\frac{\eta_0 - \eta}{2}}$ by~\eqref{eqn:hierachy} and the fact that $D \geq \exp(\log^\mu N)$. This completes the proof of Theorem~\ref{thm:main2proof}.
	\end{proof}
	
	\section{Nearly perfect matchings in almost regular hypergraphs \texorpdfstring{\\}{} with small codegree}
	\label{sec:KRimproved}
	Using Lemma~\ref{cor:comb} and Theorem~\ref{thm:main2proof}, in this section we prove the following general theorem which implies Theorem~\ref{thm:main3intro} by setting $\mathcal{V} := \{V(H)\}$.\COMMENT{It is worth considering if we should move the proof of Theorem ~\ref{thm:main3proof} and/or proof of Lemma~\ref{cor:comb} to the Appendix.} 	Theorem \ref{thm:main3proof} sharpens a result of Alon and Yuster \cite{AY2005} under slightly stronger assumptions. It is used in Section~\ref{sec:ChromaticIndex} to prove new bounds on the chromatic index of hypergraphs with small codegree.
	
	\begin{theorem}\label{thm:main3proof}
		Let $k > 3$, $D,N$ be integers, $C,K > 0$, $\varepsilon \in (0,  1 - \frac{1}{k-1})$ and  $0 < \gamma,\mu < 1$ be real numbers.\COMMENT{In Vu's paper, his theorem can be only applied when $\varepsilon < 1 - \frac{1}{k-1}$ as in our theorem. Maybe the theorem won't be true if $\varepsilon > 1 - \frac{1}{k-1}$; it would be nice if there is an example to show this.} Let $\eta^{\diamond} := \frac{k-3}{(k-1)(k^3 - 2k^2 - k + 4)}$ and
		\[\eta_0 := 
		\begin{cases}
		\min \left ( \eta^{\diamond} \: , \: 1 - \frac{1}{k-1} - \varepsilon \right )& \text{if }\varepsilon > 1/2, \\
		\eta^{\diamond}& \text{otherwise}
		\end{cases}
		\]
		and $\eta \in (0,\eta_0)$. Then there exists $N_0 = N_0(k , K , \gamma , \varepsilon ,\eta, \mu)$ such that for any integer $N \geq N_0$, the following holds, provided that $D \geq \exp(\log^\mu N)$ and $C \leq D^{1-\gamma}$.
		
		Let $H$ be a $k$-uniform
		$(D \pm KC^{1-\varepsilon}D^\varepsilon)$-regular
		multi-hypergraph on $N$ vertices with codegree at most $C$. Let $\mathcal{V}$ be a collection of subsets of $V(H)$ such that $|S| \geq \sqrt{D/C} \log N$ for each $S \in \mathcal{V}$ and $|\mathcal{V}| \leq \exp(\log^{4/3} N)$.
		
		Then there is a matching in $H$ that covers all but at most $|S| (D/C)^{-\frac{1}{k-1} - \eta}$ vertices in $S$, for every $S \in \mathcal{V}$.
	\end{theorem}
	\begin{proof}
        Choose $N_0 \in \mathbb{N}$ and a new constant $\delta$ such that  
$0 < N_0^{-1} \ll \delta \ll \gamma, \eta , \eta_0-\eta, \varepsilon,  1- \frac{1}{k-1} - \varepsilon , \mu , k^{-1}, K^{-1}$ and such that $\delta < 2 \varepsilon -1$ if $\varepsilon > 1/2$. Define
		\[\widetilde{\eta_0} := 
		\begin{cases}
		\min \left ( \eta^{\diamond} \: , \: 1 - \frac{1}{k-1} - \frac{\varepsilon - \delta/2}{1 - \delta/2} \right )& \text{if }\varepsilon > 1/2, \\
	\eta^{\diamond} & \text{otherwise}.
		\end{cases}
		\]
		Our choice of $\delta$ implies that\COMMENT{Since $\frac{\frac{1}{k-1} + \eta}{1-\delta} - \frac{1}{k-1}$ tends to $\eta$ and $\widetilde{\eta_0}$ tends to $\eta_0$ as $\delta \to 0$, this is easy to verify by the continuity of the functions with respect to $\delta$ near 0.}
		\begin{equation}\label{eqn:assump_delta}
		\frac{\varepsilon - \delta/2}{1 - \delta/2} < 1 - \frac{1}{k-1} \:\:\:\text{and}\:\:\:\frac{\frac{1}{k-1} + \eta}{1-\delta} - \frac{1}{k-1} < \widetilde{\eta_0}.
		\end{equation}

		In the proof below, the implicit constants in $\O(\cdot)$ and  $\Theta(\cdot)$ will only depend on the parameters $\delta,\gamma,\varepsilon,\mu,k,K$.
		
		We divide the proof into two cases depending on whether $\log D \leq C \leq D^{1-\gamma}$ or $C < \log D$. 
		
		First consider the case $\log D \leq C \leq D^{1-\gamma}$. By applying Lemma~\ref{cor:comb} with $\delta/2$ playing the role of $\delta$, there exists a simple $k$-uniform $N$-vertex subhypergraph $H_{\rm simp}$ of $H$ such that for any $v \in V(H)$,
		\begin{equation}\label{eqn:degsimp}
		d_{H_{\rm simp}}(v) = \left (1 - \frac{1}{\log(D/C)} \right )^s \left ( \frac{D^{1-\delta/2}}{\log(D/C) C^{1-\delta/2}} \pm \frac{KD^{\varepsilon - \delta/2}}{\log(D/C) C^{\varepsilon-\delta/2}} \right ) \pm 8s \sqrt{\frac{(s+1)d_H (v)}{C^{1-\delta/2} D^{\delta/2}}},
		\end{equation}
		where $s \in (1 + 4 \binom{k}{2} \delta^{-1}\: , \:3 + 4\binom{k}{2}\delta^{-1})$. In particular, $H_{\rm simp}$ is $(D^* \pm \Delta^*)$-regular, where
		\begin{equation}\label{eqn:deg_estimate}
		D^* = \Theta \left ( \frac{(D/C)^{1-\delta/2}}{\log D} \right ) \:\:\:\: \text{and} \:\:\:\:
		\Delta^* = \Theta \left ( \max \left (\frac{(D/C)^{\varepsilon - \delta/2}}{\log D} \: , \: (D/C)^{\frac{1-\delta/2}{2}}\right ) \right ).
		\end{equation}
		
		Since $\delta < 2\varepsilon - 1$ if $\varepsilon > 1/2$, we have
		\[\Delta^* =
		\begin{cases}
		\Theta( (D/C)^{\frac{1-\delta/2}{2}})& \text{for }\varepsilon \leq 1/2, \\
		\Theta \left ( \frac{(D/C)^{\varepsilon - \delta/2}}{\log D} \right )&\text{otherwise.}
		\end{cases}
		\]
		Let $\eta^* := \frac{\frac{1}{k-1} + \eta}{1-\delta} - \frac{1}{k-1}$. Then $\eta^* < \widetilde{\eta_0}$ by~\eqref{eqn:assump_delta}. We also let
		\[\varepsilon^* =
		\begin{cases}
		\frac{1}{2} + \frac{\delta}{4}& \text{for }\varepsilon \leq 1/2, \\
		\frac{\varepsilon - \delta/2}{1 - \delta/2}&\text{otherwise.}
		\end{cases}
		\]
		Then we have $\Delta^* \leq (D^*)^{\varepsilon^*}$. In either case, we have
		$\widetilde{\eta_0} = \min \left ( \eta^{\diamond} \: , \: 1 - \frac{1}{k-1} - \varepsilon^* \right ).$ Moreover, \eqref{eqn:assump_delta} implies that $\varepsilon^* < 1 - \frac{1}{k-1}$. Since $\eta^* < \widetilde{\eta_0}$ and $|S| \geq \sqrt{D^*} \log N$ for each $S \in \mathcal{V}$, we may apply Theorem~\ref{thm:main2proof} to $H_{\rm simp}$ (with $\varepsilon^*$, $\eta^*$ and $\mu/2$ playing the roles of $\varepsilon$, $\eta$ and $\mu$), to obtain a matching in $H_{\rm simp}$ which covers all but at most
		\begin{equation}\label{eqn:final_uncov_cod}
		|S| (D^*)^{-\frac{1}{k-1} - \eta^*} < |S| (D/C)^{-(1-\delta)\left ( \frac{1}{k-1} + \eta^* \right )} \leq |S| (D/C)^{-\frac{1}{k-1} - \eta}
		\end{equation}
		vertices in $S$, for every $S \in \mathcal{V}$.
		
		For the case $C < \log D$, we may enlarge the codegree bound $C$ to $\log D$ and apply the analogous statements as above, where $D^* = \Theta \left ( \frac{D^{1 - \delta/2}}{(\log D)^{2 - \delta/2}} \right )$. As in~\eqref{eqn:final_uncov_cod}, the hypergraph $H_{\rm simp}$ contains a matching covering all but at most
		\begin{equation*}
		|S| (D^*)^{-\frac{1}{k-1} - \eta^*} < |S|D^{-(1 - \delta)\left ( \frac{1}{k-1} + \eta^* \right )} \leq |S|D^{-\frac{1}{k-1} - \eta}
		\end{equation*}
		vertices in $S$, for every $S \in \mathcal{V}$.
	\end{proof}
	
	\section{Chromatic index of hypergraphs with small codegree}
	\label{sec:ChromaticIndex}
	
	Before proving Corollary~\ref{cor:chromindex}, we need the following lemma which allows us to embed a given hypergraph into an almost regular hypergraph with not too many vertices while preserving both the maximum degree and codegree. This allows us to apply Theorem~\ref{thm:main3proof} (which applies only to almost regular hypergraphs).\COMMENT{ It is important for us that $|V(H')|$ is small in the lemma below as we will see in the proof of Corollary 1.8. Otherwise there is a trivial inductive proof to construct $H'$. We do not know if such a lemma with a good bound on $|V(H')|$ is known.}
	
	\begin{lemma}\label{lem:embed}
	For any integer $k \geq 3$, there exist $N_0 = N_0(k)$,  $D_0 = D_0(k)$, and $K = K(k)$ satisfying the following. For any $N \geq N_0$, $D \geq D_0$ and $C \in \mathbb{N}$, let $H$ be an $N$-vertex $k$-uniform multi-hypergraph with maximum degree at most $D$ and codegree at most $C$. There exists a $k$-uniform multi-hypergraph $H'$ such that 
	\begin{itemize}
	    \item $H \subseteq H'$, 
	    \item every vertex $v \in V(H')$ has degree between $D - K$ and $D$,
	    \item $H'$ has codegree at most $C$,
	    \item $|V(H')| = (k-1)^2 D^2 N$.
	\end{itemize}
	 Moreover, if $H$ has no multiple edges then $H'$ also has no multiple edges.
	\end{lemma}

Since it is straightforward to prove Lemma~\ref{lem:embed} based on the existence of Steiner systems with parameters $(2,k,n)$
(which was proved by Wilson~\cite{wilson1972_1, wilson1972_2, wilson1975}), we present its proof in Appendix~\ref{appendix:embed}. 

Now we are ready to
deduce Corollary~\ref{cor:chromindex} as follows: the existence of the desired edge-colouring of $H$ follows from the existence of a suitable matching in an auxiliary `incidence' hypergraph $H'$. We find such a matching via Theorem~\ref{thm:main3proof}.

	\begin{proof}[Proof of Corollary~\ref{cor:chromindex}]
	We may assume that $D$ is an integer. By Lemma~\ref{lem:embed}, there exists a $k$-uniform hypergraph $H'$ such that $H \subseteq H'$, every vertex of $H'$ has degree between $D - K$ and $D$ for some $K = K(k)$, the codegree of $H'$ is at most $C$, and\COMMENT{Here we just express  both~\eqref{eq:boundVH'} and~\eqref{eqn:verticesh0} by polynomials in $D$ and $N$, which suffices for our purpose. We will only need the bound $\log |V(H_0)| =  \O(\max(\log N , \log D))$ which follows from~\eqref{eqn:verticesh0}.}
	\begin{equation}
	\label{eq:boundVH'}
	   |V(H')| = (k-1)^2 D^2 N.
	\end{equation}
	Hence it suffices to bound the chromatic index of $H'$ as it is at least the chromatic index of $H$.
	
	Now let us define a $(k+1)$-uniform auxiliary hypergraph $H_0$ as follows. The vertex-set $V(H_0)$ is the disjoint union $E(H') \cup V^1 \cup \dots \cup V^{D}$, where $V^i := \{v^i \: : \: v \in V(H') \}$ is a copy of $V(H')$ for $1 \leq i \leq D$. The edge-set $E(H_0)$ is defined as
	\begin{align*}
	    E(H_0) := \{ e v_1^i v_2^i \dots v_k^i \: : \: e = \{v_1,\dots,v_k \} \in E(H') \: , \: 1 \leq i \leq D \}.
	\end{align*}
	Then it is straightforward to check that every $e \in E(H')$ has degree exactly $D$ in $H_0$, and for every $v \in V(H')$ and $1 \leq i \leq D$, we have $d_{H_0}(v^i) = d_{H'} (v)$. Moreover, $H_0$ has codegree at most $C$ and 
	\begin{equation}\label{eqn:verticesh0}
	|V(H_0)| \leq 2D |V(H')| \overset{\eqref{eq:boundVH'}}{=} 2(k-1)^2 D^3 N.
	\end{equation}

	Also note that every matching $M$ of $H_0$ corresponds to a partial edge-colouring of $H'$ with $D$ colours by colouring, for each edge $e v_1^i \dots v_k^i \in E(M)$,
	the edge $e \in E(H')$ with colour $i$.
	
	Let $\mathcal{V} := \{ \{v^1 , \dots , v^D \} \: : \: v \in V(H') \}$ and $\eta' := \frac{\eta_0 + \eta}{2}$, where $\eta_0 := \frac{k-2}{k(k^3 + k^2 - 2k + 2)}$. Note that $|\mathcal{V}| = |V(H')|$, so $|\mathcal{V}| \le \exp(\log^{4/3} |V(H_0)|)$. Also note that for every $S \in \mathcal{V}$, by ~\eqref{eqn:verticesh0}, the inequality $|S| \geq \sqrt{D/C} \log |V(H_0)|$ is satisfied since $|S| = D$ and we assumed that $D \geq \exp(\log^{\mu} N)$.\COMMENT{Note that $|S|= D \geq \sqrt{D/C} \log |V(H_0)|$ holds, if $ \log |V(H_0)| \le \sqrt{D}$. As $\log |V(H_0)| \overset{~\eqref{eqn:verticesh0}}{=}  \O(\max(\log N , \log D))$, the only non-trivial part to check is that $\log N \ll \O(\sqrt{D})$ but this holds by our assumption that $D \geq \exp(\log^\mu N)$.} Also note that this assumption and ~\eqref{eqn:verticesh0} imply that
	$D \geq \exp(\log^{\mu/2} |V(H_0)|)$.\COMMENT{We know $D \geq \exp(\log^\mu N)$ by our assumption, which implies $D \geq \exp(\log^{\mu/2} |V(H_0)|)$, since we have $\log^{\mu/2} |V(H_0)| = \O(\max(\log^{\mu/2} N , \log^{\mu/2} D))$ by~\eqref{eqn:verticesh0} and $\mu \in (0,1)$.} So we can apply Theorem~\ref{thm:main3proof} to $H_0$ (with $\mu/2$, $\eta'$ and 1/4 playing the roles of $\mu$, $\eta$ and $\varepsilon$), to obtain a matching $M$ that covers all but at most $D(D/C)^{-1/k - \eta'}$ vertices in $S$, for every $S \in \mathcal{V}$.  In other words, there exists a partial edge-colouring of $H'$ with $D$ colours, where each vertex $v \in V(H')$ is incident to at most $D(D/C)^{-1/k - \eta'}$ non-coloured edges. Let $H''$ be the subhypergraph of $H'$ consisting only of non-coloured edges of $H'$; thus $H''$ has maximum degree at most $D(D/C)^{-1/k - \eta'}$. Since every edge of $H''$ shares a vertex with at most $k(D(D/C)^{-1/k - \eta'} - 1)$ other edges, $H''$ has chromatic index at most
	\begin{equation*}
	    k(D(D/C)^{-1/k - \eta'} - 1) + 1 \leq kD(D/C)^{-1/k - \eta'} < D(D/C)^{-1/k - \eta}.
	\end{equation*}
	Hence, to colour $H''$ one needs less than $D(D/C)^{-1/k - \eta}$ new colours (different from the $D$ colours used for the partial edge-colouring of $H'$),  implying that $H'$ has chromatic index at most $D + D(D/C)^{-1/k - \eta}$, as desired.
	\end{proof}

\section*{Acknowledgements}
We are grateful to the anonymous referee for their thoughtful comments and suggestions.
		
\bibliographystyle{amsplain}

\providecommand{\bysame}{\leavevmode\hbox to3em{\hrulefill}\thinspace}
\providecommand{\MR}{\relax\ifhmode\unskip\space\fi MR }
\providecommand{\MRhref}[2]{%
  \href{http://www.ams.org/mathscinet-getitem?mr=#1}{#2}
}
\providecommand{\href}[2]{#2}

\appendix
\section{Finding simple subhypergraphs of hypergraphs with small codegrees}
\label{appendix:simplesubhypergraphs}
	
	We will use the following four lemmas taken directly from Lemma 7  and the proofs of Lemmas 4--6 in \cite{KR1998}, where they are stated for hypergraphs, but they are also applicable to multi-hypergraphs, since the proofs are based on random edge-colourings.
	
	\begin{lemma}\label{lem:part1}
		Let $k \geq 3$ be an integer, and $\alpha , \beta, D > 0$, $0 <\gamma < 1$ be real numbers satisfying $D^{-1} \ll \alpha , \beta , \gamma , k^{-1}$, and let $C$ be an integer satisfying $\log D \leq C \leq D^{1-\gamma}$.
		
		Let $H$ be a $k$-uniform multi-hypergraph with codegree at most $C$ and $\alpha D \leq d_H(v) \leq \beta D$ for any $v \in V(H)$. Then there exists $E' \subseteq E(H)$ such that the multi-hypergraph $H' := (V(H) , E')$ satisfies the following.
		\begin{itemize}
			\item[$(\rm 1)$] The codegree of $H'$ is at most $\log(D/C)$.
			\item[$(\rm 2)$] $d_{H'}(v) = \frac{d_H(v)}{C} \pm 4\sqrt{\frac{d_H(v)}{C} \log D}$.
		\end{itemize}
	\end{lemma}

	\begin{lemma}\label{lem:part2}
		Let $k \geq 3$ be an integer, and let $\alpha,\beta, D > 0$, $0 < \delta < 1/3$ be real numbers satisfying $D^{-1} \ll \alpha , \beta , \delta , k^{-1}$.
		
		Let $H$ be a $k$-uniform multi-hypergraph with codegree at most $\log D$ and $\alpha D \leq d_H(v) \leq \beta D$ for any $v \in V(H)$. Then there exists $E' \subseteq E(H)$ such that $H' := (V(H) , E')$ satisfies the following.
		\begin{itemize}
			\item[$(\rm 1)$] The codegree of $H'$ is at most $2 \delta^{-1}$.
			\item[$(\rm 2)$] $d_{H'}(v) = \frac{d_H(v)}{D^\delta} \pm 4\sqrt{\frac{d_H(v)}{D^\delta} \log D}$.
		\end{itemize}
	\end{lemma}

	\begin{lemma}\label{lem:part3}
		Let $k \geq 3$ be an integer, and let $\alpha,\beta, D > 0$ be real numbers satisfying $D^{-1} \ll \alpha , \beta , k^{-1}$. Let $s$ and $T$ be integers satisfying $2k \leq s \leq \log^2 D$ and $s (\log D)^{1/2} \leq T < D^{1/3}$.
		
		Let $H$ be a $k$-uniform multi-hypergraph such that for any $v \in V(H)$, $\alpha D \leq d_H(v) \leq \beta D$ and let $E_H(v)$ be the set of edges $e \in E(H)$ containing $v$. Let $G$ be an $s$-regular graph with vertex set $E(H)$. Then there exists $E' \subseteq E(H)$ such that for any vertex $v$ of the multi-hypergraph $H' := (V(H) , E')$, the number of $e \in E_H(v)$ that are isolated vertices in $G[E']$ is
		\begin{equation*}
		\frac{d_H(v)}{T} \left (1 - \frac{1}{T} \right )^s \pm 4s \sqrt{\frac{(s+1)d_H(v)}{T} \log D}.
		\end{equation*}
	\end{lemma}

	\begin{lemma}\label{lem:regular}
		Let $G$ be a graph with maximum degree $D$. If $|V(G)| > 2D^3$, then for any even integer $d$ with $D < d < |V(G)|$, there is a $d$-regular graph $G'$ containing $G$ such that $V(G') = V(G)$.\COMMENT{The lemma above is stated in~\cite{KR1998} without proof. We might consider including a reference to the proof but we could not find it.}
	\end{lemma}
	
	Combining the four lemmas above, we are now ready to prove Lemma \ref{cor:comb}.

	\begin{proof}[Proof of Lemma \ref{cor:comb}]
		Let $H$ be a $k$-uniform multi-hypergraph satisfying all the conditions of the corollary. By Lemma~\ref{lem:part1}, there exists $E' \subseteq E(H)$ such that the $k$-uniform multi-hypergraph $H' := (V(H) , E')$ has codegree $C' \leq \log(D/C)$, and for any $v \in V(H')$, 
		\begin{equation}\label{eqn:degh'}
		d_{H'}(v) = \frac{d_H(v)}{C} \pm 4\sqrt{\frac{d_H(v)}{C} \log D}.
		\end{equation}\COMMENT{Since $C \leq D^{1-\gamma}$, note that
		\begin{equation*}\label{eqn:compareh'}
	 4\sqrt{\frac{d_H(v)}{C} \log D} \leq \frac{d_H(v)}{2C}.
		\end{equation*}}
		Then $\frac{\alpha D}{2C} \le \frac{d_H(v)}{2C}  \le d_{H'}(v) \le \frac{3d_H(v)}{2C} \le \frac{\beta D}{2C}$, so we may apply Lemma~\ref{lem:part2} to the multi-hypergraph $H'$ (where $D/C$ plays the role of $D$ in Lemma~\ref{lem:part2}) to show that there exists $E'' \subseteq E'$ such that the $k$-uniform multi-hypergraph $H'' := (V(H) , E'')$ has codegree at most $2\delta^{-1}$ and for any $v \in V(H'')$, 
		\begin{align}
		d_{H''}(v) &= \frac{d_H(v)}{C^{1-\delta} D^\delta} \pm \frac{4C^{\delta}}{D^\delta} \sqrt{\frac{d_H(v)}{C} \log D} \pm 4 \sqrt{ \frac{1.5 d_H(v)}{C^{1-\delta} D^\delta} \log (D/C)} \nonumber \\
		&= \frac{d_H(v)}{C^{1-\delta} D^\delta} \pm 10 \sqrt{ \frac{ d_H(v)}{C^{1-\delta}D^\delta} \log (D/C)}.
		\label{eqn:degh''}
		\end{align}
\COMMENT{Since $C \leq D^{1-\gamma}$, note that
		\begin{equation*}\label{eqn:compareh''}
		10 \sqrt{ \frac{d_H(v)}{C^{1-\delta}D^\delta} \log (D/C)} \leq 	\frac{d_H(v)}{2C^{1-\delta} D^\delta}.
		\end{equation*}}
		Let $G$ be a graph with $V(G) := E''$ such that $e_1 \ne e_2 \in E''$ are adjacent if and only if $|e_1 \cap e_2| \geq 2$. Since the codegree of $H''$ is at most $2\delta^{-1}$, we have $\Delta(G) \le 2 \binom{k}{2} \delta^{-1}$. Hence by Lemma~\ref{lem:regular} there exists an $s$-regular graph $G_{\rm reg}$ with vertex set $E''$ such that $G_{\rm reg}$ contains $G$, where $s$ is an even integer between $1 + 2\binom{k}{2}\delta^{-1}$ and $3 + 2\binom{k}{2} \delta^{-1}$.\COMMENT{since $|V(G)| = |E''|$ is much larger than $2 \binom{k}{2}\delta^{-1}$}
		
		Applying Lemma~\ref{lem:part3} to $H''$ (with $(D/C)^{1-\delta}$, $\log(D/C)$, $G_{\rm reg}$ playing the role of $D$, $T$ and $G$ in the statement of Lemma~\ref{lem:part3}), we obtain
		$E^* \subseteq E''$ such that the number of edges $e \in E_{H''}(v)$ that are isolated vertices in $G_{\rm reg}[E^*]$ is 
		\begin{align*}
		&\left (1 - \frac{1}{\log(D/C)} \right )^s \frac{ d_{H''}(v)}{\log(D/C)} \pm 4s \sqrt{(s+1) d_{H''}(v)} \\
		\overset{\eqref{eqn:degh''}}&{=} \left (1 - \frac{1}{\log(D/C)} \right )^s \frac{ d_H(v)}{\log(D/C)C^{1-\delta}D^\delta} \pm 8s \sqrt{\frac{(s+1)d_H(v)}{C^{1-\delta}D^\delta}}.
		\end{align*}

Let $E_{\rm simp}$ be the collection of $e \in E^*$ that are isolated in $G_{\rm reg}[E^*]$. Then by the definition of $G_{\rm reg}$, the $k$-uniform hypergraph $H_{\rm simp} := (V(H) , E_{\rm simp})$ is a simple hypergraph satisfying $(\rm i)$ and $(\rm ii)$, as desired.
	\end{proof}
	
\section{Almost independence of events}
\label{appendix:almostindependence}
\begin{proof}[Proof of Proposition~\ref{lem:almost_indep}]
	Let $s$ be an arbitrary integer with $1 \le s \le n_1 + n_2 \leq 6k$. If $F_1, F_2, \ldots, F_{s}$ are disjoint edges in $H_{i-1}$, then 
		\begin{align}
		\mathbb{P}\left (\bigwedge_{j = 1}^{s} (F_j \in E(M_i)) \right ) &= \left ( \frac{1}{D_{i-1}} \right )^s \left (1 - \frac{1}{D_{i-1}} \right )^{t(F_1,\dots,F_s)}, \label{eqn:probwedge}\\
		\prod_{j = 1}^{s} \mathbb{P}(F_j \in E(M_i)) &= \left ( \frac{1}{D_{i-1}} \right )^s \left (1 - \frac{1}{D_{i-1}} \right )^{t(F_1)+\dots+t(F_s)}, \label{eqn:probprod}
		\end{align}
		where $t(F_1 , \dots , F_s)$ is the number of edges in $E(H_{i-1}) \setminus \{ F_1 , \dots , F_s \}$ intersecting at least one of $F_1 , \dots , F_s$, and $t(F_j)$ is the number of edges in $E(H_{i-1}) \setminus \{ F_j \}$ intersecting $F_j$ for $1 \leq j \leq s$. 
		
		Since $H_{i-1}$ is simple, we have $0 \leq t(F_1) + \dots + t(F_s) - t(F_1 , \dots , F_s) \leq s^2 k^2 \leq 36k^4$ as there are at most $k^2$ edges intersecting both  $F_{j_1}$ and $F_{j_2}$ for $1 \leq j_1 \ne j_2 \leq s$. Comparing~\eqref{eqn:probwedge} and~\eqref{eqn:probprod},
		\begin{align}
		\mathbb{P}\left (\bigwedge_{j = 1}^{s} (F_j \in E(M_i)) \right ) &= \left (1 - \frac{1}{D_{i-1}} \right )^{t(F_1,\dots,F_s) - (t(F_1)+\dots+t(F_s))} \prod_{j = 1}^{s} \mathbb{P}(F_j \in E(M_i)) \nonumber \\
		&= \left (1 \pm \frac{36k^4 + 1}{D_{i-1}} \right ) \prod_{j = 1}^{s} \mathbb{P}(F_j \in E(M_i)).\label{eq:comparingFi}
		\end{align}

		Let $z_1, z_2, \ldots, z_s \in U_{i-1}$ be distinct vertices.
		Then 
		\begin{equation}
		S := \prod_{j=1}^{s} \mathbb{P}(z_j \in V(M_i)) = \sum_{(F_1, F_2, \ldots, F_s)} \prod_{j=1}^s \mathbb{P}(F_j \in E(M_i)) \label{eqn:sumnormal},
		\end{equation}
		where the sum here runs over all $s$-tuples of edges $(F_1, F_2, \ldots, F_s)$ such that $z_j \in F_j$. By $(\rm L1)_{i-1}$ and Proposition~\ref{claim:matching},
		\begin{itemize}
			\item[$(\rm B1)$] there are $\Theta(D_{i-1}^s)$ such $s$-tuples $(F_1 , \dots , F_s)$, 
			\item[$(\rm B2)$] $\prod_{j=1}^{s} \mathbb{P}(F_j \in E(M_i)) = \Theta(D_{i-1}^{-s})$ for each $s$-tuple $(F_1 , \dots , F_s)$, and thus
			\item[$(\rm B3)$]  $S = \Theta_{k}(1)$.
		\end{itemize}
		On the other hand, 
		\begin{equation}
		S^* := \mathbb{P} \left (
		\bigwedge_{j=1}^{s} (z_j \in V(M_i)) \right ) = \sum^*_{(F_1, F_2, \ldots, F_s)}\mathbb{P} \left (
		\bigwedge_{j=1}^{s} (F_j \in E(M_i)) \right), \label{eqn:sum*}
		\end{equation}
		where the sum here runs over all  $s$-tuples of edges $(F_1, F_2, \ldots, F_s)$ such that $z_j \in F_j$ where $F_j \cap F_{j'} = \emptyset$ whenever $F_j \not = F_{j'}$.

		\begin{claim}
		For any $s$ with $1 \le s \le n_1 + n_2$, $S^* = S (1 + \O(D_{i-1}^{-1})).$
		\end{claim}
		
		\begin{subproof}
		
		Every $s$-tuple $(F_1, F_2, \ldots, F_s)$ with the property that $z_j \in F_j$ for $j \in [s]$ falls into one of the following three types: Let $S_1$ be the set of $s$-tuples where $F_1, F_2, \ldots, F_s$ are disjoint and distinct,
	let $S_2$ be the set of $s$-tuples where $F_1, F_2, \ldots, F_s$ are distinct but $F_j \cap F_{j'} \ne \emptyset$ for some distinct $j, j' \in [s]$, and let $S_3$ be the set of $s$-tuples where $F_j = F_{j'}$ for some distinct $j, j' \in [s]$. 
	
		Note that $|S - S^*| \le Q_1+Q_2+Q_3$ where $$Q_i = \left|\sum_{(F_1, F_2, \ldots, F_s) \in S_i} \prod_{j=1}^s \mathbb{P}(F_j \in E(M_i)) - \sum^*_{(F_1, F_2, \ldots, F_s) \in S_i}\mathbb{P} \left (
		\bigwedge_{j=1}^{s} (F_j \in E(M_i)) \right)\right|.$$

		Let us first estimate $Q_2$.  There are no terms corresponding to $s$-tuples of $S_2$ in $\sum^*$ but they exist in $\sum$. Hence by $(\rm B2)$, we have $Q_2 \le |S_2| \cdot \Theta(D_{i-1}^{-s}).$ Since $|S_2| \le \binom{s}{2}k \cdot (D_{i-1} \pm \Delta_{i-1})^{s-1} = \O(D_{i-1}^{s-1})$, we obtain  $Q_2 \le \O(D_{i-1}^{-1})$.

		Now let us estimate $Q_3$. An edge is called \emph{special} if it contains at least two of the vertices $z_1, z_2, \ldots, z_s$. The number of special edges is at most $s^2 \le 36 k^2$. Let $T_{\ell}$ be the set of $s$-tuples $(F_1, F_2, \ldots, F_s)$ in $S_3$ such that there are exactly $\ell$ distinct edges among $F_1, F_2, \ldots, F_s$. Then $|T_{\ell}| = \O(D_{i-1}^{\ell - 1}) $ -- indeed, one of these $\ell$ edges is a special edge, so there are at most $s \cdot 36k^2 \leq 216k^3$ choices for this edge (and its position), and the remaining $\ell - 1$ edges (and their positions) can be chosen in at most $\binom{s}{\ell - 1} (D_{i-1} \pm \Delta_{i-1})^{\ell - 1} = \O(D_{i-1}^{\ell-1})$ ways.\COMMENT{since for each of those edges $F$ there is a vertex $z \in \{z_1, z_2, \ldots, z_s\}$ such that $z \in F$.} Finally, having chosen the $\ell$ edges and their positions, there are $\O(1)$ ways to complete this into an $s$-tuple in $T_{\ell}$. For any such $s$-tuple in $T_{\ell}$, by Proposition~\ref{claim:matching},\COMMENT{There is a subtlety here. The term in $\sum$ would be at most $\O(D_{i-1}^{-s})$ and the term in $\sum^*$ would be at most $\O(D_{i-1}^{-\ell})$, so the difference is at most $\O(D_{i-1}^{-\ell})$, as $\ell < s$.} 
		
		\begin{equation}
		  \left|\prod_{j=1}^s \mathbb{P}(F_j \in E(M_i)) - \mathbb{P} \left (
		\bigwedge_{j=1}^{s} (F_j \in E(M_i)) \right)\right| = \O(D_{i-1}^{-\ell}).
		\end{equation}

		Thus, 
		$$ Q_3 = \sum_{\ell = 1}^{s-1} |T_{\ell}| \cdot \O(D_{i-1}^{-\ell}) =  \O(D_{i-1}^{-1}).$$

		Finally, for any $s$-tuple $(F_1, F_2, \ldots, F_s)$ in $S_1$, by \eqref{eq:comparingFi}, we know that $$\left|\prod_{j=1}^s \mathbb{P}(F_j \in E(M_i)) -\mathbb{P} \left (
		\bigwedge_{j=1}^{s} (F_j \in E(M_i)) \right )\right| = \O(D_{i-1}^{-1}) \prod_{j=1}^s \mathbb{P}(F_j \in E(M_i)).$$ Therefore, $Q_1 \le \O(D_{i-1}^{-1})|S|$. Thus $|S - S^*| \le Q_1+Q_2+Q_3 \le \O(D_{i-1}^{-1}) |S|$. This combined with $(\rm B3)$ proves the claim.
		\end{subproof}
		
		Now let $t$ be an arbitrary integer with $1 \le t \le n_1 \le 3k$. Since the event $w_j \in W_i$ is independent of all the other events, it follows that for distinct vertices $w_1, \dots , w_t \in U_{i-1}$ (where a vertex $w_j$ is allowed to be contained in $\{z_1 , \dots , z_s\}$),
		\begin{align}
		\mathbb{P} \left (
		\bigwedge_{j=1}^{s} (z_j \in V(M_i)) \wedge \bigwedge_{j=1}^{t} (w_j \in W_i) \right ) 
		&= (1 + \O(D_{i-1}^{-1}))\prod_{j=1}^{s} \mathbb{P}(z_j \in V(M_i)) \prod_{j=1}^{t} \mathbb{P}(w_j \in W_i) \nonumber\\
		&= \prod_{j=1}^{s} \mathbb{P}(z_j \in V(M_i))\prod_{j=1}^{t} \mathbb{P}(w_j \in W_i) + \O(D_{i-1}^{-1} ). \label{eq:comparing}
		\end{align}
		We are now ready to prove the lemma. For $x_j \in U_{i-1}$, the event that $x_j \in U_i$ is equivalent to the event that $x_j \not \in V(M_i)$ and $x_j \not \in W_i$. Then by the inclusion-exclusion principle the desired probability $\mathbb{P} \left (\bigwedge_{j=1}^{n_1}(x_j \in U_i) \:\wedge\:
		\bigwedge_{j=1}^{n_2} (y_j \in V(M_i))\right )$ can be written as the sum and difference of probabilities of conjunctions of various subsets of the set of events $\{x_j \in V(M_i), x_j \in W_i \mid 1 \le j \le n_1 \} \cup \{y_j \in V(M_i) \mid 1 \le j \le n_2\}$. The number of these terms is at most $3^{2n_1+n_2} = \O(1)$, so using \eqref{eq:comparing} for each of these terms, we obtain that the desired probability is equal to $\prod_{j=1}^{n_1} \mathbb{P}(x_j \in U_i) \cdot \prod_{j=1}^{n_2} \mathbb{P}(y_j \in V(M_i))+ \O(D_{i-1}^{-1})$ which can be written in the form given in the proposition, since $\prod_{j=1}^{n_1} \mathbb{P}(x_j \in U_i) \prod_{j=1}^{n_2} \mathbb{P}(y_j \in V(M_i)) \sim (1-e^{-k})^{n_1}(e^{-k})^{n_2}$ is bounded away from zero, by \eqref{eqn:pUi} and Proposition ~\ref{claim:matching}. 
\end{proof}	

\section{Embedding a hypergraph into an almost regular hypergraph}
\label{appendix:embed}

	\begin{proof}[Proof of Lemma \ref{lem:embed}]
	Choose $N_0 , D_0$ and $K$ sufficiently large compared to $k$. We may also assume that $D \in \mathbb{N}$.
	
    It follows from well-known results~\cite{wilson1972_1, wilson1972_2, wilson1975} that there exists $N_1 = N_1(k)$ such that a Steiner system $S(2,k,M)$ exists as long as $M \geq N_1$ and both $M-1$ and $M(M-1)$ are divisible by $k-1$ and $k(k-1)$, respectively. For any $M \geq N_1$, there exists an integer $0 \leq t < k(k-1)$ such that $M+t$ satisfies these divisibility conditions. Let $S(M)$ be the $M$-vertex hypergraph obtained by removing any $t$ vertices from $S(2,k,M+t)$. Hence $S(M)$ is simple and each vertex of $S(M)$ has degree between $\frac{M-1}{k-1} - k(k-2)$ and $\frac{M-1}{k-1} + k$.\COMMENT{Note that $S(2,k,M+t)$ has degree exactly $\frac{M-1+t}{k-1}$, hence the degree of $S(M)$ would be between $\frac{M-1}{k-1} - k(k-2) \leq \frac{M-1+t}{k-1} - t$ and $\frac{M-1+t}{k-1} \leq \frac{M-1}{k-1} + k$}

	For each integer $N_1 \leq d \leq D$, we construct a simple hypergraph $H_d$ such that $|V(H_d)| = (k-1)^2 D^2$ and every vertex of $H_d$ has degree between $d-(k+1)(k-1)$ and $d$ as follows. Let $\ell := \lceil (k-1)D^2 / (d-k) \rceil$ and $a_1 \geq \dots \geq a_{\ell}$ be a sequence such that each $a_i$ is either $(k-1)(d-k)-1$ or $(k-1)(d-k)$, and $\sum_{i=1}^{\ell} a_i = (k-1)^2 D^2$.\COMMENT{ where such sequence exists since $\lceil \frac{(k-1)^2 D^2}{(k-1)(d-k)} \rceil \geq (k-1)(d-k)$} Now define $H_d$ by taking the vertex-disjoint union of $S(a_1) , \dots , S(a_\ell)$. Then we deduce that each vertex of $H_d$ has degree between $d-(k+1)(k-1) \leq \frac{(k-1)(d-k) - 2}{k-1} - k(k-2)$ and $\frac{(k-1)(d-k)-1}{k-1} + k \leq d$.
	
	Now we define our desired multi-hypergraph $H'$ as follows. Let us first take the union of $T := (k-1)^2 D^2$ vertex-disjoint copies of $H$. 
	
	For each $v \in V(H)$ with $d_H(v) \leq D - N_1$, let $v^1 , \dots , v^T$ be the $T$ clone vertices of $v \in V(H)$ in $H'$, and extend $H'$ by making $H'[v^1 , \dots , v^T]$ induce a copy of $H_{D - d_H(v)}$. Since every vertex of $H_{D - d_H(v)}$ has degree between $D - d_H(v) - (k+1)(k-1)$ and $D - d_H(v)$, we then have
	$D - (k+1)(k-1) \leq d_{H'}(v^i) \leq D$ for $1 \le i \le T$.

	If $v^1 , \dots , v^T$ are the $T$ clone vertices of some  vertex $v \in V(H)$ with $d_H(v) > D - N_1$, then
	$d_{H'}(v^i) = d_H(v) > D - N_1$ for $1 \le i \le T$. 
	
	Therefore, for any vertex $v \in V(H')$, we have
	\begin{equation*}
	    D - \max(N_1 , (k+1)(k-1)) \leq d_{H'}(v) \leq D.
	\end{equation*}
	It is also straightforward to check that the codegree of $H'$ is at most the codegree of $H$, and $|V(H')| = TN = (k-1)^2 D^2 N$. By the construction of $H'$, if $H$ has no multiple edges then $H'$ also has no multiple edges. This completes the proof of the lemma.
	\end{proof}




\end{document}